\documentclass[letterpaper, 11pt,  reqno]{amsart}
\usepackage{graphicx} 
\usepackage{amsmath,amssymb,amscd,amsthm,amsxtra,esint}
\usepackage{color}
\usepackage[margin=1.1in,marginparwidth=1.5cm, marginparsep=0.5cm]{geometry} 
\usepackage{tikz-cd}
\tikzset{smalltext/.style={"\textup{\small #1}" description}}

\usepackage[markup=default]{changes}
\definechangesauthor[name={Reviewer 1}, color=orange]{R1} 
\definechangesauthor[name={Reviewer 2}, color=blue]{R2} 
\definechangesauthor[name={Reviewer 3}, color=green]{R3}

\usepackage[implicit=true]{hyperref}
\usepackage{comment}

\allowdisplaybreaks[4]

\definecolor{gr}{rgb}   {0.,   0.69,   0.23 }
\definecolor{bl}{rgb}   {0.,   0.5,   1. }
\definecolor{mg}{rgb}   {0.85,  0.,    0.85}
\definecolor{yl}{rgb}   {0.8,  0.7,   0.}
\definecolor{or}{rgb}  {0.7,0.2,0.2}

\DeclareMathOperator*{\supp}{supp}

\newcommand{\noi}{\noindent}
\newcommand{\R}{\mathbb{R}}
\newcommand{\T}{\mathbb{T}}
\newcommand{\Z}{\mathbb{Z}}
\newcommand{\N}{\mathbb{N}}

\newcommand{\EE}{\mathcal{E}}

\newcommand{\Pb}{\mathbf{P}}
\newcommand{\Qb}{\mathbf{Q}}

\newcommand{\F}{\mathcal{F}}

\newcommand{\Id}{\textup{Id}}

\newcommand{\al}{\alpha}
\newcommand{\be}{\beta}
\newcommand{\ta}{\theta}

\newcommand{\eps}{\varepsilon}

\newcommand{\Dl}{\Delta}
\newcommand{\dl}{\delta}
\newcommand{\ld}{\lambda}
\newcommand{\nb}{\nabla}

\newcommand{\dt}{\partial_t}
\newcommand{\dx}{\partial_x}
\renewcommand{\dh}{\partial_h}
\newcommand{\ind}{\mathbf 1}
\newcommand{\ft}{\widehat}
\newcommand{\wt}{\widetilde}
\newcommand{\les}{\lesssim}
\newcommand{\ges}{\gtrsim}
\newcommand{\cj}{\overline}

\newcommand{\jb}[1]{\langle #1 \rangle}

\newtheorem{theorem}{Theorem}[section]
\newtheorem{lemma}[theorem]{Lemma}
\newtheorem{proposition}[theorem]{Proposition}
\newtheorem{corollary}[theorem]{Corollary}
\newtheorem{remark}[theorem]{Remark}

\newtheorem*{ackno}{Acknowledgements}

\numberwithin{equation}{section}
\numberwithin{theorem}{section}

\makeatletter
\@namedef{subjclassname@2020}{%
  \textup{2020} Mathematics Subject Classification}
\makeatother

\title[FPU and KdV: a low-regularity continuum limit]{The Fermi-Pasta-Ulam system and the Korteweg-de Vries equation: a low-regularity continuum limit}

\author[H.~Koch and R.~Liu]{Herbert Koch and Ruoyuan Liu}

\address{
Herbert Koch, Mathematical Institute\\
University of Bonn\\
Endenicher Allee 60\\
53115\\
Bonn\\
Germany}

\email{koch@math.uni-bonn.de}

\address{
Ruoyuan Liu, Mathematical Institute\\
University of Bonn\\
Endenicher Allee 60\\
53115\\
Bonn\\
Germany}

\email{ruoyuanl@math.uni-bonn.de}


\begin{document}

\baselineskip = 14pt


\begin{abstract} 
For the infinite Fermi-Pasta-Ulam (FPU) system, general solutions can be approximated by counter-propagating waves associated with solutions to the Korteweg-de Vries (KdV) equation as the lattice mesh size goes to zero. 
We show that, by exploiting the conservation of the FPU Hamiltonian, the continuum limit from the FPU system to the KdV equation with $L^2$-level initial data holds in appropriate norms on an arbitrary time interval, thereby answering an open question posed by Hong, Kwak, and Yang (2021).
Moreover, for the local-in-time continuum limit of the FPU system to the KdV equation, we lower the required Sobolev regularity to $H^s$ with $s > - \frac 34$. Our continuum limit results also apply to the case of the Toda lattice in Flaschka's form, thereby lowering the regularity requirement in our previous work (2026).
To establish this low-regularity continuum limit, we prove key trilinear estimates that are sharp up to the endpoint by combining linear estimates, the transversality of characteristic curves, and multilinear dispersive smoothing properties of the linear FPU flow.
\end{abstract}


\maketitle

\tableofcontents

\section{Introduction}

\subsection{Fermi-Pasta-Ulam system}

The Fermi-Pasta-Ulam (FPU) system is one of the classical models in nonlinear lattice dynamics and describes an infinite chain of particles with nearest-neighbor interactions.
In terms of the variables $q = q(n)$ and $p = p(n)$, representing the position and the momentum of the $n$-th particle, the dynamics are generated by the FPU Hamiltonian
\begin{align}
H(q, p) = \sum_{n \in \Z} \Big( \frac{1}{2} p (n)^2 + V (q (n + 1) - q (n)) \Big) , 
\label{defH}
\end{align}

\noi
where $V : \R \to \R$ is the interaction potential, and are governed by
\begin{align}
\begin{split}
\begin{cases}
\dt q (t, n) = p (t, n) \\
\dt p (t, n) = V' (q (t, n + 1) - q (t, n)) - V' (q (t, n) - q (t, n - 1)) ,
\end{cases}
\quad (t, n) \in \R \times \Z .
\end{split}
\label{FPU}
\end{align}

\noi
The original numerical experiments of Fermi, Pasta, Ulam, and Tsingou \cite{FPUT} revealed the unexpected recurrence phenomenon (i.e.~the FPU paradox) and initiated the modern study of the interplay among nonlinearity, dispersion, and near-integrability in Hamiltonian lattice systems.

A particularly remarkable perspective, dating back to the discovery of solitons by Zabusky and Kruskal \cite{ZK}, is that the long-wave dynamics of the FPU system are closely related to the Korteweg-de Vries (KdV) equation. Since then, the rigorous derivation of the KdV equation from the FPU system has been studied in a variety of settings. 
At the level of solitary waves, this connection has been investigated in \cite{Eil, FW, FP99, FP02, FP04, FP04-2, Miz}. 
For more general dynamics, Schneider and Wayne \cite{SW} used a dynamical systems approach to show that solutions to the FPU system can be approximated by counter-propagating KdV waves. This was later improved in a recent work by Hong, Kwak, and Yang \cite{HKY}, who proved the continuum limit of FPU to KdV in a Sobolev framework using techniques from dispersive PDEs. We would like to remark that the continuum limit achieved in \cite{HKY} is valid only on a short time interval and requires a high Sobolev regularity assumption $s > \frac 34$ on the initial data.

In the special case of an exponential potential $V(r) = e^{r} - r - 1$ (after normalization), the FPU system reduces to the Toda lattice, which is a completely integrable lattice system. This integrable structure has attracted considerable attention in the study of the Toda lattice; see \cite{Gie1, Tes, Gie2, BKP1, BKP2, BKP3}. 
In particular, in our previous work \cite{KL}, we studied the Toda lattice using Flaschka's variables and proved its  continuum limit to the KdV equation with initial data in the Sobolev space $H^1$. Using complete integrability, we constructed an $H^1$-level conserved energy of the Toda lattice in Flaschka's form, which enabled us to establish the continuum limit over an arbitrary time interval.

Nevertheless, for a general potential $V$, due to the lack of complete integrability, one may not be able to construct such an $H^1$-level energy, and the only known conserved quantity is the FPU Hamiltonian \eqref{defH}. The main difficulty is that the FPU Hamiltonian \eqref{defH} only controls the $L^2$-norm of the solution, whereas the continuum limit established in \cite{HKY} requires $H^s$-level initial data with $s > \frac 34$. Due to the gap $0 \leq s \leq \frac 34$, the global-in-time continuum limit of the FPU system to the KdV equation was posed as an open question in \cite{HKY}.

In this paper, we develop a low-regularity dispersive framework for continuum limits of Hamiltonian lattice systems. In particular, we provide a positive answer to the aforementioned open question in \cite{HKY}. More precisely, for $L^2$-initial data, we establish the continuum limit of the FPU system to the KdV equation with a general potential on an arbitrary time interval. In fact, if we only consider the local-in-time limit, we are able to lower the regularity requirement in \cite{HKY} from $s > \frac 34$ to $s > - \frac 34$. Using the conservation of the Hamiltonian \eqref{defH}, we can then upgrade the local-in-time continuum limit to a global-in-time result at $s = 0$. Our results also apply to the case of the Toda lattice in Flaschka's form, thereby improving the regularity requirement in our previous work \cite{KL}.
Our proof of the low regularity continuum limit is based on $U^p$- and $V^p$-type spaces, which are introduced and developed in the field of dispersive PDEs by Koch, Tataru, Hadac, and Herr \cite{KT05, KT07, HHK}. Within this framework, we establish sharp trilinear estimates (up to the endpoint) by combining linear estimates, transversality of characteristic curves, and multilinear dispersive smoothing properties for both the linear FPU flow and the Airy (linear KdV) flow. In addition, we show several key estimates to handle frequency shifts and interactions of oppositely moving waves.

\subsection{Continuum limit for the FPU system}

Let us now be more precise on the setup for the continuum limit problem for the FPU system \eqref{FPU}.
We first write the FPU system in the following form:
\begin{align}
\dt^2 q (t, n) = V' (q (t, n + 1) - q (t, n)) - V' (q (t, n) - q (t, n - 1)) , \quad (t, n) \in \R \times \Z .
\label{FPU2}
\end{align}

\noi
As in \cite{HKY}, we make the following assumptions on the potential energy:
\begin{align*}
V \in C^5 (\R), \quad V (0) = V' (0) = 0, \quad V'' (0) > 0, \quad V''' (0) \neq 0 .
\end{align*}

\noi
With $a = V'' (0) > 0$ and $b = V''' (0) \neq 0$, by normalizing the equation using the scaling $q (t, n) \mapsto \frac{b}{a} q (\frac{t}{\sqrt{a}}, n)$, we may assume for simplicity that
\begin{align}
V \in C^5 (\R), \quad V (0) = V' (0) = 0, \quad V'' (0) = 1, \quad V''' (0) = 1 .
\label{Vcond}
\end{align}

\noi
By defining
\begin{align*}
r (t, n) := q (t, n + 1) - q (t, n)
\end{align*}

\noi
as the relative displacement between two adjacent points, we write the FPU system \eqref{FPU2} in the following equivalent form:
\begin{align}
\dt^2 r (t, n) = \Delta_1 (V' (r)) (t, n) , \quad (t, n) \in \R \times \Z ,
\label{req1}
\end{align}

\noi
where $\Dl_1$ is the discrete Laplacian on $\Z$ defined by 
\begin{align*}
\Dl_1 f := f (\cdot + 1) + f (\cdot - 1) - 2 f .
\end{align*}

\noi
By extracting the linear term from the nonlinearity on the right-hand side of \eqref{req1}, we get the following discrete nonlinear wave equation:
\begin{align}
\dt^2 r - \Delta_1 r = \Delta_1 \big( V' (r) - r \big) ,
\label{req2}
\end{align}

\noi
The FPU Hamiltonian \eqref{defH} can then be written as
\begin{align}
H (r) = \sum_{n \in \Z} \Big( \frac 12 \big( |\nb_1|^{-1} \dt r (n) \big)^2 + V (r (n)) \Big) ,
\label{Hr}
\end{align}

\noi
where $|\nb_1| = \sqrt{- \Dl_1}$ and the identity
\begin{align*}
\sum_{n \in \Z} (\dt q (n))^2 = \sum_{n \in \Z} \big( |\nb_1|^{-1} \dt r (n) \big)^2
\end{align*}

\noi
follows directly by working on the Fourier side; see Subsection~\ref{SUB:note} below. 
Let us make the following claim on the existence of a unique global-in-time solution to the equation \eqref{req2}.
\begin{proposition}[Global well-posedness for the FPU system]
\label{PROP:FPUGWP}
Let $r_0 \in \ell^2 (\Z)$ and $|\nb_1|^{-1} r_1 \in \ell^2 (\Z)$ satisfying
\begin{align}
\| r_0 \|_{\ell^2 (\Z)} + \| |\nb_1|^{-1} r_1 \|_{\ell^2 (\Z)} \leq \frac{1}{20 \max (1, \sup_{|R| \leq 1} |V''' (R)|)} .
\label{hcond0}
\end{align}

\noi
Then, there exists a unique solution $r \in C (\R; \ell^2 (\Z))$ to the FPU system \eqref{req2} with initial data $(r, \dt r) |_{t = 0} = (r_0, r_1)$. Moreover, for any $t \in \R$, we have the bound
\begin{align*}
\| r (t) \|_{\ell^2 (\Z)} + \| |\nb_1|^{-1} \dt r (t) \|_{\ell^2 (\Z)} \leq 2 \| r_0 \|_{\ell^2 (\Z)} + 2 \| |\nb_1|^{-1} r_1 \|_{\ell^2 (\Z)} .
\end{align*}
\end{proposition} 

The uniqueness statement in Proposition~\ref{PROP:FPUGWP} means that for any $T > 0$, there exists a unique solution $r \in C ([-T, T]; \ell^2 (\Z))$ to the FPU system \eqref{req2}. To prove global well-posedness for \eqref{req2}, we show the conservation of the FPU Hamiltonian \eqref{Hr}, which provides a global-in-time a priori bound for $\| r (t) \|_{\ell^2 (\Z)} + \| |\nb_1|^{-1} \dt r \|_{\ell^2 (\Z)}$. The smallness condition \eqref{hcond0} is required to ensure that the high order nonlinear terms are under control. We present the proof of Proposition~\ref{PROP:FPUGWP} in Section~\ref{SEC:GWP}.

\medskip
In order to study the continuum limit of the FPU system \eqref{req2}, we perform the following scaling compatible with the KdV equation with a small parameter $0 < h \leq 1$:
\begin{align}
r^h (t, \ld) := \frac{1}{h^2} r \Big( \frac{t}{h^3}, \frac{\ld}{h} \Big), \quad (t, \ld) \in \R \times h \Z .
\label{rhscale}
\end{align}

\noi
Then, the equation \eqref{req2} is equivalent to
\begin{align}
\dt^2 r^h - \frac{1}{h^{4}} \Dl_h r^h = \frac{1}{h^{6}} \Dl_h \big( V' (h^2 r^h) - h^2 r^h \big) ,
\label{rhwave0}
\end{align}

\noi
where $\Dl_h$ is the discrete Laplacian on $h \Z$ defined by
\begin{align}
\Dl_h f^h := \frac{1}{h^2} \big( f^h (\cdot + h) + f^h (\cdot - h) - 2 f^h \big) .
\label{laph}
\end{align}

\noi
Formally, as $h \to 0$, the solution $r^h$ to the scaled FPU system \eqref{rhwave0}  decomposes into right- and left-moving components:
\begin{align*}
r^h (t, \ld) \approx u_+ \Big(t, \ld - \frac{t}{h^2}\Big) + u_- \Big(t, \ld + \frac{t}{h^2}\Big) ,
\end{align*}

\noi
where each $u_\pm$ is the solution to the KdV equation
\begin{align}
\dt u_\pm \pm \frac{1}{24} \dx^3 u_\pm \mp \frac 14 \dx (u_\pm^2) = 0.
\label{KdV}
\end{align}

\noi
We outline this mechanism in Section~\ref{SEC:FPU} in more details.

For the KdV equation \eqref{KdV}, thanks to the conservation of mass, we have for any $t \in \R$ that
\begin{align}
\| u_\pm (t) \|_{L^2 (\R)} = \| u_\pm (0) \|_{L^2 (\R)} .
\label{mass}
\end{align}

\noi
This, together with a local solution theory in $L^2 (\R)$ (see \cite{Bour93, KPV96}), gives a global solution of \eqref{KdV} in $C (\R; L^2 (\R))$.
In particular, the solution to the KdV equation \eqref{KdV} is unique in the space $C([-T, T]; L^2 (\R))$ for any $T > 0$; see \cite{Zhou}. See also \cite{BS, KPV91, CCT, CKSTT03, Guo09, Kish09, KV19} for more research on well-posedness problems of the KdV equation \eqref{KdV}.

To state the main theorem, we need to introduce more notations. As in \cite{KL}, given $0 < h \leq 1$ and a discrete function $f^h \in L^2 (h \Z)$, we extend it to the continuum $\R$ via the following extension operator:
\begin{align}
\mathcal{E} f^h (x) := \frac{1}{2 \pi} \int_{- \frac{\pi}{h}}^{\frac{\pi}{h}} \F^h f^h (\xi) e^{i \xi x} d \xi ,
\label{extend}
\end{align}

\noi
where $\F^h$ is the discrete Fourier transform defined in \eqref{dFour} below. 
We then let $\nb_h$ be a Fourier multiplier operator that acts on smooth functions via
\begin{align}
\nb_h f (x) = \frac{1}{h} \Big( f \Big( x + \frac{h}{2} \Big) - f \Big( x - \frac{h}{2} \Big) \Big) ,
\label{nbh}
\end{align}

\noi
which can be viewed alternatively as an operator from $L^2 (h \Z)$ to $L^2 (h (\Z + \frac 12))$.
Also, given any $t \in \R$, we use the notation
\begin{align*}
e^{t \dx} f (x) := f (x + t) ,
\end{align*}

\noi
which is useful when $f$ is merely a distribution.
See Subsection~\ref{SUB:note} for more details on the notations.
Moreover, we need to use Sobolev spaces on the lattice defined in Subsection~\ref{SUB:sob}.

Also, from the solution bound in Proposition~\ref{PROP:FPUGWP}, we know that $r(n)$ never exceeds 1 for any $n \in \Z$. Thus, we may assume for simplicity that $|V^{(4)}| + |V^{(5)}|$ attains its maximum on $[-1, 1]$ and define the constant
\begin{align}
M_{V} := 1 + \sup_{|R| \leq 1} (|V^{(4)} (R)| + |V^{(5)} (R)|) .
\label{defMV}
\end{align}

We now give the statements of the main theorem to be proved in this paper.

\begin{theorem}[Continuum limit for the FPU system]
\label{THM:FPU}
Let $0 < h \leq 1$. Let $r_0^h \in L^2 (h \Z)$ and $|\nb_h|^{-1} r_1^h \in L^2 (h \Z)$. Suppose that we have the condition
\begin{align}
h^{\frac 32} \big( \| r_0^h \|_{L^2 (h \Z)} + \| h^2 |\nb_h|^{-1} r_1^h \|_{L^2 (h \Z)} \big) \leq \frac{1}{20 \max (1, \sup_{|R| \leq 1} |V''' (R)|)} .
\label{hcond}
\end{align}

\noi
Let $r^h \in C (\R; L^2 (h \Z))$ be the unique solution to the scaled FPU system \eqref{rhwave0} with initial data $(r^h, \dt r^h) |_{t = 0} = (r_0^h, r_1^h)$.

\smallskip \noi
\textup{(i) (Local-in-time continuum limit)} 
Let $- \frac 34 < s < 0$ and $0 < \dl \leq 1$ satisfying $\dl < \frac 32 + 2s$. Let $u_{\pm, 0} \in H^s (\R)$ and $r_{\pm, 0}^h$ be the following interpolated initial data:
\begin{align}
r^h_{\pm, 0} := \frac 12 \big( \EE r_0^h \mp h^2 \nb_h^{-1} \EE r_1^h \big) .
\label{rhinit}
\end{align}

\noi 
Suppose that there exist constants $C_1, C_2 > 0$ independent of $h$ such that
\begin{align}
\| r_0^h \|_{H^s (h \Z)} + \| h^2 |\nb_h|^{-1} r_1^h \|_{H^{s} (h \Z)} + \| u_{+, 0} \|_{H^s (\R)} + \| u_{-, 0} \|_{H^s (\R)} \leq C_1
\label{r0cond1-1}
\end{align}

\noi
and
\begin{align}
\| r_{+, 0}^h - u_{+, 0} \|_{H^{s - \dl} (\R)} + \| r_{-, 0}^h - u_{-, 0} \|_{H^{s - \dl} (\R)} \leq C_2 h^\dl.
\label{r0cond1-2}
\end{align}

\noi
Then, there exists $T = T (C_1, C_2, M_{V}) > 0$ such that there exists a unique solution $u_\pm$ to the KdV equation \eqref{KdV} with initial data $u_\pm |_{t = 0} = u_{\pm, 0}$ in a subspace of $C([-T, T]; H^s (\R))$, and we have
\begin{align}
\begin{split}
&\big\| \EE r^h - e^{- \frac{t}{h^2} \dx} u_+ - e^{\frac{t}{h^2} \dx} u_- \big\|_{C ([-T, T]; H^{s - \dl} (\R))} \leq  \min ( 200 C_1, \wt C h^{\frac 25 \dl} ) , \\
&\big\| h^2 \nb_h^{-1} \EE \dt r^h - e^{\frac{t}{h^2} \dx} u_- + e^{- \frac{t}{h^2} \dx} u_+ \big\|_{C ([-T, T]; H^{s - \dl} (\R))} \leq \min ( 200 C_1, \wt C h^{\frac 25 \dl} )
\end{split}
\label{conv_goal1}
\end{align}

\noi
for some constant $\wt C = \wt C (C_1, C_2, M_{V}) > 0$.

\smallskip \noi
\textup{(ii) (Global-in-time continuum limit)}
Let $s = 0$. Let $u_{\pm, 0} \in L^2 (\R)$ and $r_{\pm, 0}^h$ be defined in \eqref{rhinit}. Suppose that there exist constants $C_1, C_2 > 0$ independent of $h$ such that
\begin{align}
\| r_0^h \|_{L^2 (h \Z)} + \| h^2 |\nb_h|^{-1} r_1^h \|_{L^2 (h \Z)} + \| u_{+, 0} \|_{L^2 (\R)} + \| u_{-, 0} \|_{L^2 (\R)} \leq C_1
\label{r0cond2-1}
\end{align}

\noi
and
\begin{align}
\| r_{+, 0}^h - u_{+, 0} \|_{H^{-1} (\R)} + \| r_{-, 0}^h - u_{-, 0} \|_{H^{-1} (\R)} \leq C_2 h.
\label{r0cond2-2}
\end{align}

\noi 
Let $u_\pm \in C (\R; L^2 (\R))$ be the unique solution to the KdV equation \eqref{KdV} with initial data $u_\pm |_{t = 0} = u_{\pm, 0}$. 
Then, for any $T > 0$, we have
\begin{align}
\begin{split}
&\big\| \EE r^h - e^{- \frac{t}{h^2} \dx} u_+ - e^{\frac{t}{h^2} \dx} u_- \big\|_{C ([-T, T]; H^{-1} (\R))} \leq \min ( 2 C_1, e^{C T} \wt C h^{\frac 25} ) , \\
&\big\| h^2 \nb_h^{-1} \EE \dt r^h -  e^{\frac{t}{h^2} \dx} u_- + e^{- \frac{t}{h^2} \dx} u_+ \big\|_{C ([-T, T]; H^{-1} (\R))} \leq \min ( 2 C_1, e^{C T} \wt C h^{\frac 25} )
\end{split}
\label{conv_goal2}
\end{align}

\noi
for some constants $C = C(C_1, M_{V}) > 0$ and $\wt C = \wt C (C_1, C_2, M_{V}) > 0$.
\end{theorem}

The condition \eqref{hcond} in the statement of Theorem~\ref{THM:FPU} is nothing but a scaled version of \eqref{hcond0} in Proposition~\ref{PROP:FPUGWP}, which guarantees the existence of a unique solution $r^h$ to the scaled FPU system \eqref{rhwave0} in $C ([-T, T]; L^2 (h \Z))$ for any $T > 0$.
Thanks to the equivalence of the $H^s (h \Z)$-norms, we know that the existence and the uniqueness hold also in $C([0, T]; H^s (h \Z))$ for any $s \in \R$.
In view of the boundedness condition \eqref{r0cond1-1} (and Lemma~\ref{LEM:embhs}) or \eqref{r0cond2-1} on the size of $r_0^h$ and $r_1^h$, we see that the condition \eqref{hcond} on $h$ can always be satisfied in either situation as long as it is chosen sufficiently small depending on the constant $C_1$ and also $\sup_{|R| \leq 1} |V''' (R)|$.

Our Theorem~\ref{THM:FPU} shows that, given suitable $T > 0$, the difference of the extension of $r^h$ to the real line and the counter-propagating KdV flows $e^{- \frac{t}{h^2} \dx} u_+ + e^{\frac{t}{h^2} \dx} u_-$ converges to zero in $C([-T, T]; H^{s - \dl} (\R))$ as $h$ goes to zero, where $s$ and $\dl$ are as given in the statement of the theorem. To provide the heuristic on why the convergence is expected, we look at the symbol of the linear FPU flow $e^{\pm \frac{2i t}{h^3} \sin (\frac{h \xi}{2})}$. In view of the expansion $\frac{2}{h^3} \sin (\frac{h \xi}{2}) = \frac{\xi}{h^2} - \frac{\xi^3}{24} + O (h^2 \xi^5)$, we see that after applying the linear correction $e^{\pm \frac{t}{h^2} \dx}$, the system exhibits cubic dispersion that corresponds to the Airy flow. The condition $V''' (0) \neq 0$ in \eqref{Vcond} ensures that the nonlinearity does not vanish, and so due to the KdV scaling \eqref{rhscale}, the quadratic nonlinearity survives as $h$ goes to zero. We also note that convergence rates $h^{\frac 25 \dl}$ and $h^{\frac 25}$ obtained in Theorem~\ref{THM:FPU} are sharp in view of the remainder $O (h^2 \xi^5)$, which provides a trade-off ratio of $2:5$ between the convergence rate in $h$ and the derivative loss.

In formulating the initial conditions \eqref{r0cond1-1}, \eqref{r0cond1-2}, \eqref{r0cond2-1}, and \eqref{r0cond2-2}, we follow our previous work \cite{KL}. These initial conditions work particularly for the following two scenarios: given $0 < h \leq 1$ and $r_{\pm, 0}^h$ defined in \eqref{rhinit} in terms of $r_0^h$ and $r_1^h$, one may take $u_{\pm, 0} = r_{\pm, 0}^h$ as the initial data for the KdV equation; or given $u_{\pm, 0}$, one may choose 
\begin{align*}
r_{\pm, 0}^h = P_{\leq \frac{\pi}{h}} u_{\pm, 0} 
\end{align*}

\noi
with $P_{\leq \frac{\pi}{h}}$ being the frequency cutoff onto $\{ \xi \in \R : |\xi| \leq \frac{\pi}{h} \}$.

In the statement of the theorem, the extension operator $\EE$ is a Fourier analytic substitute of the linear interpolation procedure used on some other discrete models (see \cite{Lady, SSB, KLS, HY19-1, HY19-2, HKY}) and is inspired by the work \cite{KOVW23, KOVW} on the continuum limits of the Ablowitz-Ladik system.
It allows us to reduce the continuum limit problem of a discrete model to a convergence problem of equations on the real line. Such reduction is convenient for us to directly invoke some Fourier analytic tools on the real line to treat the discrete model. However, the drawback is that this extension operator is not entirely compatible with products of functions. Indeed, the extension of a product of discrete functions is the product of the extended functions together with an additional frequency-shifted version of the product; see Lemma~\ref{LEM:Eprod} below. Therefore, some extra care is required to handle the frequency-shifted product.

Our Theorem~\ref{THM:FPU} establishes the KdV limit of the FPU system for initial data in $H^s$ with $s > - \frac 34$, which significantly improves the previously known result in \cite{HKY} with regularity requirement $s > \frac 34$. In particular, we are able to show global-in-time continuum limit at the $L^2$-level ($s = 0$) using the conservation of the Hamiltonian, thereby resolving an open problem posed in \cite{HKY}. 
Although our result is stated for the real-line extension of the solution $r^h$ to the scaled FPU system \eqref{rhwave0}, one can also formulate the result on the discrete lattice $h \Z$ with only minor additional care. See Remark~\ref{RMK:lw} below for further discussion on this point.

Before explaining the main ideas behind the proof of Theorem~\ref{THM:FPU}, we first briefly review the approach developed in \cite{HKY}.
The authors in \cite{HKY} exploited dispersive smoothing properties of the FPU system using the Fourier restriction norm method based on $X^{s, b}$-spaces (see \cite{Bour93-1, Bour93}). 
In particular, by using multilinear dispersive smoothing effects, they established uniform-in-$h$ bilinear estimates for the linear FPU flow for spatial regularity $s \geq 0$ and yielded a uniform-in-$h$ a priori bound for the solution to the FPU system for such regularity. Nevertheless, the $X^{s, b}$-norm turns out to be sensitive on the choice of the underlying linear propagator, as the linear FPU flow was found to be unbounded in the $X^{s, b}$-norm adapted to the Airy flow. Thus, to prove the convergence of FPU to KdV, they instead used linear estimates (the Strichartz estimate, the local smoothing, and the maximal function estimate as in \cite{KPV91}), which led to the restriction $s > \frac 34$ with the maximal function estimate being the limiting factor.

In this paper, we exploit linear and multilinear dispersive smoothing properties of the FPU system in the framework of $U^p$- and $V^p$-spaces, developed in \cite{HHK} (see also \cite{KT05, KT07, HTTz}); see also Remark~\ref{RMK:UpVp} below for further discussion of the role of $U^p$- and $V^p$-spaces in this paper. 
In particular, we show some useful bilinear estimates by exploiting the transversality of the characteristic curves associated with both the FPU flow and the Airy flow via the coarea formula. This formulation turns out to be particularly effective for the FPU system, not only for establishing uniform-in-$h$ a priori bounds for solutions but also for proving the convergence of the FPU dynamics as $h$ goes to zero. We then establish some key trilinear estimates by performing case-by-case analysis according to the high- and low-frequency interactions, including trilinear terms involving the frequency-shifted factor mentioned above and/or the interaction of oppositely moving waves. An interesting observation is that these trilinear terms have distinct phase functions, which, in the case of the frequency-shifted factor and/or the interaction of oppositely moving waves, reveal clearly how they vanish as the lattice mesh size goes to zero. 
As for the difference estimate of two linear propagators, we encounter the same unboundedness phenomenon as in \cite{HKY}. In fact, in the $U^p$ and $V^p$ framework, we have to lose derivatives but at the same time create a convergence factor. Nevertheless, by combining a careful interpolation argument with uniform trilinear estimates, we are able to overcome this loss of derivatives.

\begin{remark} \rm
\label{RMK:lw}
One may ask whether it is possible to deduce from Theorem~\ref{THM:FPU} a long-wave limit result of the following form:
\begin{align*}
\Big\| r^h (t, \ld) - u_+ \Big( t, \ld - \frac{t}{h^2} \Big) - u_- \Big( t, \ld + \frac{t}{h^2} \Big) \Big\|_{C ([-T, T]; H^{s - \gamma} (h \Z))} \leq e^{C T} \wt C h^{\frac 25 \gamma}
\end{align*}

\noi
where all the parameters are as in the statement of Theorem~\ref{THM:FPU}, $r^h$ is the solution to the scaled FPU system \eqref{rhwave0}, and each $u_\pm$ is the solution to the KdV equation \eqref{KdV} in $C([-T, T]; H^s (\R))$. 
Unfortunately, in the low regularity setting of this paper, it is not the case as the spatial pointwise evaluations of $u_\pm$ are not intrinsically well-defined even in $L^2 (\R)$ (i.e. $s = 0$). This is in contrast with the situation in our previous work \cite{KL}, in which we are able to prove a long-wave limit for solutions in the $H^1 (\R)$-space consisting of bounded continuous functions. 
Nevertheless, one can obtain the following substitute of the long-wave limit:
\begin{align*}
\Big\| r^h (t, \ld) - \rho * u_+ \Big( t, \ld - \frac{t}{h^2} \Big) - \rho * u_- \Big( t, \ld + \frac{t}{h^2} \Big) \Big\|_{C ([-T, T]; H^{s - \gamma} (h \Z))} \leq e^{C T} \wt C h^{\frac 25 \gamma} ,
\end{align*}

\noi
where $\rho$ is a smooth mollifier. 
\end{remark}

\begin{remark} \rm
\label{RMK:34}
In this paper, we prove a local-in-time continuum limit from the FPU system to the KdV equation with initial data in $H^{s} (\R)$ for $- \frac 34 < s < 0$. At the endpoint $s = - \frac 34$, one may also obtain such a continuum limit.
Indeed, we expect that a uniform-in-$h$ local well-posedness result holds for the scaled FPU system \eqref{rhwave0}; see Remark~\ref{RMK:LWP34}. One can then possibly prove the continuum limit using a compactness argument as in \cite{KOVW23, KOVW}. 

It would be of interest to investigate the continuum limit of the FPU system to the KdV equation with regularity $s < - \frac 34$. In view of the well-posedness threshold $s = -1$ for the KdV equation (see \cite{Mol11, KV19}), one might be able to extend the regularity requirement for the continuum limit of FPU to KdV to $s \geq - 1$. In the range $-1 \leq s < - \frac 34$, we expect that tools from completely integrable systems are required as in the case of well-posedness of the KdV equation. Thus, one may need to consider the particular case of the Toda lattice, which is discussed in the next subsection.
\end{remark}

\begin{remark} \rm
Besides the FPU system on the whole lattice, there has been extensive research on the periodic FPU system and how it is related to the KdV equation on the circle \cite{Gie1, Gie2, PB05, BP06, GMWZ, BKP2, BKP3}. Unfortunately, our method in this paper does not apply to the periodic setting directly. As for now, the best known regularity requirement for the continuum limit of periodic FPU to KdV is $s > 0$ achieved by Kwak and Yang in \cite{KY} using the normal form reduction. 
Nevertheless, in view of the analytic well-posedness of the KdV equation with $s \geq - \frac 12$ (see \cite{KPV96, CKSTT03, CKSTT04}), we expect that the continuum limit of periodic FPU to KdV can also be proved in this range or at least for the non-endpoint case $s > - \frac 12$.
\end{remark}

\subsection{Continuum limit for the Toda lattice in Flaschka's form}
\label{SUB:toda}

As mentioned earlier, in the case of an exponential potential $V(r) = e^r - r - 1$, the FPU system \eqref{req1} reduces to the Toda lattice, whose continuum limit at the $H^1$-regularity has been studied in our previous work \cite{KL}. In \cite{KL}, we worked on Flaschka's variables \cite{Fla}:
\begin{align}
\al (t, n) := \exp \Big( \frac{1}{2} r (t, n) \Big) - 1 \quad \text{and} \quad \be (t, n) := \frac 12 (\partial_1^+)^{-1} \dt r (t, n), \quad (t, n) \in \R \times \Z
\label{defab}
\end{align}

\noi
with $\partial_1^+$ being the right discrete derivative $\partial_1^+ f = f (\cdot + 1) - f (\cdot)$, we obtain the following equivalent form of \eqref{req1}:
\begin{align}
\begin{cases}
\dt \al (n) = ( 1 + \al (n) ) (\be (n + 1) - \be (n)) \\
\dt \be (n) = (1 + \frac 12 \al (n) + \frac 12 \al (n - 1))  (\al (n) - \al (n - 1)) ,
\end{cases} \quad (t, n) \in \R \times \Z .
\label{Todaab}
\end{align}

\noi
We then defined the variable
\begin{align*}
\gamma_\pm (t, n) :=
\begin{cases}
\alpha (t, \frac{n}{2}) & \text{if $n$ is even} \\
\mp \beta (t, \frac{n + 1}{2}) & \text{if $n$ is odd}
\end{cases}
\end{align*}

\noi
and also used the same scaling regime
\begin{align}
\gamma_\pm^h (t, \ld) := \frac{1}{h^2} \gamma_\pm \Big( \frac{t}{h^3}, \frac{\ld}{h} \Big), \quad (t, \ld) \in \R \times h \Z .
\label{defgmh}
\end{align}

\noi
Then, we see that the Toda lattice \eqref{Todaab} is equivalent to the following scaled Toda lattice:
\begin{align}
\dt \gamma_\pm^h (\ld) = \mp \frac{1}{h^3} (\gamma_\pm^h (\ld + h) - \gamma_\pm^h (\ld - h)) \mp \mathcal{N}^h (\gamma_\pm^h) (\ld) ,
\label{Todag}
\end{align}

\noi
where the nonlinearity $\mathcal{N}^h (\gamma^h)$ is given by
\begin{align*}
\mathcal{N}^h (\gamma^h) (\ld) =
\begin{cases}
\frac{1}{h} \gamma^h (\ld) (\gamma^h (\ld + h) - \gamma^h (\ld - h)) & \text{if $\frac{\ld}{h}$ is even} \\
\frac{1}{2h} (\gamma^h (\ld + h) + \gamma^h (\ld - h)) (\gamma^h (\ld + h) - \gamma^h (\ld - h)) & \text{if $\frac{\ld}{h}$ is odd} .
\end{cases}
\end{align*}

\noi
In \cite[Theorem~1.1 and Remark~1.3]{KL}, we showed that under the $H^1$-regularity assumption and suitable difference assumption on the initial data, the following object under the extension operator and a linear translation
\begin{align*}
e^{\pm \frac{2t}{h^2} \dx} \EE \gamma_\pm^h (t, x) = \EE \gamma_\pm^h \Big( t, x \pm \frac{2 t}{h^2} \Big)
\end{align*}

\noi
converges to the solution to (a slightly scaled version of) the KdV equation \eqref{KdV}.

In this paper, we lower the regularity requirement in \cite{KL} by showing the following low-regularity continuum limit result for the Toda lattice in Flaschka's form. Compared to Theorem~\ref{THM:FPU} for the general FPU system, the following theorem states the continuum limit directly in the standard Flaschka variables $\al^h$, $\be^h$, and $\gamma_\pm^h$, avoiding a nontrivial translation from the displacement variables $\EE r^h$ and $h^2 \nb_h^{-1} \EE \dt r^h$.

\begin{theorem}[Continuum limit for the Toda lattice: Flaschka's form]
\label{THM:toda}
Let $0 < h \leq 1$, $\alpha_0^h \in L^2 (h \Z)$, and $\beta_0^h \in L^2 (h \Z)$.
Suppose that $h$ satisfies
\begin{align}
h^{\frac 32} \| \alpha_0^h \|_{L^2 (h \Z)} \leq \frac 14 .
\label{hcond2}
\end{align} 

\noi
Define
\begin{align*}
\gamma_{\pm, 0}^h (\ld) := 
\begin{cases}
\alpha_0^h (\frac{\ld}{2}) & \textup{if $\frac{\ld}{h}$ is even} \\
\mp \beta_0^h (\frac{\ld + h}{2}) & \textup{if $\frac{\ld}{h}$ is odd} ,
\end{cases} \quad \ld \in h \Z .
\end{align*}

\smallskip \noi
\textup{(i) (Global well-posedness for the Toda lattice)} There exists a unique solution $\gamma_\pm^h \in C (\R; L^2 (h \Z))$ to the scaled Toda lattice \eqref{Todag} with initial data $\gamma_\pm^h |_{t = 0} = \gamma_{\pm, 0}^h$. Moreover, for any $t \in \R$, we have the bound
\begin{align}
\| \al^h (t) \|_{L^2 (h \Z)} + 2 \| \be^h (t) \|_{L^2 (h \Z)} \leq 4 \| \al_0^h \|_{L^2 (h \Z)} + 2 \| \be_0^h \|_{L^2 (h \Z)} .
\label{abbdd}
\end{align}

\smallskip \noi
\textup{(ii) (Continuum limit)} Let $u_{\pm, 0} \in L^2 (\R)$. Suppose that there exist constants $C_1, C_2 > 0$ independent of $h$ such that
\begin{align}
\| \alpha_0^h \|_{L^2 (h \Z)} + \| \beta_0^h \|_{L^2 (h \Z)} + \| u_{+, 0} \|_{L^2 (\R)} + \| u_{-, 0} \|_{L^2 (\R)} \leq C_1
\label{ab0cond2-1}
\end{align} 

\noi
and
\begin{align}
\Big\| \EE \gamma_{+, 0}^h - \frac 12 u_{+, 0} \Big( \frac{1}{2} \cdot \Big) \Big\|_{H^{-1} (\R)} + \Big\| \EE \gamma_{-, 0}^h - \frac 12 u_{-, 0} \Big( \frac{1}{2} \cdot \Big) \Big\|_{H^{-1} (\R)} \leq C_2 h .
\label{ab0cond2-2}
\end{align}

\noi
Let $u_\pm \in C (\R; L^2 (\R))$ be the unique solution to the KdV equation \eqref{KdV} with initial data $u_\pm |_{t = 0} = u_{\pm, 0}$. Then, for any $T > 0$, we have
\begin{align}
\Big\| \EE \gamma_\pm^h - \frac 12 e^{\mp \frac{2 t}{h^2} \dx} u_\pm \Big( \frac{1}{2} \cdot \Big) \Big\|_{C ([-T, T]; H^{-1} (\R))} \leq \min \big(C (1 + C_1^2), e^{C' T} \wt C h^{\frac 25}\big)
\label{abconv2}
\end{align}

\noi
for some constants $C > 0$, $C' = C' (C_1) > 0$, and $\wt C = \wt C (C_1, C_2) > 0$, where the scaling factor $\frac 12 \cdot$ means the factor on the spatial variable.
\end{theorem}

We can also write out a local-in-time continuum limit result for the Toda lattice in Flaschka's form, where the regularity assumption can be reduced to $s > - \frac 34$. For simplicity of presentation, we only write in details the global-in-time continuum limit with $L^2$-level initial data.

In Theorem~\ref{THM:toda}, the condition \eqref{hcond2} guarantees the existence of a unique solution $\gamma_\pm^h$ to the scaled Toda lattice \eqref{Todag} as stated in Theorem~\ref{THM:toda}~(i). This was shown by \cite[Proposition~3.1, (3.18), and (3.19)]{KL} and a straightforward scaling argument. Note that \cite[Proposition~3.1]{KL} requires $h^2 \| \al_0^h \|_{L^\infty (h \Z)} < 1$, which follows directly from \eqref{hcond2} thanks to the embedding in Lemma~\ref{LEM:embh} below.
The conditions \eqref{ab0cond2-1} and \eqref{ab0cond2-2} in Theorem~\ref{THM:toda} correspond exactly to the conditions  \eqref{r0cond2-1} and \eqref{r0cond2-2}. In particular, the boundedness condition  \eqref{ab0cond2-1} implies the condition \eqref{hcond2} for $h$ sufficiently small depending on the constant $C_1$.

We only need to show Theorem~\ref{THM:toda}~(ii). To prove it, we reduce it to Theorem~\ref{THM:FPU}~(ii) by exploiting the connection between the scaled Toda lattice in Flaschka's form and the scaled FPU system via the relations \eqref{defab}. In particular, in establishing our low-regularity continuum limit results, we do not use the complete integrability of the Toda lattice and the KdV equation. Note that the limiting object $v_\pm (t, x) = \frac 12 u_\pm (t, \frac 12 x)$ indeed satisfies the following slightly scaled version of the KdV equation:
\begin{align*}
\dt v_\pm \pm \frac 13 \dx^3 v_\pm \mp \dx (v_\pm^2) = 0 .
\end{align*}

\begin{remark} \rm
The boundedness condition \eqref{ab0cond2-1} in our Theorem~\ref{THM:toda} is stated for both $\al_0^h$ and $\be_0^h$, whereas the one in \cite[Theorem~1.1]{KL} is stated only for the interpolated data $\gamma_{\pm, 0}^h$. These two conditions are equivalent due to $\| \al_0^h \|_{L^2 (h \Z)} + \| \be_0^h \|_{L^2 (h \Z)} \sim \| \gamma_{\pm, 0}^h \|_{L^2 (h \Z)}$. 
\end{remark}

\subsection{Organization of the paper}

The paper is organized as follows. In Section~\ref{SEC:note}, we discuss notations, function spaces, and preliminary lemmas. 
In Section~\ref{SEC:FPU}, we further reformulate the scaled FPU system \eqref{rhwave0} and define the main objects that we will use to show the convergence of dynamics.
In Section~\ref{SEC:GWP}, we show global well-posedness of the FPU system \eqref{req2} as stated in Proposition~\ref{PROP:FPUGWP}, using the conservation of the Hamiltonian \eqref{Hr}. 
In Section~\ref{SEC:est}, we establish key trilinear estimates, which are used in Section~\ref{SEC:conv} to show a uniform-in-$h$ a priori bound for the solution to the scaled FPU system and also the convergence of dynamics as $h$ goes to zero. Theorem~\ref{THM:FPU} is proved in Subsection~\ref{SUB:conv} and Theorem~\ref{THM:toda} is proved in Subsection~\ref{SUB:toda2}.

\section{Notations, function spaces, and preliminary lemmas}
\label{SEC:note}

\subsection{Notations}
\label{SUB:note}

Let us introduce some notations, most of which are taken from our previous study in \cite{KL}.

Given quantities $A, B > 0$, we use $A \les B$ to mean that $A \leq C B$ for some constant $C > 0$ independent of the ranges where $A$ and $B$ are allowed to vary. We use $A \sim B$ to mean that $A \les B$ and $A \ges B$. We also use $A \ll B$ to mean that $A \leq c B$ for some small $c > 0$.

Thanks to the time reversibility of all the equations in this paper, we work with solutions only on the positive time line $\R_+$. Any uniqueness statement in the space $C(\R_+; B)$ for some Banach space $B$ means uniqueness in $C([0, T]; B)$ for any $T > 0$.
Regarding function space norms on space-time functions or distributions, when there is no confusion, we often use the abbreviation of function spaces such as $L_t^p L_x^q = L_t^p (\R; L_x^q (\R))$ and $L_T^p H_x^s = L_t^p ([0, T]; H_x^s (\R))$. Also, in this paper, we assume that all functions are real-valued.

Given a function $f \in L^1 (\R)$, we define the Fourier transform of $f$ as
\begin{align*}
\ft f (\xi) := \int_\R f (x) e^{- i \xi x} dx , \quad \xi \in \R ,
\end{align*}

\noi
and also its inverse Fourier transform as
\begin{align*}
f^\vee (x) = \frac{1}{2 \pi} \int_\R f (\xi) e^{i \xi x} d \xi .
\end{align*}

\noi
We denote $\dx$ as the usual partial derivative on the spatial variable, which is also viewed as the Fourier multiplier operator with symbol $i \xi$.
Given $R > 0$, we denote by $P_{\leq R}$ the spatial frequency projector onto $[-R, R]$, or in other words $\ft{P_{\leq R} f} = \ind_{[-R, R]} \ft f$. 
We also denote $P_{> R} = \Id - P_{\leq R}$. Given a set $E \subset \R$, we denote by $P_E$ as the spatial frequency projector onto $E$.

Let $\varphi : \R \to [0, 1]$ be a smooth even cutoff function such that $\varphi \equiv 1$ on $[- \frac 54, \frac 54]$ and $\varphi \equiv 0$ outside of $[- \frac 85, \frac 85]$.
Given a dyadic number $N \in 2^{\Z}$, we set 
$$\varphi_N (\xi) = \varphi (\tfrac{|\xi|}{N}) - \varphi (\tfrac{2 |\xi|}{N}),$$

\noi
so that we have 
\begin{align}
\sum_{N \in 2^\Z} \varphi_N (\xi) = 1, \quad \xi \in \R \setminus \{0\}.
\label{PNdecomp}
\end{align}

\noi
We define the Littlewood-Paley projector $\Pb_N$ as the Fourier multiplier operator with symbol $\varphi_N$. 
In this paper, we will implicitly use the identity $\Pb_N = \Pb_N (\Pb_{\frac{N}{2}} + \Pb_N + \Pb_{2N})$ in many occasions.
To deal with functions or distributions with both space and time variables, we use $\Pb_N$ and $\Qb_N$ to denote the spatial and temporal Fourier multiplier operators with symbol $\varphi_N$, respectively. For example, given $M, N \in 2^\Z$ and a space-time function $u$, $\Qb_M \Pb_N u$ denotes $u$ with spatial frequency localized on $\{ \xi \in \R : |\xi| \sim N \}$ and temporal frequency localized on $\{ \tau \in \R : |\tau| \sim M \}$.
Given $M, N \in 2^\Z$, we also denote $\Qb_{\leq M} = \sum_{M' \in 2^{\Z}, M' \leq M} \Qb_{M'}$, $\Qb_{> M} = \Id - \Qb_{\leq M}$, $\Pb_{\leq N} = \sum_{N' \in 2^{\Z}, N' \leq N} \Pb_{N'}$, and $\Pb_{\les N} = \sum_{N' \in 2^{\Z}, N' \les N} \Pb_{N'}$.

We have the following boundedness of the temporal Fourier multiplier operator.
\begin{lemma}
\label{LEM:QLpt}
Let $1 \leq p \leq \infty$ and $B$ be a Banach space only for the spatial variable. Then, for any $M \in 2^\Z$, we have
\begin{align*}
&\| \Qb_{\leq M} u \|_{L^p_t B} \les \| u \|_{L^p_t B} , \\
&\| \Qb_{> M} u \|_{L^p_t B} \les \| u \|_{L^p_t B} ,
\end{align*}

\noi
where the underlying constants are independent of $M$.
\end{lemma}

\begin{proof}
The second bound follows from the first bound since $\Qb_{> M} = \Id - \Qb_{\leq M}$. From the definition, we see that the symbol $\psi_M$ for $\Qb_{\leq M}$ satisfies $\psi_M (\cdot) = \psi_1 (M^{-1} \cdot)$. Thus, by Minkowski's integral inequality and Young's convolution inequality, we obtain
\begin{align*}
\| \Qb_{\leq M} u \|_{L^p_t B} &\leq M \big\| |\psi_1^\vee (M \cdot)| * \| u \|_B \big\|_{L^p_t} \leq M \| \psi_1^\vee (M \cdot) \|_{L^1} \| u \|_{L_t^p B} \les \| u \|_{L_t^p B} ,
\end{align*}

\noi
as desired.
\end{proof}

We now look at functions defined on the (scaled) lattice. Given $0 < h \leq 1$ and a function $f^h : h \Z \to \mathbb{C}$, we define the Lebesgue spaces $L^p (h \Z)$ via the norm
\begin{align*}
\| f^h \|_{L^p (h \Z)} := 
\begin{cases}
\Big( h \displaystyle\sum_{\ld \in h \Z} |f^h (\ld)|^p \Big)^{\frac 1p} & \text{if } 1 \leq p < \infty \\
\displaystyle\sup_{\ld \in h \Z} |f^h (\ld)| & \text{if } p = \infty .
\end{cases}
\end{align*}

\noi
Note that we have the following embedding property of $L^p (h \Z)$-spaces; see \cite[Lemma~2.1]{KL}. 
\begin{lemma}
\label{LEM:embh}
For any $0 < h \leq 1$ and $1 \leq p \leq q \leq \infty$, we have
\begin{align*}
\| f^h \|_{L^q (h \Z)} \leq h^{\frac{1}{q} - \frac{1}{p}} \| f^h \|_{L^p (h \Z)} .
\end{align*}
\end{lemma}

For $f^h \in L^1 (h \Z)$, we define the discrete Fourier transform of $f^h$ by
\begin{align}
\F^h f^h (\xi) := h \sum_{\ld \in h \Z} f^h (\ld) e^{- i \xi \ld}, \quad \xi \in \R / (\tfrac{2 \pi}{h} \Z) \cong [- \tfrac{\pi}{h}, \tfrac{\pi}{h}) .
\label{dFour}
\end{align}

\noi
Note that $\F^h f^h$ is a $\frac{2 \pi}{h}$-periodic function. We also have the inverse Fourier transform
\begin{align}
f^h (\ld) = \frac{1}{2 \pi} \int_{- \frac{\pi}{h}}^{\frac{\pi}{h}} \F^h f^h (\xi) e^{i \xi \ld} d \xi .
\label{dFour_inv}
\end{align}

\noi
Moreover, we have the following Plancherel's (or Parseval's) identity:
\begin{align}
h \sum_{\ld \in h \Z} f^h (\ld) \cj{g^h (\ld)} = \frac{1}{2 \pi} \int_{- \frac{\pi}{h}}^{\frac{\pi}{h}} \F^h f^h (\xi) \cj{\F^h g^h (\xi)} d \xi,
\label{Plan}
\end{align}

\noi
which also extends the discrete Fourier transform $\F^h$ to functions in $L^2 (h \Z)$. In particular, the inversion formula \eqref{dFour_inv} holds for functions in $L^2 (h \Z)$.

We define $\dh^+$ and $\dh^-$ as the right and the left discrete derivatives, given by
\begin{align}
\dh^+ f^h (\ld) := \frac{1}{h} \big( f^h (\ld + h) - f^h (\ld) \big) 
\quad \text{and} \quad
\dh^- f^h (\ld) := \frac{1}{h} \big( f^h (\ld) - f^h (\ld - h) \big) ,
\label{dhpm}
\end{align}

\noi
respectively. 
Note that $\dh^+$ and $\dh^-$ correspond to the discrete Fourier multiplier operator with symbols $\frac{e^{i h \xi} - 1}{h} = \frac{2i}{h} e^{\frac{i h \xi}{2}} \sin (\frac{h \xi}{2})$ and $\frac{1 - e^{- i h \xi}}{h} = \frac{2i}{h} e^{\frac{- i h \xi}{2}} \sin (\frac{h \xi}{2})$, respectively.
We also note that the discrete Laplacian defined in \eqref{laph} is given by $\Dl_h = h^{-1} (\dh^+ - \dh^-) = \dh^+ \dh^-$ and corresponds to the discrete Fourier multiplier operator with symbol $- \frac{4}{h^2} \sin (\frac{h \xi}{2})^2$.
For later convenience, we denote $|\nb_h| = \sqrt{- \Dl_h}$, which is the discrete Fourier multiplier operator with symbol $\frac{2}{h} |\sin (\frac{h \xi}{2})|$. 
Also, $\Dl_h$ and $|\nb_h|$ can also be seen as Fourier multiplier operators on functions defined on the whole real line $\R$ with the same symbols.
Note that $\frac{2}{h} |\sin (\frac{h \xi}{2})| \sim |\xi|$ when $\xi \in [-\frac{\pi}{h}, \frac{\pi}{h})$, so that $|\nb_h| P_{\leq \frac{\pi}{h}}$ is $L^2$-equivalent to $|\dx|$ (see Lemma~\ref{LEM:Hs_equi}). 

Let us also define $\nb_h$ as the Fourier multiplier operator on functions defined on the real line $\R$ with symbol $\frac{2i}{h} \sin (\frac{h \xi}{2})$. We note that $\frac{2i}{h} \sin (\frac{h \xi}{2})$ is not $\frac{2 \pi}{h}$-periodic, and so we do not view it as a discrete Fourier multiplier operator. In fact, it is not hard to see that $\nb_h$ is given by the difference quotient in \eqref{nbh}. We also note that $\nb_h^2 = \Dl_h$.

Let us recall the extension operator $\EE$ in \eqref{extend}, which extends a discrete function in $L^2 (h \Z)$ to a function defined on the continuum $\R$.
We mention the following identity on the extension of a product of two discrete functions. For a proof, see \cite[Lemma~2.4]{KL}.
\begin{lemma}
\label{LEM:Eprod}
Let $0 < h \leq 1$. Let $f^h, g^h \in L^2 (h \Z)$. Then, we have
\begin{align*}
\mathcal{E} (f^h g^h) = P_{\leq \frac{\pi}{h}} \big( (1 + 2 \cos (\tfrac{2 \pi}{h} \cdot)) \EE f^h \EE g^h \big) .
\end{align*}
\end{lemma}

\subsection{Sobolev spaces}
\label{SUB:sob}

In this paper, we mainly work on $L^2$-based Sobolev spaces, both on the real line  and on the scaled lattice.

Given $s \in \R$, we define the inhomogeneous $L^2$-based Sobolev space  $H^s (\R)$ on the real line $\R$ via the norms
\begin{align*}
\| f \|_{H^s (\R)} &:= \| \jb{\dx}^s f \|_{L^2 (\R)} = \frac{1}{\sqrt{2 \pi}} \| \jb{\xi}^s \ft f (\xi) \|_{L^2_\xi (\R)} ,
\end{align*}

\noi
where the second equalities follows from Plancherel's identity.

Also, following \cite{KL}, given $0 < h \leq 1$ and $s \in \R$, we define the inhomogeneous $L^2$-based Sobolev space  $H^s (h \Z)$ on the scaled lattice $h \Z$ via the norms
\begin{align*}
\| f^h \|_{H^s (h \Z)} &:= \| \jb{\nb_h}^s f^h \|_{L^2 (h \Z)} = \frac{1}{\sqrt{2 \pi}} \big\| \jb{\tfrac{2}{h} \sin (\tfrac{h \xi}{2})}^s \F^h f^h (\xi) \big\|_{L_\xi^2 ([- \frac{\pi}{h}, \frac{\pi}{h}))} ,
\end{align*}

\noi
where the second equalities follow from Plancherel's identity \eqref{Plan}. 
Note that when $s = 0$, we have
\begin{align}
\| f^h \|_{H^0 (h \Z)} = \| f^h \|_{L^2 (h \Z)}.
\label{L2h}
\end{align}

We recall the following lemma on the equivalence of the $H^s (\R)$-norm and the $H^s (h \Z)$-norm under the extension operator $\mathcal{E}$ defined in \eqref{extend}. For a proof, see \cite[Lemma~2.3]{KL}, which is shown for the inhomogeneous case but also works for the homogeneous case.
\begin{lemma}
\label{LEM:Hs_equi}
Let $0 < h \leq 1$ and $s \leq 0$. Then, we have
\begin{align*}
\| f^h \|_{H^s (h \Z)} \les \| \mathcal{E} f^h \|_{H^s (\R)} \leq \| f^h \|_{H^s (h \Z)} 
\end{align*}

\noi
with the underlying constant independent of $h$.
\end{lemma}

Let us also mention the following equivalence of $H^s (h \Z)$-spaces.
\begin{lemma}
\label{LEM:embhs}
For any $0 < h \leq 1$ and $s_1, s_2 \in \R$ satisfying $s_1 > s_2$, we have
\begin{align*}
\| f^h \|_{H^{s_2} (h \Z)} \leq \| f^h \|_{H^{s_1} (h \Z)} \les h^{s_2 - s_1} \| f^h \|_{H^{s_2} (h \Z)} .
\end{align*}
\end{lemma}
\begin{proof}
The embeddings follow directly from the fact that $\jb{\tfrac{2}{h} \sin (\tfrac{h \xi}{2})} \geq 1$ and that given $\xi \in [- \frac{\pi}{h}, \frac{\pi}{h})$,
\begin{align*}
\jb{\tfrac{2}{h} \sin (\tfrac{h \xi}{2})}^{s_1 - s_2} \les \jb{\xi}^{s_1 - s_2} \les h^{s_2 - s_1} .
\end{align*}
\end{proof}

\subsection{$U^p$- and $V^p$-spaces}

In this subsection, we introduce the definitions and properties of $U^p$- and $V^p$-spaces using the formalism as presented in \cite{HHK}. These function spaces are also called the critical function spaces and play an important role in low regularity well-posedness of dispersive PDEs.

Let $H$ be a separable Hilbert space over $\mathbb{C}$. 
Let $\mathcal{T}$ be the set of finite partitions $\{t_k\}_{k = 0}^K$ of $\R$ such that $- \infty < t_0 < t_1 < \cdots < t_K \leq \infty$. 
Let $1 \leq p < \infty$. We define a $U^p$-atom as a step function $a : \R \to H$ of the form
\begin{align*}
a = \sum_{k = 1}^K \ind_{[t_{k - 1}, t_{k})} \phi_{k - 1} ,
\end{align*}

\noi
where $\{t_k\}_{k = 0}^K \in \mathcal{T}$ and $\{ \phi_k \}_{k = 0}^{K - 1} \subset H$ with $\sum_{k = 0}^{K - 1} \| \phi_k \|_{H}^p = 1$.
We then define the atomic space $U^p (\R ; H)$ as the set of functions $u : \R \to H$ of the form
\begin{align*}
u = \sum_{j = 1}^\infty \ld_j a_j
\end{align*}

\noi
with $a_j$'s being $U^p$-atoms and $\{\ld_j\}_{j \in \N} \in \ell^1$, and we define the norm
\begin{align*}
\| u \|_{U^p H} := \inf \bigg\{ \sum_{j = 1}^\infty |\ld_j| : u = \sum_{j = 1}^\infty \ld_j a_j, \, \ld_j \in \mathbb{C}, \, a_j \text{ $U^p$-atom} \bigg\} .
\end{align*}

\noi
The spaces $U^p (\R; H)$ are Banach spaces. Also, every $u \in U^p (\R; H)$ is right-continuous and satisfies $\lim_{t \to - \infty} u (t) = 0$.

Given $1 \leq p < \infty$, we define $V^p (\R; H)$ as the space of functions $u : \R \to H$ such that the norm
\begin{align*}
\| u \|_{V^p H} := \sup_{\{t_k\}_{k = 0}^K \in \mathcal{T}} \bigg( \sum_{k = 1}^K \| u (t_k) - u (t_{k - 1}) \|_{H}^p \bigg)^{\frac 1p}
\end{align*}

\noi
is finite. Here, we use the convention that $u (\infty) = 0$ for any function $u \in \R \to H$.
We also define $V^p_{\text{rc}} (\R ; H)$ as the closed subspace of all $u \in V^p (\R; H)$ that are right-continuous and satisfy $\lim_{t \to - \infty} u (t) = 0$, endowed with the same $V^p (\R; H)$ norm.
The spaces $V^p (\R; H)$ and $V^p_{\text{rc}} (\R; H)$ are Banach spaces.

We have the following embedding properties of $U^p$ and $V^p$ spaces; see \cite[Proposition~2.2, Proposition~2.4, and Corollary~2.6]{HHK}.
\begin{lemma}
\label{LEM:UVemb}
Let $1 \leq p < q < \infty$. Then, we have continuous embeddings: $U^p (\R; H) \subset V_{\text{rc}}^p (\R; H) \subset U^q (\R; H) \subset L^\infty (\R; H)$.
\end{lemma}

We also record the following interpolation result. For a proof, see \cite[Proposition~2.20]{HHK} and \cite[Lemma~2.4]{HTTz}.
\begin{lemma}
\label{LEM:interp}
Let $H$ be a Hilbert space. Let $B$ be a Banach space with norm $\| \cdot \|_B$. Let $k \in \N$ and $p_1, \dots, p_k > 2$. Suppose that $T : U^{p_1} (\R; H) \times \cdots \times U^{p_k} (\R; H) \to B$ is a bounded $k$-linear operator such that
\begin{align*}
\| T (u_1, \dots, u_k) \|_B \leq C_1 \prod_{j = 1}^k \| u_j \|_{U^{p_j} H}
\end{align*}

\noi
for some $C_1 > 0$, and there exists $0 < C_2 \leq C_1$ such that
\begin{align*}
\| T (u_1, \dots, u_k) \|_B \leq C_2 \prod_{j = 1}^k \| u_j \|_{U^{2} H} .
\end{align*}

\noi
Then, for any $u_1, \dots, u_k \in V_{\textup{rc}}^2 (\R; H)$, we have
\begin{align*}
\| T (u_1, \dots, u_k) \|_B \les C_2 \Big( \log \frac{C_1}{C_2} + 1 \Big)^k \prod_{j = 1}^k \| u_j \|_{V^2 H} .
\end{align*}
\end{lemma}


Given $1 < p, q < \infty$ with $\frac 1p + \frac{1}{q} = 1$, $u \in U^p (\R; H)$, $v \in V^{q} (\R; H)$, a partition $\mathfrak{t} = \{t_k\}_{k = 0}^K \in \mathcal{T}$, we define 
\begin{align}
B_{\mathfrak t} (u, v) := \sum_{k = 1}^K \big\langle u (t_{k - 1}) , v (t_k) - v (t_{k - 1}) \big\rangle_H ,
\label{defBuv}
\end{align}

\noi
where $\langle \cdot, \cdot \rangle_H$ is the inner product for the Hilbert space $H$.
Let us now recall the the following duality properties of $U^p$- and $V^p$-spaces.
See \cite[Proposition~2.7]{HHK} and \cite[Theorem~B.12]{KT18}.
\begin{lemma}
\label{LEM:UVdual}
Let $1 < p, q < \infty$ with $\frac 1p + \frac{1}{q} = 1$, $u \in U^p (\R; H)$, and $v \in V^{q} (\R; H)$. 
Then, for any partition $\mathfrak t \in \mathcal{T}$, we have
\begin{align}
|B_{\mathfrak{t}} (u, v)| \leq \| u \|_{U^p H} \| v \|_{V^{q} H} .
\label{Buv}
\end{align}

\noi
Also, there exists a unique number $B (u, v)$ such that for any $\eps > 0$, there exists $\mathfrak t \in \mathcal{T}$ such that for any $\mathfrak{t}' \supset \mathfrak t$,
\begin{align*}
|B_{\mathfrak{t}'} (u, v) - B (u, v)| < \eps .
\end{align*}

\noi
Moreover, we have
\begin{align}
\| u \|_{U^p H} = \sup_{v' \in V_{\textup{c}}^q: \| v' \|_{V^{q} H} \leq 1} |B (u, v')| ,
\label{UVdual}
\end{align}

\noi
where $V_{\textup{c}}^q \subset V^q (\R; H)$ is the set of continuous functions that vanish at $\pm \infty$.
\end{lemma}

Moreover, we recall the following transference principle on multilinear estimates. For a proof, see \cite[Proposition~2.19]{HHK}.
\begin{lemma}
\label{LEM:trans}
Let $k \in \N$ and $T: L^2 (\R) \times \cdots \times L^2 (\R) \to L^1_{\textup{loc}} (\R)$ be a $k$-linear operator. Given $t \in \R$, let $S(t) : L^2 (\R) \to L^2 (\R)$ be a linear operator which is continuous in $t$.

\smallskip \noi
\textup{(i)}
Let $1 \leq q < \infty$ and $1 \leq r \leq \infty$. Suppose that we have
\begin{align*}
\big\| T ( S (\cdot) \phi_1, \dots,  S (\cdot) \phi_k) \big\|_{L^q_t L_x^r} \les \prod_{j = 1}^k \| \phi_j \|_{L^2} .
\end{align*}

\noi
Then, for any $u_1, \dots, u_k \in U^q (\R; L^2 (\R))$, we have
\begin{align*}
\big\| T ( S(\cdot) u_1, \dots, S (\cdot) u_k) \big\|_{L^q_t L_x^r} \les \prod_{j = 1}^k \| u_j \|_{U^q L^2} .
\end{align*}

\smallskip \noi
\textup{(ii)} Let $1 \leq q, r \leq \infty$. Suppose that we have
\begin{align*}
\big\| T ( S(\cdot) \phi_1, \dots,  S(\cdot) \phi_k) \big\|_{L^q_x L_t^p} \les \prod_{j = 1}^k \| \phi_j \|_{L^2} .
\end{align*}

\noi
Then, for any finite $p \leq \min (q, r)$ and $u_1, \dots, u_k \in U^p (\R; L^2 (\R))$, we have
\begin{align*}
\big\| T ( S(\cdot) u_1, \dots, S(\cdot) u_k) \big\|_{L^r_x L_t^q} \les \prod_{j = 1}^k \| u_j \|_{U^p L^2} .
\end{align*}
\end{lemma}

We recall that $\Pb_N$ given $N \in 2^\Z$ denotes the smooth spatial frequency projection onto $\{ \xi \in \R: |\xi| \sim N \}$. Then, we have the following useful estimates.
\begin{lemma}
\label{LEM:PNl2}
\textup{(i)} We have
\begin{align*}
\| u \|_{V^2 L^2} \les \bigg( \sum_{N \in 2^\Z} \| \Pb_N u \|_{V^2 L^2}^2 \bigg)^{\frac 12} \les \bigg( \sum_{N \in 2^\Z} \| \Pb_N u \|_{U^2 L^2}^2 \bigg)^{\frac 12} \les \| u \|_{U^2 L^2} .
\end{align*}

\smallskip \noi
\textup{(ii)} For any $N \in 2^\Z$, we have
\begin{align*}
\| \dx \Pb_N u \|_{U^2 L^2} &\les N \| u \|_{U^2 L^2}, \\
\| \nb_h \Pb_N u \|_{U^2 L^2} &\les N \| u \|_{U^2 L^2}, \\
\| \dx^{-1} \Pb_N u \|_{U^2 L^2} &\les N^{-1} \| u \|_{U^2 L^2}, \\
\| \nb_h^{-1} \Pb_N P_{\leq \frac{\pi}{h}} u \|_{U^2 L^2} &\les N^{-1} \| u \|_{U^2 L^2}
\end{align*}

\noi
with the underlying constant independent of $h$.
\end{lemma}

\begin{proof}
For part (i), the first inequality holds due to the $L^2$-orthogonality and the second inequality holds due to Lemma~\ref{LEM:UVemb}.
For the third inequality, we note that it holds for atoms due to the $L^2$-orthogonality, which implies the result for general $u \in U^2 (\R; L^2 (\R))$. 
The estimate in part (ii) obviously holds for atoms thanks to $|\frac{2i}{h} \sin (\frac{h \xi}{2})| \les |\xi|$ for all $\xi \in \R$ and $|\frac{2i}{h} \sin (\frac{h \xi}{2})| \sim |\xi|$ for all $\xi \in [-\frac{\pi}{h}, \frac{\pi}{h})$, and so they hold for general $u \in U^2 (\R; L^2 (\R))$.
\end{proof}

We recall that $\Qb_M$ given $M \in 2^\Z$ denotes the smooth temporal frequency projection onto $\{ \tau \in \R: |\tau| \sim M \}$.
We also have the following properties on the temporal regularity. For a proof, see \cite[Corollary~2.18]{HHK}.
\begin{lemma}
\label{LEM:QM}
Let $M \in 2^\Z$. 

\smallskip \noi
\textup{(i)}
We have
\begin{align*}
\| \Qb_{> M} u \|_{L^2_t L^2_x} \les M^{- \frac 12} \| u \|_{V^2 L^2} 
\end{align*}

\noi
with the underlying constant independent of $M$.

\smallskip \noi
\textup{(ii)} We have
\begin{align*}
&\| \Qb_{\leq M} u \|_{V^2 L^2} \les \| u \|_{V^2 L^2} , \\
&\| \Qb_{> M} u \|_{V^2 L^2} \les \| u \|_{V^2 L^2} , \\
&\| \Qb_{\leq M} u \|_{U^2 L^2} \les \| u \|_{U^2 L^2} , \\
&\| \Qb_{> M} u \|_{U^2 L^2} \les \| u \|_{U^2 L^2} ,
\end{align*}

\noi
where the underlying constants are independent of $M$.
\end{lemma}

Let us also introduce the following $X^s$- and $Y^s$-spaces given $s \in \R$ based on the Littlewood-Paley decomposition. 
Given $s \in \R$, we define $X^s$ as the closure of the space of all $u \in C(\R; H^s (\R))$ such that the map $t \mapsto \Pb_N u (t)$ is in $U^2 L^2$ and the norm
\begin{align}
\| u \|_{X^s} := \bigg( \| \Pb_{\leq 1} u \|_{U^2 L^2}^2 + \sum_{N \in 2^{\N}} N^{2s} \| \Pb_N u \|_{U^2 L^2}^2 \bigg)^{\frac 12} 
\label{defXs}
\end{align}

\noi
is finite.
We also define $Y^s$ as the closure of the space of all $u \in C(\R; H^s (\R))$ such that the norm
\begin{align*}
\| u \|_{Y^s} := \bigg( \| \Pb_{\leq 1} u \|_{V^2 L^2}^2 + \sum_{N \in 2^\N} N^{2s} \| \Pb_N u \|_{V^2 L^2}^2 \bigg)^{\frac 12} 
\end{align*}

\noi
is finite. Using Lemma~\ref{LEM:UVemb}, the Littlewood-Paley decomposition \eqref{PNdecomp}, and Lemma~\ref{LEM:PNl2}~(i), we easily see that for any $s \in \R$,
\begin{align}
\| u \|_{L_t^\infty H_x^s} \les \| u \|_{V^2 H^s} \les \| u \|_{Y^s} \les \| u \|_{X^s} .
\label{V2Ys}
\end{align}

Given a time interval $I \subset \R$ and a Banach space $B$ with norm $\| \cdot \|_B$ on the space of distribution-valued functions $u : \R \to \mathcal{S}' (\R)$, we define the local-in-time norm
\begin{align}
\| u \|_{B_I} := \inf \{ \| v \|_B : u = v \text{ on } I \} .
\label{Bloc}
\end{align}

\noi
If $I = [0, T]$ for some $T > 0$, we write $B_T = B_{[0, T]}$.
From the definitions, it is not hard to see that for any time interval $I \subset \R$ and $1 \leq p < \infty$ (see also \cite[Lemma~A.1]{BOP2}),
\begin{align}
\| \ind_I u \|_{U^p L^2} \leq \| u \|_{U^p L^2} \quad \text{and} \quad
\| \ind_I u \|_{V^p L^2} \leq \| u \|_{V^p L^2} .
\label{timeUV}
\end{align}

\noi
Moreover, for any $s \in \R$, if $I \subset \R$ is an interval with $I = \bigcup_{j = 1}^K I_j$ for some intervals $I_j \subset \R$, $j = 1, \dots, K$ and $K \in \N$, then we have from \cite[Lemma~A.4]{BOP2} that
\begin{align}
\| u \|_{X^s_I} \leq \sum_{j = 1}^K \| u \|_{X^s_{I_j}} .
\label{XsIj}
\end{align}

Let us mention the following duality property of $X^s$- and $Y^s$-norms. The proof follows from minor modifications of that of \cite[Proposition~2.11]{HTTz} using the duality property in Lemma~\ref{LEM:UVdual}, and so we omit details.
\begin{lemma}
\label{LEM:XYdual}
Let $s \in \R$ and $T > 0$. Then, for any $u \in L^1 ([0, T]; H^s (\R))$, we have
\begin{align*}
\bigg\| \int_0^\cdot u (t) dt \bigg\|_{X_T^s} \leq \sup_{\| v \|_{Y^{-s}} \leq 1} \bigg| \int_0^T \int_\R u (t, x) v (t, x)  dx dt \bigg|.
\end{align*}
\end{lemma}

We also show the following lemma on the high frequency estimate in the $Y^s$-norm.
\begin{lemma}
\label{LEM:PRuUV}
Let $s_1, s_2 \in \R$ be such that $s_1 < s_2$. Then, for any $R > 0$, we have
\begin{align*}
\| P_{> R} u \|_{Y^{s_1}} &\les R^{s_1 - s_2} \| u \|_{Y^{s_2}}
\end{align*}

\noi
with the underlying constant independent of $R$.
\end{lemma}

\begin{proof}
By using the fact that $\Pb_N P_{> R} \neq 0$ if and only if $N \ges R$, we have
\begin{align*}
\| P_{> R} u \|_{Y^{s_1}} 
&= \bigg( \| \Pb_{\leq 1} P_{> R} u \|_{V^2 L^2}^2 + \sum_{\substack{N \in 2^\Z \\ N \geq 2}} N^{2 s_1} \| \Pb_N P_{> R} u \|_{V^2 L^2}^2 \bigg)^{\frac 12} \\
&\les \bigg( \max (1, R)^{2 (s_1 - s_2)} \| \Pb_{\leq 1} P_{> R} u \|_{V^2 L^2}^2 + \sum_{\substack{N \in 2^\Z \\ N \geq 2}} R^{2 (s_1 - s_2)} N^{2 s_2} \| \Pb_N P_{> R} u \|_{V^2 L^2}^2 \bigg)^{\frac 12} \\
&\leq R^{s_1 - s_2} \| u \|_{Y^{s_2}} ,
\end{align*}

\noi
as desired. 
\end{proof}

\section{Reformulation of equations}
\label{SEC:FPU}

In this section, we reformulate the scaled FPU system \eqref{rhwave0} and show how it is related to the KdV equation \eqref{KdV}. The presentation in this section explains the heuristics behind the convergence results in Theorem~\ref{THM:FPU}.

Let us first write out the Duhamel formulation for the initial value problem of \eqref{rhwave0}:
\begin{align}
\begin{split}
r^h (t) &= \cos (\tfrac{t}{h^2} |\nb_h|) r^h_0 +  \frac{h^2 \sin (\tfrac{t}{h^2} |\nb_h|)}{|\nb_h|} r^h_1 \\
&\quad - h^{-4} \int_0^t \sin (\tfrac{t - t'}{h^2} |\nb_h|) |\nb_h| \big( V' (h^2 r^h (t')) - h^2 r^h (t') \big) dt' 
\end{split}
\label{rhDuh0}
\end{align}

\noi
with initial data $(r_0^h, r_1^h) = (r^h, \dt r^h) |_{t = 0}$.
Using Taylor's theorem and the conditions in \eqref{Vcond}, we write
\begin{align*}
h^{-6} \big( V' (h^2 r^h) - h^2 r^h \big) = \frac 12 h^{-2} \big( (r^h)^2 + h^2 \mathcal{R} (r^h) \big) ,
\end{align*}

\noi
where the remainder term $\mathcal{R} (r^h)$ is defined by
\begin{align}
\mathcal{R} (r^h) := (r^h)^3 \int_0^1 (1 - \ta)^2 V^{(4)} (\ta h^2 r^h) d \ta .
\label{defR}
\end{align}

\noi
Thus, we arrive at the following equation for $r^h$:
\begin{align}
\begin{split}
r^h (t) &= \cos (\tfrac{t}{h^2} |\nb_h|) r^h_0 +  \frac{h^2 \sin (\tfrac{t}{h^2} |\nb_h|)}{|\nb_h|} r^h_1 \\
&\quad - \frac 12 \int_0^t \sin (\tfrac{t - t'}{h^2} |\nb_h|) |\nb_h| \big( (r^h (t'))^2 + h^2 \mathcal{R} (r^h (t')) \big) dt' .
\end{split}
\label{rhDuh}
\end{align}

Since the functions $x \mapsto \cos (x)$, $x \mapsto \frac{\sin (x)}{x}$, and $x \mapsto x \sin (x)$ are even, it makes sense to replace the Fourier multiplier operator $|\nb_h|$ by $i \nb_h$ with symbol $- \frac{2}{h} \sin (\frac{h \xi}{2})$, but this symbol is not $\frac{2 \pi}{h}$-periodic. Thus, instead of viewing $\nb_h$ as a Fourier multiplier operator on functions defined on the scaled lattice $h \Z$, we first apply the extension operator $\EE$ on both sides of \eqref{rhDuh} and then rewrite the Fourier multiplier operators to obtain
\begin{align}
\begin{split}
\EE r^h (t) &= \cos (\tfrac{t}{h^2} i \nb_h) \EE r^h_0 +  \frac{h^2 \sin (\tfrac{t}{h^2} i \nb_h)}{i \nb_h} \EE r^h_1 \\
&\quad - \frac 12 \int_0^t \sin (\tfrac{t - t'}{h^2} i \nb_h) i \nb_h \Big( \EE \big( (r^h (t'))^2 \big) + h^2 \EE \big( \mathcal{R} (r^h (t')) \big) \Big) dt' ,
\end{split}
\label{ErhDuh}
\end{align}

\noi
where the Fourier support of the either side is in $[- \frac{\pi}{h}, \frac{\pi}{h})$.
On this support, we can write
\begin{align}
\begin{split}
\cos (\tfrac{t}{h^2} i \nb_h) &= \frac 12 (e^{- \frac{t}{h^2} \nb_h} + e^{\frac{t}{h^2} \nb_h}), \\  
\sin (\tfrac{t}{h^2} i \nb_h) &= \frac{1}{2 i} (e^{- \frac{t}{h^2} \nb_h} - e^{\frac{t}{h^2} \nb_h}) .
\end{split}
\label{cossin}
\end{align}

Let us now define $r^h_+$ and $r^h_-$ as
\begin{align}
r^h_\pm (t) := \frac 12 \big( \mathcal{E} r^h (t) \mp h^2 \nb_h^{-1} \mathcal{E} \partial_t r^h (t) \big) ,
\label{rhpm}
\end{align}

\noi
so that we have the following decompositions of $\mathcal{E} r^h$ and $\EE \dt r^h$ in terms of $r_+^h$ and $r_-^h$:
\begin{align}
\begin{split}
\mathcal{E} r^h (t) &= r^h_+ (t) + r^h_- (t) , \\
h^2 \nb_h^{-1} \mathcal{E} \dt r^h (t) &= r^h_- (t) - r^h_+ (t)  .
\end{split}
\label{rhdecomp0}
\end{align}

\noi
Note that we have the expression
\begin{align}
\begin{split}
h^2 \nb_h^{-1} \EE \dt r^h (t) &= - i \sin (\tfrac{t}{h^2} i \nb_h) \EE r^h_0 +  \frac{h^2 \cos (\tfrac{t}{h^2} i \nb_h)}{\nb_h} \EE r^h_1 \\
&\quad + \frac 12 \int_0^t \cos (\tfrac{t - t'}{h^2} i \nb_h) \nb_h \Big( \EE \big( (r^h (t'))^2 \big) + h^2 \EE \big( \mathcal{R} (r^h (t')) \big) \Big) dt' .
\end{split}
\label{ErhDuh2}
\end{align}

\noi
Thus, a straightforward computation using \eqref{rhpm}, \eqref{ErhDuh}, \eqref{ErhDuh2}, and the identities \eqref{cossin} shows that $r^h_+$ and $r^h_-$ satisfy
\begin{align}
r^h_\pm (t) = e^{\mp \frac{t}{h^2} \nb_h} r^h_{\pm, 0} \mp \frac 14 \int_0^t e^{\mp \frac{t - t'}{h^2} \nb_h} \nb_h \Big( \EE \big( (r^h (t'))^2 \big) + h^2 \EE \big( \mathcal{R} (r^h (t')) \big) \Big) dt' 
\label{rhpmeq}
\end{align}

\noi
with $r^h_{\pm, 0}$ defined in \eqref{rhinit}.

We now define $w^h_+$ and $w^h_-$ as 
\begin{align}
w_\pm^h (t) := e^{\pm \frac{t}{h^2} \nb_h} r_\pm^h (t) = \frac 12 \big( e^{\pm \frac{t}{h^2} \nb_h} \mathcal{E} r^h (t) \mp h^2 e^{\pm \frac{t}{h^2} \nb_h} \nb_h^{-1} \mathcal{E} \partial_t r^h (t) \big) .
\label{defwh}
\end{align}

\noi
From \eqref{rhdecomp0} and \eqref{defwh}, we have the decompositions of $\mathcal{E} r^h$ and $\EE \dt r^h$ in terms of $w_+^h$ and $w_-^h$:
\begin{align}
\begin{split}
\EE r^h (t) &= e^{- \frac{t}{h^2} \nb_h} w^h_+ (t) + e^{\frac{t}{h^2} \nb_h} w^h_- (t) , \\
h^2 \nb_h^{-1} \EE \dt r^h (t) &= e^{\frac{t}{h^2} \nb_h} w_-^h (t) - e^{- \frac{t}{h^2} \nb_h} w_+^h (t) 
\end{split}
\label{rhdecomp}
\end{align}

\noi
and also in the discrete form
\begin{align}
r^h (t) = \mathcal{E} r^h (t) |_{h \Z} = \big( e^{- \frac{t}{h^2} \nb_h} w^h_+ (t) + e^{\frac{t}{h^2} \nb_h} w^h_- (t) \big) \big|_{h \Z} .
\label{rhwh}
\end{align}

\noi
From \eqref{defwh}, Lemma~\ref{LEM:Eprod}, and \eqref{rhpmeq}, we see that $w_+^h$ and $w_-^h$ satisfy the equations
\begin{align}
w^h_\pm (t) = r_{\pm, 0}^h \mp \frac 14 \int_0^t e^{\pm \frac{t'}{h^2} \nb_h} \nb_h \Big( P_{\leq \frac{\pi}{h}} \big( ( e^{\mp \frac{t'}{h^2} \nb_h} w^h_{\pm} (t') )^2 \big) + F (w^h_\pm) (t') \Big) dt' ,
\label{whDuh}
\end{align}

\noi
where 
\begin{align}
\begin{split}
F (w^h_\pm) (t) &:= 
2 P_{\leq \frac{\pi}{h}} \big( \cos (\tfrac{2 \pi}{h} \cdot) ( e^{\mp \frac{t}{h^2} \nb_h} w_\pm^h (t) )^2 \big) \\
&\quad + 2 P_{\leq \frac{\pi}{h}} \big( (1 + 2 \cos (\tfrac{2 \pi}{h} \cdot)) e^{\mp \frac{t}{h^2} \nb_h} w_\pm^h (t) e^{\pm \frac{t}{h^2} \nb_h} w_\mp^h (t) \big) \\
&\quad + P_{\leq \frac{\pi}{h}} \big( (1 + 2 \cos (\tfrac{2 \pi}{h} \cdot)) ( e^{\pm \frac{t}{h^2} \nb_h} w_\mp^h (t) )^2 \big) 
+ h^2 \EE \big( \mathcal{R} (r^h (t)) \big) .
\end{split}
\label{FR}
\end{align}

\noi
Here, we note that from \eqref{rhwh} that $r^h$ can be written in terms of $w_\pm^h$.


Note that $F$ contains three types of terms: terms with shifted frequency (by $\cos (\frac{2 \pi}{h} \cdot)$), interactions of oppositely moving waves, and the high-order remainder term. We will show that the terms in $F$ converge to zero as $h \to 0$ in appropriate norms. If the initial data $r_{\pm, 0}^h$ converges to some $u_{\pm, 0}$ as $h \to 0$, then we expect that the equation \eqref{whDuh} converge to
\begin{align}
w_\pm (t) = u_{\pm, 0} \mp \frac 14 \int_0^t e^{\pm \frac{t'}{24} \dx^3} \dx \big( (e^{\mp \frac{t'}{24} \dx^3} w_\pm (t'))^2 \big) dt' .
\label{wDuh}
\end{align} 

\noi
Indeed, we can expand the symbol of the Fourier multiplier operator $\frac{1}{h^2} \nb_h$ as
\begin{align}
\tfrac{2 i}{h^3} \sin (\tfrac{h \xi}{2}) = \tfrac{i}{h^2} \xi - \tfrac{i}{24} \xi^3 + O (h^2),
\label{disp} 
\end{align}

\noi
so that after inserting interacting linear dispersions in the time integral in \eqref{whDuh}, we see cubic dispersions that correspond to those in \eqref{wDuh}.
Each solution $w_\pm$ to the equation \eqref{wDuh} is nothing but the interaction representation of the solution $u_\pm$ to the KdV equation \eqref{KdV}. Indeed, we have the relation
\begin{align}
w_\pm (t) = e^{\pm \frac{t}{24} \dx^3} u_\pm (t) .
\label{wu}
\end{align}


\noi
The difference estimates \eqref{conv_goal1} and \eqref{conv_goal2} can then be explained by the expected convergence of $w_\pm^h$ to $w_\pm$, the relations \eqref{rhdecomp} and \eqref{wu}, and the expansion in \eqref{disp}. 

\begin{remark} \rm
\label{RMK:UpVp}
In existing literature, $U^p$- and $V^p$-spaces are usually used in the form adapted to the linear flow of the equation.
In this paper, we work with the interaction representations of the FPU system \eqref{whDuh} and the KdV equation \eqref{wDuh}, which allows us to use $U^p$- and $V^p$-type spaces without any adaptation to linear flows. This is particularly suitable in this work since we have two different linear flows: the linear FPU flow and the Airy flow. In fact, working with $U^p$- and $V^p$-type spaces adapted to the linear flow is equivalent to using the usual $U^p$- and $V^p$-type spaces on the interaction representation of the equation.
\end{remark}

\section{Global uniform bound for the FPU system}
\label{SEC:GWP}

In this section, we prove Proposition~\ref{PROP:FPUGWP} on global well-posedness of the FPU system \eqref{req1}. As a consequence, we then  show a global uniform-in-$h$ bound for the solution to the scaled FPU system \eqref{rhDuh0} by using a scaling argument.

\begin{proof}[Proof of Proposition~\ref{PROP:FPUGWP}]
For clarity, we split our proof into several steps.

\smallskip \noi
{\bf Step 1:} \eqref{req2} is locally well-posed.

We first show local well-posedness for \eqref{req2}, whose Duhamel formulation can be written as
\begin{align}
\begin{cases}
r (t) = \Gamma_1 (r) (t) \\
|\nb_1|^{-1} \dt r (t) = \Gamma_2 (r) (t) ,
\end{cases}
\label{req1Duh}
\end{align}

\noi
where
\begin{align*}
\Gamma_1 (r) (t) 
&:= \cos (t |\nb_1|) r_0 + \frac{\sin (t |\nb_1|)}{|\nb_1|} r_1 - \int_0^t \sin ((t - t') |\nb_1|) |\nb_1| \big( V' (r (t')) - r (t') \big) dt' , \\
\Gamma_2 (r) (t) 
&:= - \sin (t |\nb_1|) r_0 + \frac{\cos (t |\nb_1|)}{|\nb_1|} r_1 - \int_0^t \cos ((t - t') |\nb_1|) |\nb_1| \big( V' (r (t')) - r (t') \big) dt'.
\end{align*}

\noi
By using the mean value theorem, the fact that $V'(0) = 0$ from the conditions in \eqref{Vcond}, and the embedding $\ell^2 (\Z) \subset \ell^\infty (\Z)$, we get
\begin{align}
|V' (r)| \leq |r| \sup_{|R| \leq \| r \|_{\ell^\infty (\Z)}} |V'' (R)| \leq |r| \sup_{|R| \leq \| r \|_{\ell^2 (\Z)}} |V'' (R)| .
\label{Vest}
\end{align}

\noi
Also, from Plancherel's identity \eqref{Plan} with $h = 1$, we see that the Fourier multiplier operators $\cos (t |\nb_1|)$, $\sin (t |\nb_1|)$, and $|\nb_1|$ are bounded on $\ell^2 (\Z)$. Thus, by using \eqref{Vest}, we see that for any $0 < T_0 \leq 1$,
\begin{align}
\begin{split}
\| \Gamma_1 (r) \|_{C_{T_0} \ell^2 (\Z)} 
&\les \| r_0 \|_{\ell^2 (\Z)} + \| |\nb_1|^{-1} r_1 \|_{\ell^2 (\Z)} + T_0 \big( \| V' (r) \|_{C_{T_0} \ell^2 (\Z)} + \| r \|_{C_{T_0} \ell^2 (\Z)} \big) \\
&\leq \| r_0 \|_{\ell^2 (\Z)} + \| |\nb_1|^{-1} r_1 \|_{\ell^2 (\Z)} + T_0 \| r \|_{C_{T_0} \ell^2 (\Z)} \bigg( 1 + \sup_{|R| \leq \| r \|_{C_{T_0} \ell^2 (\Z)}} \! |V'' (R)| \bigg).
\end{split}
\label{G1}
\end{align}

\noi
Also, given $T_0 > 0$ and $r, r' \in C([0, T_0]; \ell^2 (\Z))$, by using the mean value theorem, we have
\begin{align}
\big| V' (r) - V' (r') \big| \leq |r - r'| \sup_{|R| \leq \max (\| r \|_{\ell^2 (\Z)} ,  \| r' \|_{\ell^2 (\Z)})} |V'' (R)| .
\label{V1bdd}
\end{align}

\noi
Thus, using \eqref{V1bdd}, we obtain the difference estimate
\begin{align}
\begin{split}
\| &\Gamma_1 (r) - \Gamma_1 (r') \|_{C_{T_0} \ell^2 (\Z)} \\
&\les T_0 \| r - r' \|_{C_{T_0} \ell^2 (\Z)} \bigg( 1 + \sup_{|R| \leq \max (\| r \|_{C_{T_0} \ell^2 (\Z)} ,  \| r' \|_{C_{T_0} \ell^2 (\Z)})} |V'' (R)| \bigg) .
\end{split}
\label{G2}
\end{align}

\noi
Similar bounds in \eqref{G1} and \eqref{G2} also apply to $\Gamma_2$. Thus, by taking $T_0 > 0$ sufficiently small, $(\Gamma_1, \Gamma_2)$ is a contraction mapping on a ball in $C([0, T_0]; \ell^2 (\Z) \times \ell^2 (\Z))$ of radius $C \| r_0 \|_{\ell^2 (\Z)} + C \| |\nb_1|^{-1} r_1 \|_{\ell^2 (\Z)}$ for some constant $C > 0$. This gives local well-posedness for \eqref{req2}. The uniqueness of the solution in the whole space $C([0, T_0]; \ell^2 (\Z) \times \ell^2 (\Z))$ follows from a standard continuity argument.

\smallskip \noi
{\bf Step 2:} $r$ is a classical solution satisfying \eqref{req2}.

Let $T_0 > 0$ be the local existence time in Step 1. Using \eqref{dFour_inv} with $h = 1$, we write the first component of \eqref{req1Duh} as
\begin{align}
\begin{split}
r (t, n)  
&= \frac{1}{2 \pi} \int_{- \pi}^{\pi} \cos (t |2 \sin (\tfrac{\xi}{2})|) \F^1 r_0 (\xi) e^{i \xi n} d \xi \\
&\quad + \frac{1}{2 \pi} \int_{- \pi}^{\pi} \frac{\sin (t |2 \sin (\tfrac{\xi}{2})|)}{|2 \sin (\tfrac{\xi}{2})|} \F^1 r_1 (\xi) e^{i \xi n} d \xi \\
&\quad - \frac{1}{2 \pi} \int_0^t \int_{- \pi}^{\pi} \sin \big( (t - t') |2 \sin (\tfrac{\xi}{2})| \big) |2 \sin (\tfrac{\xi}{2})| \big( \F^1 V' (r (t')) (\xi) \\
&\qquad \qquad - \F^1 r (t') (\xi) \big) e^{i \xi n} d\xi dt' , \quad (t, n) \in [0, T_0] \times \Z .
\end{split}
\label{rhDuhF}
\end{align}

\noi
We then take the time derivative of \eqref{rhDuhF} twice and check that \eqref{req2} holds for every $n \in \Z$,
where the switching of $\dt$ with the integral in $\xi$ can be easily justified using the dominated convergence theorem along with the facts that $r_0, |\nb_1|^{-1} r_1 \in \ell^2 (\Z)$ and $r, V' (r) \in C([0, T_0]; \ell^2 (\Z))$. 
Thus, we see that $r$ satisfies the equation \eqref{req2} in the classical sense.

\smallskip \noi
{\bf Step 3:} the FPU Hamiltonian $H (r)$ in \eqref{Hr} is conserved.

We again let $T_0 > 0$ be the local existence time in Step 1.
Recalling the notation in \eqref{dhpm}, we define 
\begin{align*}
p (t, n) = (\partial_1^+)^{-1} \dt r (t, n), \quad (t, n) \in [0, T_0] \times \Z ,
\end{align*}

\noi
so that
\begin{align}
\dt r (t, n) = \partial_1^+ p (t, n) = p (t, n + 1) - p (t, n) \big), \quad (t, n) \in [0, T_0] \times \Z .
\label{rhph}
\end{align}

\noi
Note that for each $t \in [0, T_0]$, $p (t)$  makes sense as an $\ell^2 (\Z)$ function since, by Plancherel's identity \eqref{Plan}, 
\begin{align}
\| p \|_{\ell^2 (\Z)} = \frac{1}{\sqrt{2 \pi}} \big\| \tfrac{1}{e^{i \xi} - 1} \F^1 \dt r \big\|_{L^2 ([- \pi, \pi))} = \| |\nb_1|^{-1} \dt r \|_{\ell^2 (\Z)} < \infty .
\label{phL2}
\end{align}

\noi
Then, by using \eqref{rhph} and the equation \eqref{req2}, we get
\begin{align*}
\dt^2 r = \partial_1^+ \dt p = \Dl_1 (V' (r)) .
\end{align*}

\noi
Since $\Dl_1 = \partial_1^+ \partial_1^-$, we get
\begin{align}
\dt p (t, n) = \partial_1^- (V' (r)) (t, n) = (V' (r)) (t, n) - (V' (r)) (t, n - 1), \quad (t, n) \in [0, T_0] \times \Z .
\label{dtph}
\end{align}

We now define the truncated scaled Hamiltonian
\begin{align*}
H^{K} (p, r) = \sum_{\substack{n \in \Z \\ |n| \leq K}} \Big( \frac 12 (p (n))^2 + V (r (n)) \Big)
\end{align*}

\noi
given $K \in \Z$ and $K > 0$. By using the identities \eqref{rhph} and \eqref{dtph}, we obtain
\begin{align*}
\frac{d}{dt} H^{K} (p, r) &=  \sum_{\substack{n \in \Z \\ |n| \leq K}} \Big( p (n) \dt p (n) + V' (r (n)) \dt r (n) \Big) \\
&= \sum_{\substack{n \in \Z \\ |n| \leq K}} \big( V' (r (n)) p (n + 1) - V' (r (n - 1)) p (n) \big) \\
&= V' (r (K)) p (K + 1) - V' (r (-K - 1)) p (-K) .
\end{align*}

\noi
Thus, since $p, V' (r) \in C([0, T_0]; \ell^2 (\Z))$, one can take the limit $K \to \infty$ and use the identity \eqref{phL2} to obtain the desired conservation of the scaled Hamiltonian $H (r)$.

\smallskip \noi
{\bf Step 4:} \eqref{req2} is globally well-posed.

We first note that by Taylor's theorem and the conditions in \eqref{Vcond}, we have
\begin{align}
V (r) = \frac 12 r^2 + \frac 12 r^3 \int_0^1 (1 - \ta)^2 V''' (\ta r) d \ta .
\label{Vexp}
\end{align}

\noi
From \eqref{Vexp} and the embedding $\ell^2 (\Z) \subset \ell^\infty (\Z)$, we obtain
\begin{align}
\begin{split}
\| r \|_{\ell^2 (\Z)}^2 &\leq  2 \sum_{n \in \Z} V (r (n))  + 2 \| r \|_{\ell^3 (\Z)}^3 \sup_{|R| \leq \| r \|_{\ell^\infty (\Z)}} |V''' (R)| \\
&\leq 2 \sum_{n \in \Z} V (r (n)) + 2 \| r \|_{\ell^2 (\Z)}^3 \sup_{|R| \leq \| r \|_{\ell^2 (\Z)}} |V''' (R)|.
\end{split}
\label{rhL2-1}
\end{align}

\noi
Similarly, we have
\begin{align}
\Big| \sum_{n \in \Z} V (r (n)) \Big| \leq \frac 12 \| r \|_{\ell^2 (\Z)}^2 + \frac 12 \| r \|_{\ell^2 (\Z)}^3 \sup_{|R| \leq \| r \|_{\ell^2 (\Z)}} |V''' (R)| .
\label{rhL2-2}
\end{align}

\noi
Thus, from \eqref{rhL2-1}, the conservation of the scaled Hamiltonian $H (r)$ in \eqref{Hr}, and \eqref{rhL2-2}, we obtain that for any $t \in \R_+$ such that $r (t)$ is defined,
\begin{align*}
\| &r (t) \|_{\ell^2 (\Z)}^2 + \| |\nb_1|^{-1} \dt r (t) \|_{\ell^2 (\Z)}^2 \\
&\leq 2 H (r (t)) + 2 \| r (t) \|_{\ell^2 (\Z)}^3 \sup_{|R| \leq \| r (t) \|_{\ell^2 (\Z)}} |V''' (R)| \\
&\leq 2 H (r (0)) + 2 \| r (t) \|_{\ell^2 (\Z)}^3 \sup_{|R| \leq \| r (t) \|_{\ell^2 (\Z)}} |V''' (R)| \\
&\leq \| r_0 \|_{\ell^2 (\Z)}^2 + \| |\nb_1|^{-1} r_1 \|_{\ell^2 (\Z)}^2 + 2 \| r_0 \|_{\ell^2 (\Z)}^3 \sup_{|R| \leq \| r_0 \|_{\ell^2 (\Z)}} |V''' (R)| \\
&\quad + 2 \| r (t) \|_{\ell^2 (\Z)}^3 \sup_{|R| \leq \| r (t) \|_{\ell^2 (\Z)}} |V''' (R)| .
\end{align*}

\noi
Then, using the condition \eqref{hcond0}, we get
\begin{align}
\begin{split}
\| &r (t) \|_{\ell^2 (\Z)}^2 + \| |\nb_1|^{-1} \dt r (t) \|_{\ell^2 (\Z)}^2 \\
&\leq \frac 54 \| r_0 \|_{\ell^2 (\Z)}^2 + \| |\nb_1|^{-1} r_1 \|_{\ell^2 (\Z)}^2 + 2  \| r (t) \|_{\ell^2 (\Z)}^3 \sup_{|R| \leq  \| r (t) \|_{\ell^2 (\Z)}} |V''' (R)| .
\end{split}
\label{boot}
\end{align}

We now perform a continuity argument. Let $T > 0$ and suppose that $r \in C([0, T]; \ell^2 (\Z))$. We also suppose that for some $0 < T_0 < T$,
\begin{align}
\| r \|_{C_{T_0} \ell^2 (\Z)}^2 + \| |\nb_1|^{-1} \dt r \|_{C_{T_0} \ell^2 (\Z)}^2 \leq 2 \| r_0 \|_{\ell^2 (\Z)}^2 + 2 \| |\nb_1|^{-1} r_1 \|_{\ell^2 (\Z)}^2 .
\label{boot1}
\end{align}

\noi
By \eqref{boot1} and continuity, there exists $T_1 > 0$ such that $T_0 + T_1 \leq T$ and
\begin{align}
\| r \|_{C_{T_0 + T_1} \ell^2 (\Z)}^2 + \| |\nb_1|^{-1} \dt r \|_{C_{T_0 + T_1} \ell^2 (\Z)}^2 \leq 3 \| r_0 \|_{\ell^2 (\Z)}^2 + 3 \| |\nb_1|^{-1} r_1 \|_{\ell^2 (\Z)}^2 .
\label{boot2}
\end{align}

\noi
Then, using \eqref{boot2} and the condition \eqref{hcond0}, we have
\begin{align}
\| r \|_{C_{T_0 + T_1} \ell^2 (\Z)} \leq 2 \| r_0 \|_{\ell^2 (\Z)} + 2 \| |\nb_1|^{-1} r_1 \|_{\ell^2 (\Z)} < 1 .
\label{boot3}
\end{align}

\noi
Thus, from \eqref{boot}, by using \eqref{boot2}, \eqref{boot3}, the fact that $(a^2 + b^2)^{\frac 32} \leq \sqrt{2} (a^3 + b^3)$ for $a, b \geq 0$, and the condition \eqref{hcond0} again along with $\frac{6 \sqrt{6}}{20} < \frac 34$, we obtain
\begin{align*}
&\| r \|_{C_{T_0 + T_1} \ell^2 (\Z)}^2 + \| |\nb_1|^{-1} \dt r \|_{C_{T_0 + T_1} \ell^2 (\Z)}^2 \\
&\leq \frac 54 \| r_0 \|_{\ell^2 (\Z)}^2 +  \| |\nb_1|^{-1} r_1 \|_{\ell^2 (\Z)}^2 + 6 \sqrt{6} \big( \| r_0 \|_{\ell^2 (\Z)}^3 + \| |\nb_1|^{-1} r_1 \|_{\ell^2 (\Z)}^3 \big) \sup_{|R| \leq 1} |V''' (R)| \\
&\leq 2 \| r_0 \|_{\ell^2 (\Z)}^2 + 2 \| |\nb_1|^{-1} r_1 \|_{\ell^2 (\Z)}^2 ,
\end{align*}

\noi
so that we get the bound \eqref{boot1} with $T_0$ replaced by $T_0 + T_1$. 
Thus, by repeating the above steps, we get
\begin{align}
\| r \|_{C_{T} \ell^2 (\Z)}^2 + \| |\nb_1|^{-1} \dt r \|_{C_{T} \ell^2 (\Z)}^2 \leq 2 \| r_0 \|_{\ell^2 (\Z)}^2 + 2 \| |\nb_1|^{-1} r_1 \|_{\ell^2 (\Z)}^2 ,
\label{rL2new}
\end{align}

\noi
which allows us to get global well-posedness of \eqref{req2} by iterating the local well-posedness argument.
\end{proof}

By the scaling \eqref{rhscale} and the bound \eqref{rL2new}, we immediately obtain the following global-in-time bound for the solution $r^h$ to the scaled FPU system \eqref{rhwave0}, which is essential for proving our global-in-time continuum limit result in Theorem~\ref{THM:FPU}~(ii).
\begin{corollary}
\label{COR:GWP}
Given $0 < h \leq 1$, let $r_0^h \in L^2 (h \Z)$ and $|\nb_h|^{-1} r_1^h \in L^2 (h \Z)$. Suppose that the condition \eqref{hcond} holds. Let $r^h \in C (\R_+; L^2 (h \Z))$ be the unique solution to the scaled FPU system \eqref{rhwave0} with $(r^h, \dt r^h) |_{t = 0} = (r_0^h, r_1^h)$. Then, for any $t \in \R_+$, we have the bound
\begin{align*}
\| r^h (t) \|_{L^2 (h \Z)} + \| h^2 |\nb_h|^{-1} \dt r^h (t) \|_{L^2 (h \Z)} \leq 2 \| r_0^h \|_{L^2 (h \Z)} + 2 \| h^2 |\nb_h|^{-1} r_1^h \|_{L^2 (h \Z)} .
\end{align*}
\end{corollary}

From Corollary~\ref{COR:GWP}, we see that if $r_0^h$ and $r_1^h$ satisfy the condition \eqref{r0cond2-1}, then we can bound $\| r^h (t) \|_{L^2 (h \Z)}$ and $\| h^2 |\nb_h|^{-1} \dt r^h (t) \|_{L^2 (h \Z)}$ uniformly in $h$.

\section{Key trilinear estimates}
\label{SEC:est}

In this section, we prove useful trilinear estimates using linear estimates and multilinear dispersion effects. In all of the estimates in this section, we work on functions defined on the continuum $\R$. 
Also, whenever we use ``$\les$'', the underlying constant is always independent of the scaling factor $h$.

\subsection{Linear estimates}
\label{SUB:lin}

In this subsection, we recall a few useful linear estimates for the linear FPU propagator $e^{\pm \frac{t}{h^2} \nb_h}$ and the Airy propagator $e^{\pm \frac{t}{24} \dx^3}$.

First of all, it is not hard to see from Plancherel's identity that these linear propagators enjoy the $H^s$-isometry for any $s \in \R$:
\begin{align*}
\| e^{\pm \frac{t}{h^2} \nb_h} \phi \|_{H^s} = \| \phi \|_{H^s} 
\quad \text{and} \quad
\| e^{\pm \frac{t}{24} \dx^3} \phi \|_{H^s} = \| \phi \|_{H^s} ,
\end{align*}

\noi
which we will use frequently in our later analysis.
We then recall some more useful linear estimates for $e^{\pm \frac{t}{h^2} \nb_h}$ and $e^{\pm \frac{t}{24} \dx^3}$ below.
\begin{lemma}
\label{LEM:lin}
Let $0 < h \leq 1$.

\smallskip \noi
\textup{(i) (Strichartz estimates)} 
Let $4 \leq q \leq \infty$ and $2 \leq r \leq \infty$ be such that $\frac{2}{q} + \frac{1}{r} = \frac 12$. 
Then, we have
\begin{align*}
\big\| |\nb_h|^{\frac 1q} e^{\pm \frac{t}{h^2} \nb_h}  P_{\leq \frac{\pi}{h}} \phi \big\|_{L_t^q L_x^r} \les \| \phi \|_{L^2} 
\end{align*}

\noi
and
\begin{align*}
\big\| |\dx|^{\frac 1q} e^{\pm \frac{t}{24} \dx^3} \phi \big\|_{L_t^q L_x^r} \les \| \phi \|_{L^2} .
\end{align*}

\smallskip \noi
\textup{(ii) (Maximal function estimates)} 
Let $s > \frac 34$ and $T > 0$. Then, we have
\begin{align*}
\big\| e^{\pm \frac{t}{h^2} (\nb_h - \dx)} P_{\leq \frac{\pi}{h}} \phi \big\|_{L_x^2  L_T^\infty} \les (1 + T)  \| \phi \|_{H^s} 
\end{align*}

\noi
and
\begin{align*}
\| e^{\pm \frac{t}{24} \dx^3} \phi \|_{L_x^2 L_T^\infty} \les (1 + T)^{\frac 12}  \| \phi \|_{H^s} .
\end{align*}

\smallskip \noi
\textup{(iii) (Local smoothing estimates)}
We have
\begin{align*}
\big\| \nb_h e^{\pm \frac{t}{h^2} (\nb_h - \dx)} P_{\leq \frac{\pi}{h}} \phi \big\|_{L_x^\infty L_t^2} \les \| \phi \|_{L^2}
\end{align*}

\noi
and
\begin{align*}
\| \dx e^{\pm \frac{t}{24} \dx^3} \phi \|_{L_x^\infty  L_t^2} \les \| \phi \|_{L^2} .
\end{align*}
\end{lemma}

\begin{proof}
The claimed estimates for the Airy flow are classical, and we refer the readers to \cite{KPV91, CS88}. 

We now consider the estimates for the FPU flow. Without loss of generality, we only work with the ``$+$'' signs.
The Strichartz estimate in part (i) and the maximal function estimate in part (ii) follow essentially from \cite[Lemma~2.7~(ii) and (iii)]{KL} with a much easier proof, and so we only indicate some key steps. 
Instead of the standard Littlewood-Paley decomposition \eqref{PNdecomp}, we define
\begin{align*}
\varphi_N^h (\xi) = \varphi ( \tfrac{h |\xi|}{\pi N} ) - \varphi ( \tfrac{2 h |\xi|}{\pi N} )
\end{align*}

\noi
given $N \in 2^\Z$, so that we have the decomposition
\begin{align*}
\sum_{\substack{N \in 2^\Z \\ N \leq 1}} \varphi_N^h (\xi) \ind_{[- \frac{\pi}{h}, \frac{\pi}{h}]} (\xi) = \ind_{[- \frac{\pi}{h}, \frac{\pi}{h}]} (\xi) .
\end{align*}

\noi
In particular, we see $\varphi_1^h$ is supported in $[- \frac{8 \pi}{5 h}, \frac{8 \pi}{5 h}]$, on which we have $|\frac{2}{h} \sin (\frac{h \xi}{2})| \sim |\xi|$.
Let $\Pb_N^h$ be the Fourier multiplier operator with symbol $\varphi_N^h$. Then, given $s \in \R$, we define the integral kernel for $|\nb_h|^s e^{\frac{t}{h^2} (\nb_h - \dx)} \Pb_N^h$ as
\begin{align*}
|\nb_h|^s K_N^h (t, x) = \frac{1}{2 \pi} \int_{- \frac{8 \pi}{5 h}}^{\frac{8 \pi}{5 h}} e^{i \frac{t}{h^3} (2 \sin (\frac{h \xi}{2}) - h \xi) + i \xi x} |\tfrac{2}{h} \sin (\tfrac{h \xi}{2})|^s \varphi_N^h (\xi) d \xi .
\end{align*}

\noi
Following the proof of \cite[Lemma~2.6]{KL}, we get the following estimates given $t \neq 0$, $T > 0$, and $N \in 2^\Z$ with $N \leq 1$:
\begin{align}
&\bigg\| \sum_{\substack{N \in 2^\Z \\ N \leq 1}} |\nb_h|^{\frac 12} K_N^h (t, x) \bigg\|_{L_x^\infty (\R)} \les |t|^{- \frac 12} , \label{Kest1} \\
&\Big\| \sup_{t \in [-2 T, 2 T]} | K_N^h (t, x) | \Big\|_{L_x^1 (\R)} \les (1 + T)^{\frac 12} h^{- \frac 32} N^{\frac 32} , \label{Kest2} \\
&\bigg\| \sup_{t \in [-2 T, 2 T]} \bigg| \sum_{\substack{N \in 2^\Z \\ N \leq h}} K_N^h (t, x) \bigg| \bigg\|_{L_x^1 (\R)} \les (1 + T)^2  \label{Kest3}
\end{align}

\noi
with the underlying constants independent of $h$, $t$, $T$, and $N$. 
The estimate \eqref{Kest1} is then used to obtain the Strichartz estimate in part (i); see the proof of \cite[Lemma~2.7~(ii)]{KL}. The estimates \eqref{Kest2} and \eqref{Kest3} are then used to obtain the maximal function estimate in part (ii); see the proof of \cite[Lemma~2.7~(iii)]{KL}.

For the local smoothing estimate in part (iii), we refer the readers to \cite[Proposition~5.1~(ii)]{HKY}.
\end{proof}

By using Lemma~\ref{LEM:lin}, the fact that $\frac{2}{h} | \sin (\frac{h \xi}{2}) | \sim |\xi|$ given $\xi \in [- \frac{\pi}{h}, \frac{\pi}{h}]$, and the transference principle in Lemma~\ref{LEM:trans}, we obtain the following linear estimates in $U^p L^2$-spaces.
\begin{corollary}
\label{COR:linU}
Let $0 < h \leq 1$.

\smallskip \noi
\textup{(i)} 
Let $4 \leq q \leq \infty$ and $2 \leq r \leq \infty$ be such that $\frac{2}{q} + \frac{1}{r} = \frac 12$. Let $N \in 2^\Z$.
Then, we have
\begin{align*}
\big\| e^{\pm \frac{t}{h^2} \nb_h} \Pb_N P_{\leq \frac{\pi}{h}} u \big\|_{L_t^q L_x^r} \les N^{- \frac 1q} \| \Pb_N u \|_{U^q L^2} 
\end{align*}

\noi
and
\begin{align*}
\big\| e^{\pm \frac{t}{24} \dx^3} \Pb_N u \big\|_{L_t^q L_x^r} \les N^{- \frac 1q} \| \Pb_N u \|_{U^q L^2} .
\end{align*}

\smallskip \noi
\textup{(ii)} 
Let $s > \frac 34$, $T > 0$, and $N \in 2^\Z$. Then, we have
\begin{align*}
\big\| e^{\pm \frac{t}{h^2} (\nb_h - \dx)} \Pb_{\les N} P_{\leq \frac{\pi}{h}} u \big\|_{L_x^2  L_T^\infty} \les (1 + T) (1 + N)^s \| \Pb_{\les N} u \|_{U^2 L^2} 
\end{align*}

\noi
and
\begin{align*}
\| e^{\pm \frac{t}{24} \dx^3} \Pb_{\les N} u \|_{L_x^2 L_T^\infty} \les (1 + T)^{\frac 12} (1 + N)^s \| \Pb_{\les N} u \|_{U^2 L^2} .
\end{align*}

\smallskip \noi
\textup{(iii)}
We have
\begin{align*}
\big\| \nb_h e^{\pm \frac{t}{h^2} (\nb_h - \dx)} P_{\leq \frac{\pi}{h}} u \big\|_{L_x^\infty  L_t^2} \les \| u \|_{U^2 L^2}
\end{align*}

\noi
and
\begin{align*}
\| \dx e^{\pm \frac{t}{24} \dx^3} u \|_{L_x^\infty  L_t^2} \les \| u \|_{U^2 L^2} .
\end{align*}
\end{corollary}

We now show the following estimates on the difference of the shifted linear FPU propagator $e^{\pm \frac{t}{h^2} (\nb_h - \dx)}$ and the Airy propagator $e^{\pm \frac{t}{24} \dx^3}$ in the $V^2 L^2$-space. 

\begin{lemma}
\label{LEM:V2diff}
Let $0 < h \leq 1$, $T > 0$, and $R > 0$. 

\smallskip \noi
\textup{(i)}
Given $u \in V_{\text{rc}}^2 (\R; L^2 (\R))$, we have
\begin{align}
\big\| \ind_{[0, T)} \big( e^{\pm \frac{t}{h^2} (\nb_h - \dx) \mp \frac{t}{24} \dx^3} - 1 \big) P_{\leq R} u \big\|_{V^2 L^2} \les T^{\frac 12} h R^{\frac 52} \| u \|_{V^2 L^2} .
\label{lin_diff}
\end{align}

\smallskip \noi
\textup{(ii)}
Given $u \in V_{\text{rc}}^2 (\R; L^2 (\R))$, we have
\begin{align}
\big\| \ind_{[0, T)} e^{\pm \frac{t}{h^2} (\nb_h - \dx) \mp \frac{t}{24} \dx^3} P_{\leq R} u \big\|_{V^2 L^2} \les (1 + T^{\frac 12} h R^{\frac 52}) \| u \|_{V^2 L^2} .
\label{lin_nodiff}
\end{align}

\smallskip \noi
\textup{(iii)}
Given $u \in U^2 (\R; L^2 (\R))$, we have
\begin{align}
\big\| \ind_{[0, T)} \big( e^{\pm \frac{t}{h^2} (\nb_h - \dx) \mp \frac{t}{24} \dx^3} - 1 \big) P_{\leq R} u \big\|_{U^2 L^2} \les (T^{\frac 12} h R^{\frac 52} + T h^2 R^5) \| u \|_{U^2 L^2} .
\label{U_diff}
\end{align}

\smallskip \noi
\textup{(iv)}
Given $u \in U^2 (\R; L^2 (\R))$, we have
\begin{align}
\big\| \ind_{[0, T)} e^{\pm \frac{t}{h^2} (\nb_h - \dx) \mp \frac{t}{24} \dx^3} P_{\leq R} u \big\|_{U^2 L^2} \les (1 + T h^2 R^5) \| u \|_{U^2 L^2} .
\label{U_nodiff}
\end{align}
\end{lemma}

\begin{proof}
(i) We only prove \eqref{lin_diff} with the ``$+$'' sign, since the case with the ``$-$'' sign is similar.
Let $\eps > 0$ be arbitrary. From the definition of the $V^2 L^2$-norm, we can find a partition $\{ t_k \}_{k = 0}^K$ with $0 \leq t_0 < t_1 < \cdots < t_K < T$
such that
\begin{align*}
&\text{LHS of \eqref{lin_diff}} \\
&\leq \bigg( \big\| \big( e^{\frac{t_0}{h^2} (\nb_h - \dx) - \frac{t_0}{24} \dx^3} - 1 \big) P_{\leq R} u (t_0) \big\|_{L^2}^2 + \big\| \big( e^{\frac{t_K}{h^2} (\nb_h - \dx) - \frac{t_K}{24} \dx^3} - 1 \big) P_{\leq R} u (t_K) \big\|_{L^2}^2 \\
&\quad + \sum_{k = 1}^K \Big\| \big( e^{\frac{t_k}{h^2} (\nb_h - \dx) - \frac{t_k}{24} \dx^3)} - 1 \big) P_{\leq R} u (t_k)  - \big( e^{\frac{t_{k - 1}}{h^2} (\nb_h - \dx) - \frac{t_{k - 1}}{24} \dx^3)} - 1 \big) P_{\leq R} u (t_{k - 1}) \Big\|_{L^2}^2 \bigg)^{\frac 12} + \eps ,
\end{align*}

\noi
and by splitting the difference term in the summation, we write
\begin{align}
\text{LHS of \eqref{lin_diff}} \leq \sum_{j = 1}^4 \textup{I}_j + \eps ,
\label{lindiff0}
\end{align}

\noi
where
\begin{align*}
\textup{I}_1 &:= \big\| \big( e^{\frac{t_0}{h^2} (\nb_h - \dx) - \frac{t_0}{24} \dx^3} - 1 \big) P_{\leq R} u (t_0) \big\|_{L^2} , \\
\textup{I}_2 &:= \big\| \big( e^{\frac{t_K}{h^2} (\nb_h - \dx) - \frac{t_K}{24} \dx^3} - 1 \big) P_{\leq R} u (t_K) \big\|_{L^2} , \\
\textup{I}_3 &:= \bigg( \sum_{k = 1}^K \Big\| \big( e^{\frac{t_k}{h^2} (\nb_h - \dx) - \frac{t_k}{24} \dx^3} - 1 \big) \big( P_{\leq R} u (t_k) - P_{\leq R} u (t_{k - 1}) \big) \Big\|_{L^2}^2 \bigg)^{\frac 12} , \\
\textup{I}_4 &:= \bigg( \sum_{k = 1}^K \Big\| \big( e^{\frac{t_k}{h^2} (\nb_h - \dx) - \frac{t_k}{24} \dx^3} - e^{\frac{t_{k - 1}}{h^2} (\nb_h - \dx) - \frac{t_{k - 1}}{24} \dx^3} \big) P_{\leq R} u (t_{k - 1}) \Big\|_{L^2}^2 \bigg)^{\frac 12} .
\end{align*}

Note that for any $\phi \in L^2 (\R)$ and $0 \leq k \leq K$, we have from Plancherel's identity that
\begin{align}
\Big\| \big( e^{\frac{t_k}{h^2} (\nb_h - \dx) - \frac{t_k}{24} \dx^3)} - 1 \big) P_{\leq R} \phi \Big\|_{L^2} = \Big\| \big( e^{\frac{i t_k}{h^3} (2 \sin (\frac{h \xi}{2}) - h \xi + \frac{(h \xi)^3}{24})} - 1 \big) (P_{\leq R} \phi)^\wedge (\xi) \Big\|_{L_\xi^2} .
\label{lindiff1-1}
\end{align}

\noi
Note that for any $\xi \in \R$, we have
\begin{align}
\tfrac{1}{h^3} \big| 2 \sin (\tfrac{h \xi}{2}) - h \xi + \tfrac{(h \xi)^3}{24} \big| \les h^2 |\xi|^5 ,
\label{taylor}
\end{align}

\noi
so that by the mean value theorem, we get
\begin{align}
\Big| e^{\frac{i t_k}{h^3} (2 \sin (\frac{h \xi}{2}) - h \xi + \frac{(h \xi)^3}{24})} - 1 \Big| \les \min (t_k h^2 |\xi|^5, 1) \les t_k^{\frac 12} h |\xi|^{\frac 52}.
\label{lindiff1-2}
\end{align}

\noi
Thus, from \eqref{lindiff1-1}, \eqref{lindiff1-2}, and Lemma~\ref{LEM:UVemb}, we get
\begin{align}
\begin{split}
\textup{I}_1 &\les \min (t_0 h^2 R^5, 1) \| P_{\leq R} u (t_0) \|_{L^2} \les T^{\frac 12} h R^{\frac 52} \| u \|_{V^2 L^2} , \\
\textup{I}_2 &\les \min (t_K h^2 R^5, 1) \| P_{\leq R} u (t_K) \|_{L^2} \les T^{\frac 12} h R^{\frac 52} \| u \|_{V^2 L^2} , \\
\textup{I}_3 &\les h R^{\frac 52} \bigg( \sum_{k = 1}^K t_k \| P_{\leq R} u (t_k) - P_{\leq R} u (t_{k - 1}) \|_{L^2}^2 \bigg)^{\frac 12} 
\leq T^{\frac 12} h R^{\frac 52} \| u \|_{V^2 L^2} .
\end{split}
\label{lindiff1}
\end{align}

\noi
For $\textup{I}_4$, we use the mean value theorem and \eqref{taylor} to obtain
\begin{align}
\begin{split}
\Big| &e^{\frac{i t_k}{h^3} (2 \sin (\frac{h \xi}{2}) - h \xi + \frac{(h \xi)^3}{24})} - e^{\frac{i t_{k - 1}}{h^3} (2 \sin (\frac{h \xi}{2}) - h \xi + \frac{(h \xi)^3}{24})} \Big| \\
&\les \min \Big( (t_k - t_{k - 1}) \big( \tfrac{1}{h^3} |2 \sin (\tfrac{h \xi}{2}) - h \xi + \tfrac{(h \xi)^3}{24} | \big), 1 \Big) \\
&\les (t_k - t_{k - 1})^{\frac 12} h |\xi|^{\frac 52} .
\end{split}
\label{lindiff2-1}
\end{align}

\noi
Thus, by using Plancherel's identity, \eqref{lindiff2-1}, and Lemma~\ref{LEM:UVemb}, we get
\begin{align}
\textup{I}_4 \les h R^{\frac 52} \bigg( \sum_{k = 1}^K (t_k - t_{k - 1}) \| P_{\leq R} u (t_{k - 1}) \|_{L^2}^2 \bigg)^{\frac 12}\les T^{\frac 12} h R^{\frac 52} \| u \|_{V^2 L^2} .
\label{lindiff2}
\end{align}

\noi
Combining \eqref{lindiff0}, \eqref{lindiff1}, and \eqref{lindiff2} and noticing that $\eps$ can be chosen arbitrarily small, we get the desired estimate \eqref{lin_diff}.

\smallskip \noi
(ii) Since
\begin{align*}
\big\| &\ind_{[0, T)} e^{\pm \frac{t}{h^2} (\nb_h - \dx) \mp \frac{t}{24} \dx^3} P_{\leq R} u \big\|_{V^2 L^2} \\
&\leq \big\| \ind_{[0, T)} \big( e^{\pm \frac{t}{h^2} (\nb_h - \dx) \mp \frac{t}{24} \dx^3} - 1 \big) P_{\leq R} u \big\|_{V^2 L^2} + \| \ind_{[0, T)} P_{\leq R} u \|_{V^2 L^2},
\end{align*}

\noi
we obtain the desired estimate \eqref{lin_nodiff} using part (i) and \eqref{timeUV}.

\smallskip \noi
(iii) To prove the estimate, we use the duality property in Lemma~\ref{LEM:UVdual}. As in part (i), we only prove \eqref{U_diff} with the ``$+$'' sign.
Consider any partition $\mathfrak{t} = \{ t_k \}_{k = -1}^{K + 1} \in \mathcal{T}$ such that $t_{-1} < 0 \leq t_0 < \cdots < t_K \leq T < t_{K + 1}$. 
Then, recalling the definition of $B_{\mathfrak{t}}$ in \eqref{defBuv}, we have for any $v \in V^2 (\R; L^2 (\R))$ that
\begin{align*}
B_{\mathfrak{t}} &\Big( \ind_{[0, T)} \big( e^{\frac{t}{h^2} (\nb_h - \dx) - \frac{t}{24} \dx^3} - 1 \big) P_{\leq R} u , v \Big)  \\
&= \sum_{k = 1}^K \int_\R \big( e^{\frac{t_{k - 1}}{h^2} (\nb_h - \dx) - \frac{t_{k - 1}}{24} \dx^3} - 1 \big) P_{\leq R} u (t_{k - 1}) ( \cj{v (t_k) - v (t_{k - 1})} ) dx \\
&\quad - \int_\R \big( e^{\frac{t_{K}}{h^2} (\nb_h - \dx) - \frac{t_{K}}{24} \dx^3} - 1 \big) P_{\leq R} u (t_{K})  \cj{v (t_K)}  dx .
\end{align*}

\noi
After rearranging the terms, we obtain
\begin{align}
\begin{split}
B_{\mathfrak{t}} &\Big( \ind_{[0, T)} \big( e^{\frac{t}{h^2} (\nb_h - \dx) - \frac{t}{24} \dx^3} - 1 \big) P_{\leq R} u , v \Big)  \\
&= B_{\mathfrak{t}} \Big( u, \ind_{[0, T)} \big( e^{- \frac{t}{h^2} (\nb_h - \dx) + \frac{t}{24} \dx^3} - 1 \big) P_{\leq R} v \Big) \\
&\quad - \sum_{k = 1}^K \int_\R u (t_{k - 1}) \cj{ \big( e^{-\frac{t_k}{h^2} (\nb_h - \dx) + \frac{t_k}{24} \dx^3} - e^{- \frac{t_{k - 1}}{h^2} (\nb_h - \dx) + \frac{t_{k - 1}}{24} \dx^3} \big)  P_{\leq R} v (t_k) } dx .
\end{split}
\label{U_diff0}
\end{align}

\noi
By using \eqref{Buv} in Lemma~\ref{LEM:UVdual} and part (i), we get
\begin{align}
\Big| B_{\mathfrak{t}} \Big( u, \ind_{[0, T)} \big( e^{- \frac{t}{h^2} (\nb_h - \dx) + \frac{t}{24} \dx^3} - 1 \big) P_{\leq R} v \Big) \Big| \leq T^{\frac 12} h R^{\frac 52} \| u \|_{U^2 L^2} \| v \|_{V^2 L^2} .
\label{U_diff1}
\end{align}

\noi
By the Cauchy-Schwarz inequality, \eqref{lindiff2-1}, and Lemma~\ref{LEM:UVemb}, we get
\begin{align}
\begin{split}
\bigg| &\sum_{k = 1}^K \int_\R u (t_{k - 1}) \cj{ \big( e^{- \frac{t_k}{h^2} (\nb_h - \dx) + \frac{t_k}{24} \dx^3} - e^{- \frac{t_{k - 1}}{h^2} (\nb_h - \dx) + \frac{t_{k - 1}}{24} \dx^3} \big)  P_{\leq R} v (t_k) } dx \bigg| \\
&\leq \sum_{k = 1}^K \| u (t_{k - 1}) \|_{L^2} \Big\| \big( e^{- \frac{t_k}{h^2} (\nb_h - \dx) + \frac{t_k}{24} \dx^3} - e^{- \frac{t_{k - 1}}{h^2} (\nb_h - \dx) + \frac{t_{k - 1}}{24} \dx^3} \big)  P_{\leq R} v (t_k) \Big\|_{L^2} \\
&\les \| u \|_{U^2 L^2} \| v \|_{V^2 L^2} \sum_{k = 1}^K (t_k - t_{k - 1}) h^2 R^5 \\
&\leq T h^2 R^5 \| u \|_{U^2 L^2} \| v \|_{V^2 L^2} .
\end{split}
\label{U_diff2}
\end{align}

\noi
Thus, the desired estimate \eqref{U_diff} follows from \eqref{UVdual} in Lemma~\ref{LEM:UVdual}, \eqref{U_diff0}, \eqref{U_diff1}, and \eqref{U_diff2}.

\smallskip \noi
\textup{(iv)} Since 
\begin{align*}
\big\| &\ind_{[0, T)} e^{\pm \frac{t}{h^2} (\nb_h - \dx) \mp \frac{t}{24} \dx^3} P_{\leq R} u \big\|_{U^2 L^2} \\
&\leq \big\| \ind_{[0, T)} \big( e^{\pm \frac{t}{h^2} (\nb_h - \dx) \mp \frac{t}{24} \dx^3} - 1 \big) P_{\leq R} u \big\|_{U^2 L^2} + \| \ind_{[0, T)} P_{\leq R} u \|_{U^2 L^2},
\end{align*}

\noi
we obtain the desired estimate \eqref{U_nodiff} using part (iii), \eqref{timeUV}, and the fact that $\| P_{\leq R} u \|_{U^2 L^2} \leq \| u \|_{U^2 L^2}$ using the atomic structure of the $U^2 L^2$-space.
\end{proof}

\subsection{Bilinear estimates}
\label{SUB:bilin}

In this subsection, we record a useful bilinear estimate using the transversality of curves and show how it can be used for the linear FPU flow and the Airy flow. 

\begin{lemma}
\label{LEM:bi_gen}
Let $\phi_1, \phi_2 \in L^2 (\R)$ whose supports are connected sets in $\R$. Let $\psi_1, \psi_2 : \R \to \R$ be continuously differentiable functions. Suppose that 
\begin{align}
|\psi_1' (\xi_1) - \psi_2' (\xi_2)| \sim L > 0
\label{trans}
\end{align} 

\noi
for all $\xi_1 \in \supp \phi_1$ and $\xi_2 \in \supp \phi_2$ with $\xi_1 + \xi_2 \in U$ on some open and connected set $U \subset \R$.  
Then, we have
\begin{align*}
\big\| (e^{i t \psi_1} \phi_1) * (e^{i t \psi_2} \phi_2) \big\|_{L^2_t L^2 (U)} \les L^{- \frac 12} \| \phi_1 \|_{L^2} \| \phi_2 \|_{L^2} ,
\end{align*} 
\end{lemma}

\begin{proof}
Note that we have
\begin{align}
\big\| (e^{i t \psi_1} \phi_1) * (e^{i t \psi_2} \phi_2) \big\|_{L_t^2 L^2}^2 = \int_\R \int_U \bigg| \int_\R e^{i t (\psi_1 (\xi_1) + \psi_2 (\xi - \xi_1)} \phi_1 (\xi_1) \phi_2 (\xi - \xi_1) d \xi_1 \bigg|^2 d \xi d t .
\label{bi1}
\end{align}

\noi
The condition \eqref{trans} implies that the function $F(\xi_1) := \psi_1 (\xi_1) + \psi_2 (\xi - \xi_1)$ is injective. By using a change of variable $\eta = F (\xi_1) = \psi_1 (\xi_1) + \psi_2 (\xi - \xi_1)$ along with the condition \eqref{trans}, we obtain from \eqref{bi1}, Fubini's theorem, Plancherel's theorem, and a reverse change of variable with \eqref{trans} that
\begin{align*}
\big\| (e^{i t \psi_1} \phi_1) * (e^{i t \psi_2} \phi_2) \big\|_{L_t^2 L^2}^2 &\les L^{-2} \int_\R \int_\R \bigg| \int_{F(\R)} e^{i t \eta} \phi_1 (F^{-1} (\eta)) \phi_2 (\xi - F^{-1} (\eta)) d \eta \bigg|^2 d \xi d t \\
&\sim L^{-2} \int_\R \int_{F(\R)} |\phi_1 (F^{-1} (\eta))|^2 |\phi_2 (\xi - F^{-1} (\eta))|^2 d \eta d \xi \\
&\les L^{-1} \int_\R \int_\R |\phi_1 (\xi_1)|^2 |\phi_2 (\xi - \xi_1)|^2 d \xi_1 d \xi \\
&= L^{-1} \| \phi_1 \|_{L^2}^2 \| \phi_2 \|_{L^2}^2 ,
\end{align*} 

\noi
which gives the desired estimate.
\end{proof}

We now use Lemma~\ref{LEM:bi_gen} to show the following useful bilinear estimates with frequency localized functions under the linear FPU flow or the Airy flow.

\begin{lemma}
\label{LEM:bilin}
Let $0 < h \leq 1$.

\smallskip \noi
\textup{(i)} If $N_1, N_2 \in 2^\Z$ satisfy $N_1 \gg N_2$, then
\begin{align*}
\big\| e^{\pm \frac{t}{h^2} \nb_h} \Pb_{N_1} P_{\leq \frac{\pi}{h}} \phi_1 e^{\pm \frac{t}{h^2} \nb_h} \Pb_{N_2} P_{\leq \frac{\pi}{h}} \phi_2 \big\|_{L_t^2 L_x^2} \les N_1^{-1} \| \Pb_{N_1} \phi_1 \|_{L^2} \| \Pb_{N_2} \phi_2 \|_{L^2} 
\end{align*}

\noi
and
\begin{align*}
\big\| e^{\pm \frac{t}{24} \dx^3} \Pb_{N_1} \phi_1 e^{\pm \frac{t}{24} \dx^3} \Pb_{N_2} \phi_2 \big\|_{L_t^2 L_x^2} \les N_1^{-1} \| \Pb_{N_1} \phi_1 \|_{L^2} \| \Pb_{N_2} \phi_2 \|_{L^2} .
\end{align*}

\smallskip \noi
\textup{(ii)}
If $N_1, N_2, N_3 \in 2^\Z$ satisfy $N_1 \sim N_2 \gg N_3$, then
\begin{align*}
\big\| \Pb_{N_3} \big( e^{\pm \frac{t}{h^2} \nb_h} \Pb_{N_1} P_{\leq \frac{\pi}{h}} \phi_1 e^{\pm \frac{t}{h^2} \nb_h} \Pb_{N_2} P_{\leq \frac{\pi}{h}} \phi_2 \big) \big\|_{L_t^2 L_x^2} \les N_1^{-\frac 12} N_3^{-\frac 12} \| \Pb_{N_1} \phi_1 \|_{L^2} \| \Pb_{N_2} \phi_2 \|_{L^2} 
\end{align*}

\noi
and
\begin{align*}
\big\| \Pb_{N_3} \big( e^{\pm \frac{t}{24} \dx^3} \Pb_{N_1} \phi_1 e^{\pm \frac{t}{24} \dx^3} \Pb_{N_2} \phi_2 \big) \big\|_{L_t^2 L_x^2} \les N_1^{-\frac 12} N_3^{-\frac 12} \| \Pb_{N_1} \phi_1 \|_{L^2} \| \Pb_{N_2} \phi_2 \|_{L^2} .
\end{align*}

\smallskip \noi
\textup{(iii)} We have
\begin{align*}
\big\| e^{\frac{t}{h^2} \nb_h} P_{\leq \frac{9 \pi}{10 h}} \phi_1 e^{- \frac{t}{h^2} \nb_h} P_{\leq \frac{\pi}{h}} \phi_2 \big\|_{L_t^2 L_x^2} \les h \| \phi_1 \|_{L^2} \| \phi_2 \|_{L^2} .
\end{align*}

\smallskip \noi
\textup{(iv)} Given $0 < N \ll \frac{1}{h}$, let $A_N = \{ \xi \in \R : \frac{N}{4} \leq |\xi - \frac{2 \pi}{h}| \leq 4N \}$ and $B_N = \{ \xi \in \R : \frac{N}{4} \leq |\xi + \frac{2 \pi}{h}| \leq 4N \}$. Then, we have
\begin{align*}
\big\| P_{A_N} \big( e^{\frac{t}{h^2} \nb_h} P_{[\frac{9 \pi}{10 h}, \frac{\pi}{h}]} \phi_1 e^{- \frac{t}{h^2} \nb_h} P_{[\frac{9 \pi}{10 h}, \frac{\pi}{h}]} \phi_2 \big) \big\|_{L_t^2 L_x^2} \les h^{\frac 12} N^{- \frac 12} \| \phi_1 \|_{L^2} \| \phi_2 \|_{L^2} 
\end{align*}

\noi
and
\begin{align*}
\big\| P_{B_N} \big( e^{\frac{t}{h^2} \nb_h} P_{[- \frac{\pi}{h}, - \frac{9 \pi}{10 h}]} \phi_1 e^{- \frac{t}{h^2} \nb_h} P_{[- \frac{\pi}{h}, - \frac{9 \pi}{10 h}]} \phi_2 \big) \big\|_{L_t^2 L_x^2} \les h^{\frac 12} N^{- \frac 12} \| \phi_1 \|_{L^2} \| \phi_2 \|_{L^2} 
\end{align*}
\end{lemma}

\begin{proof}
The estimates in (i) and (ii) follow directly from Lemma~\ref{LEM:bi_gen} with $\psi_1 (\xi) = \psi_2 (\xi) = \frac{2}{h^3} \sin (\frac{h \xi}{2})$ or $\psi_1 (\xi) = \psi_2 (\xi) = \xi^3$ and the observations that given $\xi = \xi_1 + \xi_2$,
\begin{align*}
\tfrac{1}{h^2} |\cos (\tfrac{h \xi_1}{2}) - \cos (\tfrac{h \xi_2}{2})| = \tfrac{2}{h^2} |\sin (\tfrac{h (\xi_1 + \xi_2)}{4}) \sin (\tfrac{h (\xi_1 - \xi_2)}{4})| \sim |\xi (\xi_1 - \xi_2)| 
\end{align*}

\noi
and
\begin{align*}
|\xi_1^2 - \xi_2^2| = |\xi (\xi_1 - \xi_2)| .
\end{align*}

\noi
For (iii), we use Lemma~\ref{LEM:bi_gen} with $\psi_1 (\xi) = \frac{2}{h^3} \sin (\frac{h \xi}{2})$ and $\psi_2 (\xi) = - \frac{2}{h^3} \sin (\frac{h \xi}{2})$ and the observation that
\begin{align*}
\tfrac{1}{h^2} |\cos (\tfrac{h \xi_1}{2}) + \cos (\tfrac{h \xi_2}{2})| = \tfrac{2}{h^2} |\cos (\tfrac{h (\xi_1 + \xi_2)}{4}) \cos (\tfrac{h (\xi_1 - \xi_2)}{4})| \sim h^{-2}
\end{align*}

\noi
given $|\xi_1| \leq \frac{9 \pi}{10 h}$ and $|\xi_2| \leq \frac{\pi}{h}$. For (iv), we focus on the first estimate, as the second estimate follows from the same way. We use the same step as in part (iii), but with the additional observations that
\begin{align*}
|\cos (\tfrac{h (\xi_1 - \xi_2)}{4})| \sim 1
\end{align*}

\noi
given $\frac{9 \pi}{10 h} \leq \xi_1, \xi_2 \leq \frac{\pi}{h}$ and
\begin{align*}
\tfrac{1}{h} |\cos (\tfrac{h (\xi_1 + \xi_2)}{4})| = \tfrac{1}{h} \big| \sin \big( \tfrac{h (\xi_1 + \xi_2 - \frac{2 \pi}{h})}{4} \big) \big| \sim N 
\end{align*}

\noi
given $\xi_1 + \xi_2 \in A_N$ with $0 < N \ll \frac{1}{h}$.
\end{proof}

We remark that the technical bilinear estimates in part (iv) will only be used once in a special occasion (see Case 2 in Lemma~\ref{LEM:rem3} below).
By using Lemma~\ref{LEM:bilin} and the transference principle in Lemma~\ref{LEM:trans}, we obtain the following bilinear estimates in $U^2 L^2$-spaces.
\begin{corollary}
\label{COR:bilinU}
Let $0 < h \leq 1$.

\smallskip \noi
\textup{(i)} If $N_1, N_2 \in 2^\Z$ satisfy $N_1 \gg N_2$, then
\begin{align*}
\big\| e^{\pm \frac{t}{h^2} \nb_h} \Pb_{N_1} P_{\leq \frac{\pi}{h}} u_1 e^{\pm \frac{t}{h^2} \nb_h} \Pb_{N_2} P_{\leq \frac{\pi}{h}} u_2 \big\|_{L_t^2 L_x^2} \les N_1^{-1} \| \Pb_{N_1} u_1 \|_{U^2 L^2} \| \Pb_{N_2} u_2 \|_{U^2 L^2} 
\end{align*}

\noi
and
\begin{align*}
\big\| e^{\pm \frac{t}{24} \dx^3} \Pb_{N_1} u_1 e^{\pm \frac{t}{24} \dx^3} \Pb_{N_2} u_2 \big\|_{L_t^2 L_x^2} \les N_1^{-1} \| \Pb_{N_1} u_1 \|_{U^2 L^2} \| \Pb_{N_2} u_2 \|_{U^2 L^2} .
\end{align*}

\smallskip \noi
\textup{(ii)}
If $N_1, N_2, N_3 \in 2^\Z$ satisfy $N_1 \sim N_2 \gg N_3$, then
\begin{align*}
\big\| \Pb_{N_3} \big( e^{\pm \frac{t}{h^2} \nb_h} \Pb_{N_1} P_{\leq \frac{\pi}{h}} u_1 e^{\pm \frac{t}{h^2} \nb_h} \Pb_{N_2} P_{\leq \frac{\pi}{h}} u_2 \big) \big\|_{L_t^2 L_x^2} 
\les N_1^{-\frac 12} N_3^{-\frac 12} \| \Pb_{N_1} u_1 \|_{U^2 L^2} \| \Pb_{N_2} u_2 \|_{U^2 L^2} 
\end{align*}

\noi
and
\begin{align*}
\big\| \Pb_{N_3} \big( e^{\pm \frac{t}{24} \dx^3} \Pb_{N_1} u_1 e^{\pm \frac{t}{24} \dx^3} \Pb_{N_2} u_2 \big) \big\|_{L_t^2 L_x^2} \les N_1^{-\frac 12} N_3^{-\frac 12} \| \Pb_{N_1} u_1 \|_{U^2 L^2} \| \Pb_{N_2} u_2 \|_{U^2 L^2} .
\end{align*}

\smallskip \noi
\textup{(iii)} We have
\begin{align*}
\big\| e^{\frac{t}{h^2} \nb_h} P_{\leq \frac{9 \pi}{10 h}} u_1 e^{- \frac{t}{h^2} \nb_h} P_{\leq \frac{\pi}{h}} u_2 \big\|_{L_t^2 L_x^2} \les h \| u_1 \|_{U^2 L^2} \| u_2 \|_{U^2 L^2} .
\end{align*}

\smallskip \noi
\textup{(iv)} Given $0 < N \ll \frac{1}{h}$, let $A_N = \{ \xi \in \R : \frac{N}{4} \leq |\xi - \frac{2 \pi}{h}| \leq 4N \}$  and $B_N = \{ \xi \in \R : \frac{N}{4} \leq |\xi + \frac{2 \pi}{h}| \leq 4N \}$. Then, we have
\begin{align*}
\big\| P_{A_N} \big( e^{\frac{t}{h^2} \nb_h} P_{[\frac{9 \pi}{10 h}, \frac{\pi}{h}]} u_1 e^{- \frac{t}{h^2} \nb_h} P_{[\frac{9 \pi}{10 h}, \frac{\pi}{h}]} u_2 \big) \big\|_{L_t^2 L_x^2} \les h^{\frac 12} N^{- \frac 12} \| u_1 \|_{U^2 L^2} \| u_2 \|_{U^2 L^2} 
\end{align*}

\noi
and
\begin{align*}
\big\| P_{B_N} \big( e^{\frac{t}{h^2} \nb_h} P_{[- \frac{\pi}{h}, - \frac{9 \pi}{10 h}]} u_1 e^{- \frac{t}{h^2} \nb_h} P_{[- \frac{\pi}{h}, - \frac{9 \pi}{10 h}]} u_2 \big) \big\|_{L_t^2 L_x^2} \les h^{\frac 12} N^{- \frac 12} \| u_1 \|_{U^2 L^2} \| u_2 \|_{U^2 L^2} .
\end{align*}
\end{corollary}

\subsection{Trilinear uniform estimates}

The main goal of this subsection is to establish a key trilinear estimate both for the linear FPU flow and for the Airy flow. The estimate will play a central role in proving the convergence of the FPU system to the KdV equation in the low regularity setting. We prove the estimate by combining the linear estimates in Subsection~\ref{SUB:lin}, the bilinear estimates in Subsection~\ref{SUB:bilin}, and the trilinear dispersive smoothing effect.

Let us first show the following technical lemma, where we discuss all the non-trivial scenarios of frequency interactions.

\begin{lemma}
\label{LEM:tri_gen}
Let $0 < h \leq 1$ and $0 < T \leq 1$. 
Let $N_1, N_2, N_3 \in 2^{\N}$. 

\smallskip \noi
\textup{(i)}
If $N_1 \sim N_2 \sim N_3 \gg 1$, then we have
\begin{align*}
\bigg| &\int_0^T \int_\R \big( e^{\pm \frac{t}{h^2} \nb_h} \Pb_{N_1} P_{\leq \frac{\pi}{h}} u_1 \big) \big( e^{\pm \frac{t}{h^2} \nb_h} \Pb_{N_2} P_{\leq \frac{\pi}{h}} u_2 \big) \big( e^{\pm \frac{t}{h^2} \nb_h} \Pb_{N_3} P_{\leq \frac{\pi}{h}} u_3 \big) dx dt \bigg| \\
&\les T^{\frac 14} N_1^{- \frac 74} \| \Pb_{N_1} u_1 \|_{V^2 L^2} \| \Pb_{N_2} u_2 \|_{V^2 L^2} \| \Pb_{N_3} u_3 \|_{V^2 L^2} . 
\end{align*}

\smallskip \noi
\textup{(ii)}
If $N_1 \sim N_2 \gg N_3 \gg 1$, then for any $\ta > 0$ sufficiently small, we have
\begin{align*}
\bigg| &\int_0^T \int_\R \big( e^{\pm \frac{t}{h^2} \nb_h} \Pb_{N_1} P_{\leq \frac{\pi}{h}} u_1 \big) \big( e^{\pm \frac{t}{h^2} \nb_h} \Pb_{N_2} P_{\leq \frac{\pi}{h}} u_2 \big) \big( e^{\pm \frac{t}{h^2} \nb_h} \Pb_{N_3} P_{\leq \frac{\pi}{h}} u_3 \big) dx dt \bigg| \\
&\les T^\ta N_1^{- \frac 32 + 10 \ta} N_3^{-1} \| \Pb_{N_1} u_1 \|_{V^2 L^2} \| \Pb_{N_2} u_2 \|_{V^2 L^2} \| \Pb_{N_3} u_3 \|_{V^2 L^2} .
\end{align*}

\smallskip \noi
\textup{(iii)}
If $N_1 \sim N_2 \gg 1$, then for any $\ta > 0$ sufficiently small, we have
\begin{align*}
\bigg| &\int_0^T \int_\R \big( e^{\pm \frac{t}{h^2} \nb_h} \Pb_{N_1} P_{\leq \frac{\pi}{h}} u_1 \big) \big( e^{\pm \frac{t}{h^2} \nb_h} \Pb_{N_2} P_{\leq \frac{\pi}{h}} u_2 \big) \big( \nb_h e^{\pm \frac{t}{h^2} \nb_h} \Pb_{\les 1} u_3 \big) dx dt \bigg| \\
&\les T^\ta N_1^{- \frac 32 + 20 \ta} \| \Pb_{N_1} u_1 \|_{V^2 L^2} \| \Pb_{N_2} u_2 \|_{V^2 L^2} \| \Pb_{\les 1} u_3 \|_{V^2 L^2} .
\end{align*}

\smallskip \noi
\textup{(iv)}
If $N_1 \sim N_3 \gg 1$, then we have
\begin{align*}
\bigg| &\int_0^T \int_\R \big( e^{\pm \frac{t}{h^2} \nb_h} \Pb_{N_1} P_{\leq \frac{\pi}{h}} u_1 \big) \big( e^{\pm \frac{t}{h^2} \nb_h} \Pb_{\les 1} u_2 \big) \big( e^{\pm \frac{t}{h^2} \nb_h} \Pb_{N_3} P_{\leq \frac{\pi}{h}} u_3 \big) dx dt \bigg| \\
&\les T^{\frac 12} N_1^{-1} \| \Pb_{N_1} u_1 \|_{U^2 L^2} \| \Pb_{\les 1} u_2 \|_{U^2 L^2} \| \Pb_{N_3} u_3 \|_{V^2 L^2} .
\end{align*}

\smallskip \noi
Moreover, all of the above estimates hold if the operator $\nb_h$ on $u_3$ is replaced by $\dx$ and/or the terms $e^{\pm \frac{t}{h^2} \nb_h} \Pb_{N_j} P_{\leq \frac{\pi}{h}} u_j$ are replaced by $e^{\pm \frac{t}{24} \dx^3} u_j$, $j = 1, 2, 3$.
\end{lemma}
\begin{proof}
We only focus on the FPU flows, since the estimates for the KdV flows follow similarly. 
Also, without loss of generality, we only prove the estimates with the ``$+$'' signs.
Also, a slight manipulation on Fourier multiplier operators allows us to write
\begin{align}
\begin{split}
&\int_0^T \int_\R \big( e^{\frac{t}{h^2} \nb_h} \Pb_{N_1} P_{\leq \frac{\pi}{h}} u_1 \big) \big( e^{\frac{t}{h^2} \nb_h} \Pb_{N_2} P_{\leq \frac{\pi}{h}} u_2 \big) \big( e^{\frac{t}{h^2} \nb_h} \Pb_{N_3} P_{\leq \frac{\pi}{h}} u_3 \big) dx dt \\
&= \int_\R \int_\R \Big( \prod_{j = 1, 2}  e^{\frac{t}{h^2} \nb_h} \Pb_{N_j} P_{\leq \frac{\pi}{h}} \ind_{[0, T)} u_j  \Big)  e^{\frac{t}{h^2} \nb_h} \Pb_{N_3} P_{\leq \frac{\pi}{h}} \ind_{[0, T)} u_3  dx dt .
\end{split}
\label{indT}
\end{align}

\noi
A similar formulation applies if any $\Pb_{N_j}$ is replaced by $\Pb_{\leq 1}$.
Moreover, a crude estimate using H\"older's inequality, Sobolev's inequalities, and Lemma~\ref{LEM:UVemb} yields
\begin{align}
\begin{split}
\bigg| &\int_0^T \int_\R \Big( \prod_{j = 1, 2} e^{\frac{t}{h^2} \nb_h} \Pb_{N_j} P_{\leq \frac{\pi}{h}} u_j \Big)  e^{\frac{t}{h^2} \nb_h} \Pb_{N_3} P_{\leq \frac{\pi}{h}} u_3  dx dt \bigg| \\
&\les T N_1^{\frac 16} N_2^{\frac 16} N_3^{\frac 16} \| \Pb_{N_1} u_1 \|_{L_t^\infty L_x^2} \| \Pb_{N_2} u_2 \|_{L_t^\infty L_x^2} \| \Pb_{N_3} u_3 \|_{L_t^\infty L_x^2} \\
&\les T N_1^{\frac 16} N_2^{\frac 16} N_3^{\frac 16} \| \Pb_{N_1} u_1 \|_{U^4 L^2} \| \Pb_{N_2} u_2 \|_{U^4 L^2} \| \Pb_{N_3} u_3 \|_{U^4 L^2} .
\end{split}
\label{crude}
\end{align}


\medskip \noi
(i) In this case, we work with the right-hand side of \eqref{indT}. 
Let $(\tau_j, \xi_j)$ be the temporal-spatial frequency of $e^{\frac{t}{h^2} \nb_h} \Pb_{N_j} P_{\leq \frac{\pi}{h}} \ind_{[0, T)} u_j$, $j = 1, 2, 3$. 
Then, we must have $\tau_1 + \tau_2 + \tau_3 = 0$, $\xi_1 + \xi_2 + \xi_3 = 0$, $|\xi_j| \sim N_j$, and $|\xi_j| \leq \frac{\pi}{h}$, $j = 1, 2, 3$. Note that each $\ind_{[0, T)} u_j$ has temporal frequency $\tau_j - \frac{2}{h^3} \sin (\frac{h \xi_j}{2})$ and we have
\begin{align}
\begin{split}
\max_{j = 1, 2, 3} \big| \tau_j - \tfrac{2}{h^3} \sin (\tfrac{h \xi_j}{2}) \big| 
&\ges \big| ( \tau_1 - \tfrac{2}{h^3} \sin (\tfrac{h \xi_1}{2}) ) + ( \tau_2 - \tfrac{2}{h^3} \sin (\tfrac{h \xi_2}{2}) ) + ( \tau_3 - \tfrac{2}{h^3} \sin (\tfrac{h \xi_3}{2}) ) \big| \\
&= \tfrac{8}{h^3} \big| \sin (\tfrac{h \xi_1}{4}) \sin (\tfrac{h \xi_2}{4}) \sin (\tfrac{h \xi_3}{4}) \big| \\
&\sim N_1 N_2 N_3 .
\end{split}
\label{max}
\end{align} 

\noi
Thus, by letting $M \in 2^\Z$ be such that $M \sim c N_1 N_2 N_3$ for some small $c > 0$, we have
\begin{align*}
\int_\R \int_\R \Big( \prod_{j = 1, 2}  e^{\frac{t}{h^2} \nb_h} \Pb_{N_j} P_{\leq \frac{\pi}{h}} \Qb_{\leq M} \ind_{[0, T)} u_j \Big)  e^{\frac{t}{h^2} \nb_h} \Pb_{N_3} P_{\leq \frac{\pi}{h}} \Qb_{\leq M} \ind_{[0, T)} u_3  dx dt = 0 ,
\end{align*}

\noi
so that to get a non-trivial contribution, at least one of $\ind_{[0, T)} u_j$'s must be equipped with the projector $\Qb_{> M}$.

Let $\Qb^M = \Qb_{\leq M}$ or $\Qb_{> M}$. If $\Qb_{> M}$ hits $\ind_{[0, T)} u_3$, then by using H\"older's inequality, the fact that $|\frac{2}{h} \sin (\frac{h \xi}{2})| \leq |\xi|$, the Strichartz estimate in Corollary~\ref{COR:linU}~(i), Lemma~\ref{LEM:QM}~(ii), Lemma~\ref{LEM:QLpt}, Lemma~\ref{LEM:QM}~(i) with $M \sim N_1 N_2 N_3$,  \eqref{timeUV}, H\"older's inequality in $t$, and Lemma~\ref{LEM:UVemb}, we get
\begin{align*}
\bigg| &\int_\R \int_\R \Big( \prod_{j = 1, 2} e^{\frac{t}{h^2} \nb_h} \Pb_{N_j} P_{\leq \frac{\pi}{h}} \Qb^M \ind_{[0, T)} u_j \Big)  e^{\frac{t}{h^2} \nb_h} \Pb_{N_3} P_{\leq \frac{\pi}{h}} \Qb_{> M} \ind_{[0, T)} u_3  dx dt \bigg| \\
&\leq  \big\| e^{\frac{t}{h^2} \nb_h} \Pb_{N_1} P_{\leq \frac{\pi}{h}} \Qb^M \ind_{[0, T)} u_1 \big\|_{L_t^4 L_x^\infty} \big\| e^{\frac{t}{h^2} \nb_h} \Pb_{N_2} P_{\leq \frac{\pi}{h}} \Qb^M \ind_{[0, T)} u_2 \big\|_{L_t^4 L_x^2} \\
&\qquad \times  \big\| \Pb_{N_3} P_{\leq \frac{\pi}{h}} \Qb_{> M} \ind_{[0, T)} u_3 \big\|_{L^2_t L^2_x} \\
&\les N_1^{- \frac 14} \| \ind_{[0, T)} \Pb_{N_1} u_1 \|_{U^4 L^2} \| \ind_{[0, T)} \Pb_{N_2} u_2 \|_{L_t^4 L_x^2} N_1^{- \frac 12} N_2^{- \frac 12} N_3^{- \frac 12} \| \ind_{[0, T)} \Pb_{N_3} u_3 \|_{V^2 L^2} \\
&\les T^{\frac 14} N_1^{- \frac 74} \| \Pb_{N_1} u_1 \|_{V^2 L^2} \| \Pb_{N_2} u_2 \|_{V^2 L^2} \| \Pb_{N_3} u_3 \|_{V^2 L^2} ,
\end{align*}

\noi
as desired. The cases when $\Qb_{> M}$ hits $\ind_{[0, T)} u_1$ or $\ind_{[0, T)} u_2$ follow from analogous ways, and so we omit details.

\medskip \noi
(ii) In this case, we work with the right-hand side of \eqref{indT}. We again exploit the maximum modulation as in part (i). By letting $M \in 2^\Z$ be such that $M \sim c N_1 N_2 N_3$ for some small $c > 0$, we use the same reasoning to deduce that at least one of 
$\ind_{[0, T)} u_j$'s must be equipped with the projector $\Qb_{> M}$. Let $\Qb^M = \Qb_{\leq M}$ or $\Qb_{> M}$.

If $\Qb_{> M}$ hits $\ind_{[0, T)} u_3$, then by using H\"older's inequality, the bilinear estimate in Corollary~\ref{COR:bilinU}~(ii), Lemma~\ref{LEM:QM}~(i) with $M \sim N_1 N_2 N_3$, Lemma~\ref{LEM:QM}~(ii), \eqref{timeUV}, and Lemma~\ref{LEM:UVemb}, we get
\begin{align}
\begin{split}
&\bigg| \int_\R \int_\R \Big( \prod_{j = 1, 2} e^{\frac{t}{h^2} \nb_h} \Pb_{N_j} P_{\leq \frac{\pi}{h}} \Qb^M \ind_{[0, T)} u_j \Big) e^{\frac{t}{h^2} \nb_h} \Pb_{N_3} P_{\leq \frac{\pi}{h}} \Qb_{> M} \ind_{[0, T)} u_3   dx dt \bigg| \\
&\les \sum_{N_3' = \frac{N_3}{2}, N_3, 2 N_3} \bigg\| \Pb_{N_3'} \Big( \prod_{j = 1, 2} e^{\frac{t}{h^2} \nb_h} \Pb_{N_j} P_{\leq \frac{\pi}{h}} \Qb^M \ind_{[0, T)} u_j \Big) \bigg\|_{L_t^2 L_x^2}  \\
&\qquad \times \big\| \Pb_{N_3} P_{\leq \frac{\pi}{h}} \Qb_{> M} \ind_{[0, T)} u_3 \big\|_{L_t^2 L_x^2} \\
&\les N_1^{- \frac 12} N_3^{- \frac 12} \Big( \prod_{j = 1, 2} \| \Qb^M \ind_{[0, T)} \Pb_{N_j} u_j \|_{U^2 L^2} \Big) N_1^{- \frac 12} N_2^{- \frac 12} N_3^{- \frac 12} \| \ind_{[0, T)} \Pb_{N_3} u_3 \|_{V^2 L^2} \\
&\sim N_1^{- \frac 32} N_3^{-1} \| \Pb_{N_1} u_1 \|_{U^2 L^2} \| \Pb_{N_2} u_2 \|_{U^2 L^2} \| \Pb_{N_3} u_3 \|_{U^2 L^2} .
\end{split}
\label{tri3-0}
\end{align}

\noi
To obtain the $T^\ta$ factor, we simply interpolate \eqref{tri3-0} with \eqref{crude} (after using the embeddings in Lemma~\ref{LEM:UVemb}). We then interpolate the resulting estimate again with \eqref{crude} using Lemma~\ref{LEM:interp} to get the desired estimates with $V^2 L^2$-norms.

If $\Qb_{> M}$ hits $\ind_{[0, T)} u_2$, then using similar steps as above but with  the bilinear estimate in Corollary~\ref{COR:bilinU}~(i) on the product of $u_1$ and $u_3$, Lemma~\ref{LEM:QM}~(i) on $u_2$, Lemma~\ref{LEM:UVemb}, and Lemma~\ref{LEM:PNl2}~(ii), we get
\begin{align*}
&\bigg| \int_\R \int_\R \big( e^{\frac{t}{h^2} \nb_h} \Pb_{N_1} P_{\leq \frac{\pi}{h}} \Qb^M \ind_{[0, T)} u_1 \big) \big( e^{\frac{t}{h^2} \nb_h} \Pb_{N_2} P_{\leq \frac{\pi}{h}} \Qb_{> M} \ind_{[0, T)} u_2 \big)  \\
&\qquad \times \big( e^{\frac{t}{h^2} \nb_h} \Pb_{N_3} P_{\leq \frac{\pi}{h}} \Qb^M \ind_{[0, T)} u_3 \big) dx dt \bigg| \\
&\les N_1^{-2} N_3^{- \frac 12} \| \Pb_{N_1} u_1 \|_{U^2 L^2} \| \Pb_{N_2} u_2 \|_{U^2 L^2} \| \Pb_{N_3} u_3 \|_{U^2 L^2} .
\end{align*}

\noi
As above, by interpolating it with \eqref{crude} (after using the embeddings in Lemma~\ref{LEM:UVemb}) and then interpolating the resulting estimate again with \eqref{crude} using Lemma~\ref{LEM:interp}, we get the desired estimate.
The case when $\Qb_{>M}$ hits $\ind_{[0, T)} u_1$ follows from a similar manner, and so we omit details.

\medskip \noi
(iii) In this case, we make the dyadic decomposition $\Pb_{\les 1} = \sum_{N_3 \les 1} \Pb_{N_3}$. We again work with the right-hand side of \eqref{indT} and exploit the maximum modulation as in part (i) and (ii), but the computations require more care. 
By letting $M \in 2^\Z$ be such that $M \sim c N_1 N_2 N_3$ for some small $c > 0$, we use the same reasoning to deduce that at least one of $\ind_{[0, T)} u_j$'s must be equipped with the projector $\Qb_{> M}$. Let $\Qb^M = \Qb_{\leq M}$ or $\Qb_{> M}$.

If $\Qb_{> M}$ hits $\ind_{[0, T)} u_3$, we  interpolate \eqref{tri3-0} and \eqref{crude} (after using the embeddings in Lemma~\ref{LEM:UVemb}) to obtain
\begin{align}
\begin{split}
&\bigg| \int_\R \int_\R \Big( \prod_{j = 1, 2}  e^{\frac{t}{h^2} \nb_h} \Pb_{N_j} P_{\leq \frac{\pi}{h}} \Qb^M \ind_{[0, T)} u_j \Big)  \nb_h e^{\frac{t}{h^2} \nb_h} \Pb_{N_3} \Qb_{> M} \ind_{[0, T)} u_3  dx dt \bigg| \\
&\les T^\ta N_1^{- \frac 32 + 10 \ta} N_3^\ta \| \Pb_{N_1} u_1 \|_{U^2 L^2} \| \Pb_{N_2} u_2 \|_{U^2 L^2} \| \Pb_{N_3} u_3 \|_{V^2 L^2}
\end{split}
\label{help}
\end{align}

\noi
for any $\ta > 0$ arbitrarily small, so that using the Cauchy-Schwarz inequality in $N_3$, Lemma~\ref{LEM:UVemb}, and Lemma~\ref{LEM:PNl2}~(i), we obtain
\begin{align}
\begin{split}
\bigg| &\sum_{N_3 \les 1} \int_\R \int_\R \Big( \prod_{j = 1, 2}  e^{\frac{t}{h^2} \nb_h} \Pb_{N_j} P_{\leq \frac{\pi}{h}} \Qb^M \ind_{[0, T)} u_j \Big)  \nb_h e^{\frac{t}{h^2} \nb_h} \Pb_{N_3} \Qb_{> M} \ind_{[0, T)} u_3  dx dt \bigg| \\
&\les T^\ta N_1^{- \frac 32 + 10 \ta} \| \Pb_{N_1} u_1 \|_{U^2 L^2} \| \Pb_{N_2} u_2 \|_{U^2 L^2} \bigg( \sum_{N_3 \les 1} \| \Pb_{N_3} u_3 \|_{U^2 L^2}^2 \bigg)^{\frac 12} \\
&\les T^\ta N_1^{- \frac 32 + 10 \ta} \| \Pb_{N_1} u_1 \|_{U^2 L^2} \| \Pb_{N_2} u_2 \|_{U^2 L^2} \| \Pb_{\les 1} u_3 \|_{U^2 L^2} .
\end{split}
\label{tri4-0-1}
\end{align}

\noi
We again use a crude estimate as in \eqref{crude} but with a summation in $N_3 \les 1$ to obtain
\begin{align}
\begin{split}
\bigg| &\sum_{N_3 \les 1} \int_\R \int_\R \Big( \prod_{j = 1, 2}  e^{\frac{t}{h^2} \nb_h} \Pb_{N_j} P_{\leq \frac{\pi}{h}} \Qb^M \ind_{[0, T)} u_j \Big)  \nb_h e^{\frac{t}{h^2} \nb_h} \Pb_{N_3} \Qb_{> M} \ind_{[0, T)} u_3  dx dt \bigg| \\
&\les T N_1^{\frac 13} \| \Pb_{N_1} u_1 \|_{U^4 L^2}\| \Pb_{N_2} u_2 \|_{U^4 L^2} \| \Pb_{\les 1} u_3 \|_{U^4 L^2} ,
\end{split}
\label{tri4-0-2}
\end{align}

\noi
so that by \eqref{tri4-0-1}, \eqref{tri4-0-2}, and the interpolation in Lemma~\ref{LEM:interp}, we obtain
\begin{align*}
\bigg| &\sum_{N_3 \les 1} \int_\R \int_\R \Big( \prod_{j = 1, 2}  e^{\frac{t}{h^2} \nb_h} \Pb_{N_j} P_{\leq \frac{\pi}{h}} \Qb^M \ind_{[0, T)} u_j \Big)  \nb_h e^{\frac{t}{h^2} \nb_h} \Pb_{N_3} \Qb_{> M} \ind_{[0, T)} u_3  dx dt \bigg| \\
&\les T^\ta N_1^{- \frac 32 + 20 \ta} \| \Pb_{N_1} u_1 \|_{V^2 L^2} \| \Pb_{N_2} u_2 \|_{V^2 L^2} \| \Pb_{\les 1} u_3 \|_{V^2 L^2} ,
\end{align*}

\noi
as desired.

If $\Qb_{>M}$ hits $\ind_{[0, T)} u_2$, we follow the same reasoning as in part (ii) and use the Cauchy-Schwarz inequality in $N_3$, Lemma~\ref{LEM:UVemb}, and Lemma~\ref{LEM:PNl2}~(i) to obtain
\begin{align}
\begin{split}
\bigg| &\sum_{N_3 \les 1} \int_\R \int_\R \big( e^{\frac{t}{h^2} \nb_h} \Pb_{N_1} P_{\leq \frac{\pi}{h}} \Qb^M \ind_{[0, T)} u_1 \big) \big( e^{\frac{t}{h^2} \nb_h} \Pb_{N_2} P_{\leq \frac{\pi}{h}} \Qb_{> M} \ind_{[0, T)} u_2 \big)  \\
&\qquad \times \big( \nb_h e^{\frac{t}{h^2} \nb_h} \Pb_{N_3} \Qb^M \ind_{[0, T)} u_3 \big) dx dt \bigg| \\
&\les T^\ta N_1^{-2 + 10 \ta}  \| \Pb_{N_1} u_1 \|_{U^2 L^2} \| \Pb_{N_2} u_2 \|_{U^2 L^2} \sum_{N_3 \les 1} N_3^{\frac 12} \| \Pb_{N_3} u_3 \|_{U^2 L^2} \\
&\les T^\ta N_1^{-2 + 10 \ta} \| \Pb_{N_1} u_1 \|_{U^2 L^2} \| \Pb_{N_2} u_2 \|_{U^2 L^2} \| \Pb_{\les 1} u_3 \|_{U^2 L^2} .
\end{split}
\label{tri4-1}
\end{align}

\noi
for any small enough $\ta > 0$. 
Also, a crude estimate using H\"older's inequality, Sobolev's inequality, and Lemma~\ref{LEM:UVemb} yields
\begin{align}
\begin{split}
\bigg| &\sum_{N_3 \les 1} \int_\R \int_\R \big( e^{\frac{t}{h^2} \nb_h} \Pb_{N_1} P_{\leq \frac{\pi}{h}} \Qb^M \ind_{[0, T)} u_1 \big) \big( e^{\frac{t}{h^2} \nb_h} \Pb_{N_2} P_{\leq \frac{\pi}{h}} \Qb_{> M} \ind_{[0, T)} u_2 \big) \\
&\qquad \times \big( \nb_h e^{\frac{t}{h^2} \nb_h} \Pb_{N_3} \Qb^M \ind_{[0, T)} u_3 \big) dx dt \bigg| \\
&\les T N_1^{\frac 13} \| \Pb_{N_1} u_1 \|_{L_t^\infty L_x^2} \| \Pb_{N_2} u_2 \|_{L_t^\infty L_x^2} \sum_{N_3 \les 1} N_3 \| \Pb_{\les 1} u_3 \|_{L_t^\infty L_x^2} \\
&\les T N_1^{\frac 13} \| \Pb_{N_1} u_1 \|_{U^4 L^2} \| \Pb_{N_2} u_2 \|_{U^4 L^2} \| \Pb_{\les 1} u_3 \|_{U^4 L^2} .
\end{split}
\label{tri4-2}
\end{align}

\noi
Thus, by \eqref{tri4-1}, \eqref{tri4-2}, and the interpolation in Lemma~\ref{LEM:interp}, we get
\begin{align*}
\bigg| &\sum_{N_3 \les 1} \int_\R \int_\R \big( e^{\frac{t}{h^2} \nb_h} \Pb_{N_1} P_{\leq \frac{\pi}{h}} \Qb^M \ind_{[0, T)} u_1 \big) \big( e^{\frac{t}{h^2} \nb_h} \Pb_{N_2} P_{\leq \frac{\pi}{h}} \Qb_{> M} \ind_{[0, T)} u_2 \big)  \\
&\qquad \times \big( \nb_h e^{\frac{t}{h^2} \nb_h} \Pb_{N_3} \Qb^M \ind_{[0, T)} u_3 \big) dx dt \bigg| \\
&\les T^\ta N_1^{-2 + 20 \ta} \| \Pb_{N_1} u_1 \|_{V^2 L^2} \| \Pb_{N_2} u_2 \|_{V^2 L^2} \| \Pb_{\les 1} u_3 \|_{V^2 L^2} ,
\end{align*}

\noi
as desired.
The case when $\Qb_{>M}$ hits $\ind_{[0, T)} u_1$ follows from a similar manner.

\medskip \noi
(iv) In this case, we work with the left-hand side of \eqref{indT}. We first note that the spatial frequencies $\xi_1$, $\xi_2$, and $\xi_3$ for $u_1$, $u_2$, and $u_3$, respectively, satisfy $\xi_1 + \xi_2 - \xi_3 = 0$. Thus, we can replace the Fourier multipliers $e^{\frac{t}{h^2} \nb_h}$ by $e^{\frac{t}{h^2} (\nb_h - \dx)}$.
By using H\"older's inequality, the local smoothing estimate in Corollary~\ref{COR:linU}~(iii), the maximal function estimate in Corollary~\ref{COR:linU}~(ii), Lemma~\ref{LEM:UVemb}, and Lemma~\ref{LEM:PNl2}~(ii), we get
\begin{align*}
&\bigg| \int_0^T \int_\R \big( e^{\frac{t}{h^2} (\nb_h - \dx)} \Pb_{N_1} P_{\leq \frac{\pi}{h}} u_1 \big) \big( e^{\frac{t}{h^2} (\nb_h - \dx)} \Pb_{\les 1} u_2 \big) \big( e^{\frac{t}{h^2} (\nb_h - \dx)} \Pb_{N_3} P_{\leq \frac{\pi}{h}} u_3 \big) dx dt \bigg| \\
&\les  \big\| e^{\frac{t}{h^2} (\nb_h - \dx)} \Pb_{N_1} P_{\leq \frac{\pi}{h}} u_1 \big\|_{L_x^\infty L_t^2} \big\| e^{\frac{t}{h^2} (\nb_h - \dx)} \Pb_{\les 1} u_2 \big\|_{L_x^2 L_T^\infty} T^{\frac 12} \| \Pb_{N_3} u_3 \|_{L_t^\infty L_x^2}  \\
&\les T^{\frac 12}  \| \nb_h^{-1} \Pb_{N_1} P_{\leq \frac{\pi}{h}} u_1 \|_{U^2 L^2} (1 + T) \| \Pb_{\les 1} u_2 \|_{U^2 L^2} \| \Pb_{N_3} u_3 \|_{V^2 L^2} \\
&\les T^{\frac 12} (1 + T) N_1^{-1} \| \Pb_{N_1} u_1 \|_{U^2 L^2} \| \Pb_{\les 1} u_2 \|_{U^2 L^2} \| \Pb_{N_3} u_3 \|_{V^2 L^2} ,
\end{align*}

\noi
as desired.
\end{proof}

We now use Lemma~\ref{LEM:tri_gen} to establish the following trilinear estimates. From then on, for convenience, we may abuse notations by writing $\Pb_1 = \Pb_{\leq 1}$.

\begin{lemma}
\label{LEM:tri}
Let $0 < h \leq 1$ and $0 < T \leq 1$.
Let $s_1, s_2 \in \R$ be such that $s_1 \leq 0$, $s_2 \geq - \frac 34$, $s_1 \leq s_2$, and $s_1 + s_2 > - \frac 32$. Then, there exists $\ta > 0$ such that
\begin{align}
\begin{split}
\bigg| &\int_0^T \int_\R \big( e^{\pm \frac{t}{h^2} \nb_h} P_{\leq \frac{\pi}{h}} u_1 \big) \big( e^{\pm \frac{t}{h^2} \nb_h} P_{\leq \frac{\pi}{h}} u_2 \big) \big( \nb_h e^{\pm \frac{t}{h^2} \nb_h} P_{\leq \frac{\pi}{h}} u_3 \big) dx dt \bigg| \\
&\les T^\ta \| u_1 \|_{X_T^{s_1}} \| u_2 \|_{X_T^{s_2}} \| u_3 \|_{Y_T^{- s_1}} 
\end{split}
\label{tri_goal}
\end{align}

\noi
and
\begin{align}
\begin{split}
\bigg| \int_0^T \int_\R (e^{\pm \frac{t}{24} \dx^3}  u_1) (e^{\pm \frac{t}{24} \dx^3} u_2) (\dx e^{\pm \frac{t}{24} \dx^3} u_3) dx dt \bigg| 
\les T^\ta \| u_1 \|_{X_T^{s_1}} \| u_2 \|_{X_T^{s_2}} \| u_3 \|_{Y_T^{- s_1}} .
\end{split}
\label{tri_goal2}
\end{align}
\end{lemma}

\begin{proof}
We only focus on the proof of \eqref{tri_goal}, as the proof of \eqref{tri_goal2} follows from a similar manner. 
Also, without loss of generality, we only prove the estimate with the ``$+$'' signs.
By working on the extensions of $u_1$, $u_2$, and $u_3$ beyond the time interval $[0, T)$ and using the definition of the local-in-time norm \eqref{Bloc}, we may ignore the subscripts $T$ on the right-hand side of \eqref{tri_goal}. 
By using the Littlewood-Paley decomposition \eqref{PNdecomp}, we write
\begin{align}
\begin{split}
&\text{LHS of \eqref{tri_goal}} \\
&= \bigg| \sum_{N_1, N_2, N_3 \geq 1} \int_0^T \int_\R  \bigg( \prod_{j = 1, 2}  e^{\frac{t}{h^2} \nb_h} \Pb_{N_j} P_{\leq \frac{\pi}{h}}  u_j \bigg) \nb_h e^{\frac{t}{h^2} \nb_h} \Pb_{N_3} P_{\leq \frac{\pi}{h}} u_3  dx dt \bigg| .
\end{split}
\label{tri0}
\end{align}

\noi
Since the largest two frequencies must be comparable, we consider the following cases.

\smallskip \noi
{\bf Case 1:} $N_1, N_2, N_3 \sim 1$.

In this case, by using H\"older's inequality, Sobolev's inequalities, and Lemma~\ref{LEM:UVemb}, we get the following contribution for \eqref{tri0}:
\begin{align*}
\bigg| &\int_0^T \int_\R \Big( \prod_{j = 1, 2}  e^{\frac{t}{h^2} \nb_h} \Pb_{\les 1} P_{\leq \frac{\pi}{h}} u_j \Big) \nb_h e^{\frac{t}{h^2} \nb_h} \Pb_{\les 1} P_{\leq \frac{\pi}{h}} u_3  dx dt \bigg| \\
&\les T  \| \Pb_{\les 1} u_1 \|_{L_T^\infty L_x^2} \| \Pb_{\les 1} u_2 \|_{L_T^\infty L_x^2} \| \Pb_{\les 1} u_3 \|_{L_T^\infty L_x^2} \\ 
&\les T \| \Pb_{\les 1} u_1 \|_{U^2 L^2} \|  \Pb_{\les 1} u_2 \|_{U^2 L^2} \| \Pb_{\les 1} u_3 \|_{V^2 L^2} ,
\end{align*}

\noi
which gives the desired bound for \eqref{tri_goal}.

\smallskip \noi
{\bf Case 2:} $N_1 \sim N_2 \sim N_3 \gg 1$.

In this case, we use Lemma~\ref{LEM:tri_gen}~(i), the fact that $\frac{2}{h} |\sin (\frac{h \xi}{2})| \leq |\xi|$, and Lemma~\ref{LEM:UVemb} to get the following contribution for \eqref{tri0}:
\begin{align*}
\bigg| &\sum_{N_1 \sim N_2 \sim N_3 \gg 1} \int_0^T \int_\R \Big( \prod_{j = 1, 2} e^{\frac{t}{h^2} \nb_h} \Pb_{N_j} P_{\leq \frac{\pi}{h}} u_j \Big) \nb_h e^{\frac{t}{h^2} \nb_h} \Pb_{N_3} P_{\leq \frac{\pi}{h}} u_3  dx dt \bigg| \\
&\les T^{\frac 14} \sum_{N_1 \sim N_2 \sim N_3 \gg 1} N_1^{- \frac 34 - s_2}  N_1^{s_1} \| \Pb_{N_1} u_1 \|_{U^2 L^2} N_2^{s_2} \| \Pb_{N_2} u_2 \|_{U^2 L^2} N_3^{- s_1} \| \Pb_{N_3} u_3 \|_{V^2 L^2} .
\end{align*}

\noi
Since $s_2 \geq -\frac 34$, we use the Cauchy-Schwarz inequality in $N_1 \sim N_2 \sim N_3$ to obtain the desired bound for \eqref{tri_goal}.

\smallskip \noi
{\bf Case 3:} $N_1 \sim N_2 \gg N_3 \gg 1$.

In this case, we use Lemma~\ref{LEM:tri_gen}~(ii), the fact that $\frac{2}{h} |\sin (\frac{h \xi}{2})| \leq |\xi|$, and Lemma~\ref{LEM:UVemb} to get the following contribution for \eqref{tri0}:
\begin{align*}
\bigg| &\sum_{N_1 \sim N_2 \gg N_3 \gg 1} \int_0^T \int_\R \Big( \prod_{j = 1, 2}  e^{\frac{t}{h^2} \nb_h} \Pb_{N_j} P_{\leq \frac{\pi}{h}} u_j \Big) \nb_h e^{\frac{t}{h^2} \nb_h} \Pb_{N_3} P_{\leq \frac{\pi}{h}} u_3  dx dt \bigg| \\
&\les T^{\ta} \sum_{N_1 \sim N_2 \gg N_3 \gg 1} N_1^{- \frac 32 - s_1 - s_2 + 10 \ta} N_3^{s_1}  N_1^{s_1} \| \Pb_{N_1} u_1 \|_{U^2 L^2} N_2^{s_2} \| \Pb_{N_2} u_2 \|_{U^2 L^2} N_3^{- s_1} \| \Pb_{N_3} u_3 \|_{V^2 L^2} 
\end{align*}

\noi
for any $\ta > 0$ arbitrarily small.
Since $s_1 \leq 0$ and $s_1 + s_2 > -\frac 32$, we use the Cauchy-Schwarz inequalities first in $N_3$ and then in $N_1 \sim N_2$ to obtain the estimate \eqref{tri_goal}.

\smallskip \noi
{\bf Case 4:} $N_1 \sim N_2 \gg N_3 \sim 1$.

In this case, we use Lemma~\ref{LEM:tri_gen}~(iii) and Lemma~\ref{LEM:UVemb} to get the following contribution for \eqref{tri0}:
\begin{align*}
\bigg| &\sum_{N_1 \sim N_2 \gg 1} \int_0^T \int_\R \Big( \prod_{j = 1, 2}  e^{\frac{t}{h^2} \nb_h} \Pb_{N_j} P_{\leq \frac{\pi}{h}} u_j \Big)  \nb_h e^{\frac{t}{h^2} \nb_h} \Pb_{\les 1} P_{\leq \frac{\pi}{h}} u_3  dx dt \bigg| \\
&\les T^{\ta} \sum_{N_1 \sim N_2 \gg 1} N_1^{- \frac 32 - s_1 - s_2 + 20 \ta}  N_1^{s_1} \| \Pb_{N_1} u_1 \|_{U^2 L^2} N_2^{s_2} \| \Pb_{N_2} u_2 \|_{U^2 L^2} \| \Pb_{\les 1} u_3 \|_{V^2 L^2} 
\end{align*}

\noi
for any $\ta > 0$ arbitrarily small.
Since $s_1 + s_2 > -\frac 32$, we use the Cauchy-Schwarz inequality in $N_1 \sim N_2$ to obtain the estimate \eqref{tri_goal}.

\smallskip \noi
{\bf Case 5:} $N_1 \sim N_3 \gg N_2 \gg 1$.

In this case, we use Lemma~\ref{LEM:tri_gen}~(ii), the fact that $\frac{2}{h} |\sin (\frac{h \xi}{2})| \leq |\xi|$, and Lemma~\ref{LEM:UVemb} to get the following contribution for \eqref{tri0}:
\begin{align*}
\bigg| &\sum_{N_1 \sim N_3 \gg N_2 \gg 1}  \int_0^T \int_\R \Big( \prod_{j = 1, 2}  e^{\frac{t}{h^2} \nb_h} \Pb_{N_j} P_{\leq \frac{\pi}{h}} u_j \Big) \nb_h e^{\frac{t}{h^2} \nb_h} \Pb_{N_3} P_{\leq \frac{\pi}{h}} u_3  dx dt \bigg| \\
&\les T^{\ta} \sum_{N_1 \sim N_3 \gg N_2 \gg 1} N_1^{- \frac 12 + 10 \ta} N_2^{-1 - s_2} N_1^{s_1} \| \Pb_{N_1} u_1 \|_{U^2 L^2} N_2^{s_2} \| \Pb_{N_2} u_2 \|_{U^2 L^2} N_3^{- s_1} \| \Pb_{N_3} u_3 \|_{V^2 L^2} 
\end{align*}

\noi
for any $\ta > 0$ arbitrarily small.
Since $-1 - s_2 < 0$ given $s_2 \geq - \frac 34$, we use the Cauchy-Schwarz inequalities first in $N_2$ and then in $N_1 \sim N_3$ to obtain the estimate \eqref{tri_goal}.

\smallskip \noi
{\bf Case 6:} $N_1 \sim N_3 \gg N_2 \sim 1$.

In this case, we use Lemma~\ref{LEM:tri_gen}~(iv) and the fact that $\frac{2}{h} |\sin (\frac{h \xi}{2})| \leq |\xi|$ to get the following contribution for \eqref{tri0}:
\begin{align*}
\bigg| &\sum_{N_1 \sim N_3 \gg 1} \int_0^T \int_\R \big( e^{\frac{t}{h^2} \nb_h} \Pb_{N_1} P_{\leq \frac{\pi}{h}} u_1 \big) \big( e^{\frac{t}{h^2} \nb_h} \Pb_{\les 1} P_{\leq \frac{\pi}{h}} u_2 \big) \big( \nb_h e^{\frac{t}{h^2} \nb_h} \Pb_{N_3} P_{\leq \frac{\pi}{h}} u_3 \big)  dx dt \bigg| \\
&\les T^{\frac 12} \sum_{N_1 \sim N_3}   N_1^{s_1} \| \Pb_{N_1} u_1 \|_{U^2 L^2} \| \Pb_{\les 1} u_2 \|_{U^2 L^2} N_3^{- s_1} \| \Pb_{N_3} u_3 \|_{V^2 L^2} .
\end{align*}

\noi
We then use the Cauchy-Schwarz inequality in $N_1 \sim N_3$ to obtain the estimate \eqref{tri_goal}.

\smallskip \noi
{\bf Case 7:} $N_2 \sim N_3 \gg N_1 \gg 1$.

In this case, we use Lemma~\ref{LEM:tri_gen}~(ii), the fact that $\frac{2}{h} |\sin (\frac{h \xi}{2})| \leq |\xi|$, and Lemma~\ref{LEM:UVemb} to get the following contribution for \eqref{tri0}:
\begin{align*}
&\bigg| \sum_{N_2 \sim N_3 \gg N_1 \gg 1}  \int_0^T \int_\R \Big( \prod_{j = 1, 2}  e^{\frac{t}{h^2} \nb_h} \Pb_{N_j} P_{\leq \frac{\pi}{h}} u_j \Big) \nb_h e^{\frac{t}{h^2} \nb_h} \Pb_{N_3} P_{\leq \frac{\pi}{h}} u_3  dx dt \bigg| \\
&\les T^{\ta} \sum_{N_2 \sim N_3 \gg N_1 \gg 1} N_1^{- 1 - s_1} N_2^{-\frac 12 + s_1 - s_2 + 10 \ta} N_1^{s_1} \| \Pb_{N_1} u_1 \|_{U^2 L^2} N_2^{s_2} \| \Pb_{N_2} u_2 \|_{U^2 L^2} N_3^{- s_1} \| \Pb_{N_3} u_3 \|_{V^2 L^2} 
\end{align*}

\noi
for any $\ta > 0$ arbitrarily small. If $s_1 > -1$, we need $- \frac 12 + s_1 - s_2 < 0$ to sum up the frequency scales, which is guaranteed by the condition $s_1 \leq s_2$. If $s_1 \leq - 1$, we need $- \frac 32 - s_2 < 0$, which is guaranteed by the condition $s_2 \geq - \frac 34$. In either case, we obtain the estimate \eqref{tri_goal}.

\smallskip \noi
{\bf Case 8:} $N_2 \sim N_3 \gg N_1 \sim 1$.

In this case, we use Lemma~\ref{LEM:tri_gen}~(iv) and the fact that $\frac{2}{h} |\sin (\frac{h \xi}{2})| \leq |\xi|$ to get the following contribution for \eqref{tri0}:
\begin{align*}
\bigg| &\sum_{N_2 \sim N_3 \gg 1} \int_0^T \int_\R \big( e^{\frac{t}{h^2} \nb_h} \Pb_{\les 1} P_{\leq \frac{\pi}{h}} u_1 \big) \big( e^{\frac{t}{h^2} \nb_h} \Pb_{N_2} P_{\leq \frac{\pi}{h}} u_2 \big) \big( \nb_h e^{\frac{t}{h^2} \nb_h} \Pb_{N_3} P_{\leq \frac{\pi}{h}} u_3  \big) dx dt \bigg| \\
&\les T^{\frac 12} \sum_{N_2 \sim N_3} N_2^{s_1 - s_2} \| \Pb_{\les 1} u_1 \|_{U^2 L^2} N_2^{s_2} \| \Pb_{N_2} u_2 \|_{U^2 L^2} N_3^{- s_1} \| \Pb_{N_3} u_3 \|_{V^2 L^2} .
\end{align*}

\noi
Since $s_1 \leq s_2$, we use the Cauchy-Schwarz inequality in $N_2 \sim N_3$ to obtain the estimate \eqref{tri_goal}.
\end{proof}

\begin{remark} \rm
\label{RMK:LWP34}
Note that, with the choice $s_1 = s_2 = - \frac 34 + \eps$ with $\eps > 0$ arbitrarily small, the trilinear estimate \eqref{tri_goal2} for the KdV flow directly implies local well-posedness of the KdV equation in $H^s (\R)$ for $s > -\frac 34$. In view of the failure of the uniform continuity of the KdV solution map in $H^s (\R)$ for any $-1 \leq s < - \frac 34$ established in \cite{CCT}, our trilinear estimate \eqref{tri_goal2} is sharp up to the endpoint $s = - \frac 34$. 

It is possible to show \eqref{tri_goal} and \eqref{tri_goal2} (without the $T^\ta$ factor) with $s_1 = s_2 = - \frac 34$ and $0 < T \leq 1$ using our formulation. The only case that requires a non-trivial refinement is the high-high to low interaction addressed in Case 4. 
The key point is that when the high modulation hits the low frequency duality term, we do not decompose the frequency scales of the duality term.
We first split the derivative $\nb_h$ or $\dx$ and put a half derivative on $u_3$ and a half derivative on the product of $u_1$ and $u_2$. Next, we use a frequency dependent version of Lemma~\ref{LEM:QM}~(i) to gain spatial derivatives for $u_1$ and $u_2$ and cancel out the half derivative. Then, we use a slight variant of the bilinear estimate Lemma~\ref{LEM:bi_gen} to deal with the product of $u_1$ and $u_2$ (with a half derivative acting on the product), where we gain some spatial derivative for $u_1$ and again cancel out the half derivative.
Since getting the endpoint case is not our main focus in this paper, we choose not to present details.
\end{remark}

\subsection{Trilinear difference estimates}

In this subsection, we establish more trilinear estimates, which are relevant for the convergence of the FPU system to the KdV equation. The main task is to create positive powers of $h$ by losing some corresponding number of derivatives, which are essential for the convergence of solutions.
For all the estimates in this subsection, we disregard the dependence on $T$ and simply use the bound $T \leq 1$.

We first show an estimate involving the difference of the differentiation operators $\nb_h$ and $\dx$.

\begin{lemma}
\label{LEM:trid1}
Let $0 < h \leq 1$ and $0 < T \leq 1$. Let $- \frac 34 < s \leq 0$ and $0 \leq \gamma \leq 2$. Then, we have
\begin{align*}
\bigg| &\int_0^T \int_\R \big( e^{\pm \frac{t}{h^2} \nb_h} P_{\leq \frac{\pi}{h}} u_1 \big) \big( e^{\pm \frac{t}{h^2} \nb_h} P_{\leq \frac{\pi}{h}} u_2 \big) \big( (\nb_h - \dx) e^{\pm \frac{t}{h^2} \nb_h} P_{\leq \frac{\pi}{h}} u_3 \big) dx dt \bigg| \\
&\les h^\gamma \| u_1 \|_{X_T^s} \| u_2 \|_{X_T^s} \| u_3 \|_{Y_T^{- s + \gamma}} .
\end{align*}
\end{lemma}

\begin{proof}
Let us write $\nb_h - \dx = \dx (\nb_h \dx^{-1} - 1)$, where the operator $\nb_h \dx^{-1}$ makes sense as a Fourier multiplier operator with a bounded and continuous symbol $\frac{2}{h \xi} \sin (\frac{h \xi}{2})$.
From Lemma~\ref{LEM:tri} with $s_1 = s_2 = s$, we get
\begin{align}
\begin{split}
\bigg| &\int_0^T \int_\R \big( e^{\pm \frac{t}{h^2} \nb_h} P_{\leq \frac{\pi}{h}} u_1 \big) \big( e^{\pm \frac{t}{h^2} \nb_h} P_{\leq \frac{\pi}{h}} u_2 \big) \big( (\nb_h - \dx) e^{\pm \frac{t}{h^2} \nb_h} P_{\leq \frac{\pi}{h}} u_3 \big) dx dt \bigg| \\
&\les \| u_1 \|_{X_T^s} \| u_2 \|_{X_T^s} \| (\nb_h \dx^{-1} - 1) u_3 \|_{Y_T^{-s}} .
\end{split}
\label{dhdx1}
\end{align}

\noi
Since $|\frac{2}{h \xi} \sin (\frac{h \xi}{2}) - 1| \les \min (1, h^2 \xi^2) \les h^\gamma |\xi|^\gamma$, we easily get
\begin{align}
\| (\nb_h \dx^{-1} - 1) u_3 \|_{Y_T^{- s}} \les h^\gamma \| u_3 \|_{Y_T^{- s + \gamma}}.
\label{dhdx2}
\end{align}

\noi
The desired estimate follows from \eqref{dhdx1} and \eqref{dhdx2}.
\end{proof}

We now look at the difference of trilinear terms, one with the linear FPU propagators and the other one with the Airy propagators.

\begin{lemma}
\label{LEM:trid2}
Let $0 < h \leq 1$ and $0 < T \leq 1$. Let $- \frac 34 < s \leq 0$ and $0 < \gamma \leq 1$. Then, for any $0 < \ta \ll 4s + 3$, we have
\begin{align}
\begin{split}
\bigg| &\int_0^T \int_\R \big( e^{\pm \frac{t}{h^2} \nb_h} P_{\leq \frac{\pi}{h}} u_1 \big) \big( e^{\pm \frac{t}{h^2} \nb_h} P_{\leq \frac{\pi}{h}} u_2 \big) \big( \dx e^{\pm \frac{t}{h^2} \nb_h} P_{\leq \frac{\pi}{h}} u_3 \big) dx dt \\
&\quad - \int_0^T \int_\R \big( e^{\pm \frac{t}{24} \dx^3} P_{\leq \frac{\pi}{h}} u_1 \big) \big( e^{\pm \frac{t}{24} \dx^3} P_{\leq \frac{\pi}{h}} u_2 \big) \big( \dx e^{\pm \frac{t}{24} \dx^3} P_{\leq \frac{\pi}{h}} u_3 \big) dx dt \bigg| \\
&\les (h^{\frac 25 \gamma} + h^{\frac 35 + \frac 45 s - \ta}) \| u_1 \|_{X_T^{s}} \| u_2 \|_{X_T^{s}} \| u_3 \|_{Y_T^{-s + \gamma}} .
\end{split}
\label{trid2_goal}
\end{align}
\end{lemma}

\begin{proof}
The structure of the proof is similar to that of Lemma~\ref{LEM:tri} but contains some additional treatment of the difference of linear propagators using Lemma~\ref{LEM:V2diff}.
Without loss of generality, we only prove the estimate with the ``$+$'' signs. 
By working on the extensions of $u_1$, $u_2$, and $u_3$ beyond the time interval $[0, T)$ and using the definition of the local-in-time norm \eqref{Bloc}, we may ignore the subscripts $T$ on the right-hand side of \eqref{tri_goal}. 

By using the Littlewood-Paley decomposition \eqref{PNdecomp}, we write
\begin{align*}
\text{LHS of } \eqref{trid2_goal} = \bigg| \int_0^T \int_\R \sum_{N_1, N_2, N_3 \geq 1} \big( \textup{I}_1^{N_1, N_2, N_3} + \textup{I}_2^{N_1, N_2, N_3} \big) dx dt \bigg| ,
\end{align*}

\noi
where
\begin{align}
\textup{I}_1^{N_1, N_2, N_3} &:= \big( e^{\frac{t}{h^2} \nb_h} \Pb_{N_1} P_{\leq \frac{\pi}{h}} u_1 \big) \big( e^{\frac{t}{h^2} \nb_h} \Pb_{N_2} P_{\leq \frac{\pi}{h}} u_2 \big) \big( \dx e^{\frac{t}{h^2} \nb_h} \Pb_{N_3} P_{\leq \frac{\pi}{h}} u_3 \big) , \label{I1N123} \\
\textup{I}_2^{N_1, N_2, N_3} &:= - \big( e^{\frac{t}{24} \dx^3} \Pb_{N_1} P_{\leq \frac{\pi}{h}} u_1 \big) \big( e^{\frac{t}{24} \dx^3} \Pb_{N_2} P_{\leq \frac{\pi}{h}} u_2 \big) \big( \dx e^{\frac{t}{24} \dx^3} \Pb_{N_3} P_{\leq \frac{\pi}{h}} u_3 \big).  \nonumber
\end{align}

\noi
Note that we can replace all $e^{\frac{t}{h^2} \nb_h}$ by $e^{\frac{t}{h^2} (\nb_h - \dx)}$ in \eqref{I1N123}. 
Alternatively, we can write
\begin{align*}
\text{LHS of } \eqref{trid2_goal} = \bigg| \int_0^T \int_\R \sum_{N_1, N_2, N_3 \geq 1} \big( \textup{II}_1^{N_1, N_2, N_3} + \textup{II}_2^{N_1, N_2, N_3} + \textup{II}_3^{N_1, N_2, N_3} \big) dx dt \bigg|,
\end{align*}

\noi
where
\begin{align*}
\textup{II}_1^{N_1, N_2, N_3} &:= \big( e^{\frac{t}{h^2} \nb_h} \Pb_{N_1} P_{\leq \frac{\pi}{h}} (1 - e^{- \frac{t}{h^2} (\nb_h - \dx) + \frac{t}{24} \dx^3}) u_1 \big) \\
&\qquad \times \big( e^{\frac{t}{h^2} \nb_h} \Pb_{N_2} P_{\leq \frac{\pi}{h}} u_2 \big) \big( \dx e^{\frac{t}{h^2} \nb_h} \Pb_{N_3} P_{\leq \frac{\pi}{h}} u_3 \big) , \\
\textup{II}_2^{N_1, N_2, N_3} &:= \big( e^{\frac{t}{h^2} \nb_h} \Pb_{N_1} P_{\leq \frac{\pi}{h}} e^{- \frac{t}{h^2} (\nb_h - \dx) + \frac{t}{24} \dx^3} u_1 \big) \\
&\qquad \times \big( e^{\frac{t}{h^2} \nb_h} \Pb_{N_2} P_{\leq \frac{\pi}{h}} (1 - e^{- \frac{t}{h^2} (\nb_h - \dx) + \frac{t}{24} \dx^3}) u_2 \big) \big( \dx e^{\frac{t}{h^2} \nb_h} \Pb_{N_3} P_{\leq \frac{\pi}{h}} u_3 \big) , \\
\textup{II}_3^{N_1, N_2, N_3} &:= \big( e^{\frac{t}{24} \dx^3} \Pb_{N_1} P_{\leq \frac{\pi}{h}} u_1 \big) \big( e^{\frac{t}{24} \dx^3} \Pb_{N_2} P_{\leq \frac{\pi}{h}} u_2 \big)  \\ 
&\qquad \times \big( \dx e^{\frac{t}{24} \dx^3} \Pb_{N_3} P_{\leq \frac{\pi}{h}} (e^{\frac{t}{h^2} (\nb_h - \dx) - \frac{t}{24} \dx^3} - 1) u_3 \big) .
\end{align*}

\smallskip \noi
{\bf Case 1:} $N_1, N_2, N_3 \sim 1$.

In this case, by using H\"older's inequality, Sobolev's inequalities along with the fact that $N_1, N_2, N_3 \sim 1$,  Lemma~\ref{LEM:V2diff}~(i) and (ii), and Lemma~\ref{LEM:UVemb}, we get
\begin{align*}
\bigg| &\int_0^T \int_\R \sum_{N_1, N_2, N_3 \sim 1} \big( \textup{II}_1^{N_1, N_2, N_3} + \textup{II}_2^{N_1, N_2, N_3} + \textup{II}_3^{N_1, N_2, N_3} \big) dx dt \bigg| \\
&\les h (1 + h) \| \Pb_{\les 1} u_1 \|_{U^2 L^2} \| \Pb_{\les 1} u_2 \|_{U^2 L^2} \| \Pb_{\les 1} u_3 \|_{V^2 L^2} ,
\end{align*}

\noi
giving the desired bound \eqref{trid2_goal}.

\smallskip \noi
{\bf Case 2:} $N_1 \sim N_2 \sim N_3 \gg 1$.

In this case, we use Lemma~\ref{LEM:tri_gen}~(i), the fact that $\frac{2}{h} |\sin (\frac{h \xi}{2})| \leq |\xi|$, and Lemma~\ref{LEM:V2diff}~(i) and (ii) to obtain
\begin{align}
\bigg| \int_0^T \int_\R \big( \textup{I}_1^{N_1, N_2, N_3} + \textup{I}_2^{N_1, N_2, N_3} \big) dx dt \bigg| \les N_1^{- \frac 34} \| \Pb_{N_1} u_1 \|_{V^2 L^2} \| \Pb_{N_2} u_2 \|_{V^2 L^2} \| \Pb_{N_3} u_3 \|_{V^2 L^2}
\label{tridd2-1}
\end{align}

\noi
and
\begin{align}
\begin{split}
\bigg| &\int_0^T \int_\R \big( \textup{II}_1^{N_1, N_2, N_3} + \textup{II}_2^{N_1, N_2, N_3} + \textup{II}_3^{N_1, N_2, N_3} \big) dx dt \bigg| \\
&\les N_1^{- \frac 34} h N_1^{\frac 52} (1 + h N_1^{\frac 52}) \| \Pb_{N_1} u_1 \|_{V^2 L^2} \| \Pb_{N_2} u_2 \|_{V^2 L^2} \| \Pb_{N_3} u_3 \|_{V^2 L^2} .
\end{split}
\label{tridd2-2}
\end{align}

\noi
By interpolating \eqref{tridd2-1} and \eqref{tridd2-2} along with the fact that $0 < \frac{3}{10} + \frac 25 s + \frac 25 \gamma < 1$ given $- \frac 34 < s \leq 0$ and $0 < \gamma \leq 1$ and using Lemma~\ref{LEM:UVemb}, we get the following contribution in this case:
\begin{align*}
h^{\frac{3}{10} + \frac 25 s + \frac 25 \gamma} N_1^{s} N_2^{s} N_3^{-s + \gamma} \| \Pb_{N_1} u_1 \|_{U^2 L^2}  \| \Pb_{N_2} u_2 \|_{U^2 L^2}  \| \Pb_{N_3} u_3 \|_{V^2 L^2}.
\end{align*}

\noi
Here, for the interpolation, one needs to separately discuss the case when $h N_1^{\frac 52} \les 1$ and the case when $h N_1^{\frac 52} \ges 1$, but one can obtain the same resulting factor in both cases (we omit mentioning this point in later cases).
We can then sum up the dyadic scales using the Cauchy-Schwarz inequality in $N_1 \sim N_2 \sim N_3$ to obtain the desired bound \eqref{trid2_goal}.

\smallskip \noi
{\bf Case 3:} $N_1 \sim N_2 \gg N_3 \gg 1$.

In this case, we use Lemma~\ref{LEM:tri_gen}~(ii), the fact that $\frac{2}{h} |\sin (\frac{h \xi}{2})| \leq |\xi|$, and Lemma~\ref{LEM:V2diff}~(i) and (ii) to obtain
\begin{align}
\begin{split}
\bigg| &\int_0^T  \int_\R \big(\textup{I}_1^{N_1, N_2, N_3} + \textup{I}_2^{N_1, N_2, N_3}\big) dx dt \bigg| \\
&\les N_1^{- \frac 32 + 10 \ta} \| \Pb_{N_1} u_1 \|_{V^2 L^2} \| \Pb_{N_2} u_2 \|_{V^2 L^2}  \| \Pb_{N_3} u_3 \|_{V^2 L^2}
\end{split}
\label{tridd3-1}
\end{align}

\noi
and
\begin{align}
\begin{split}
\bigg| &\int_0^T \int_\R \big( \textup{II}_1^{N_1, N_2, N_3} + \textup{II}_2^{N_1, N_2, N_3} + \textup{II}_3^{N_1, N_2, N_3} \big) dx dt \bigg| \\
&\les N_1^{- \frac 32 + 10 \ta}  h N_1^{\frac 52} (1 + h N_1^{\frac 52}) \| \Pb_{N_1} u_1 \|_{V^2 L^2} \| \Pb_{N_2} u_2 \|_{V^2 L^2}  \| \Pb_{N_3} u_3 \|_{V^2 L^2}
\end{split}
\label{tridd3-2}
\end{align}

\noi
for any $\ta > 0$ sufficiently small. 
By interpolating \eqref{tridd3-1} and \eqref{tridd3-2} along with the fact that $0 < \frac 35 - 4 \ta + \frac 45 s < 1$ given $- \frac 34 < s \leq 0$ and $0 < \ta \ll 4s + 3$ and using Lemma~\ref{LEM:UVemb}, we get the following contribution in this case:
\begin{align*}
h^{\frac{3}{5} - 4 \ta + \frac 45 s} N_1^{s} N_2^s \| \Pb_{N_1} u_1 \|_{U^2 L^2} \| \Pb_{N_2} u_2 \|_{U^2 L^2} \| \Pb_{N_3} u_3 \|_{V^2 L^2}.
\end{align*}

\noi
Since $-s + \gamma > 0$, we can sum up the dyadic scales using the Cauchy-Schwarz inequalities first in $N_3$ and then in $N_1 \sim N_2$ to obtain the desired bound \eqref{trid2_goal}.

\smallskip \noi
{\bf Case 4:} $N_1 \sim N_2 \gg N_3 \sim 1$.

In this case, we use Lemma~\ref{LEM:tri_gen}~(iii) and Lemma~\ref{LEM:V2diff}~(i) and (ii) to obtain
\begin{align}
\begin{split}
\bigg| &\sum_{N_3 \sim 1} \int_0^T \int_\R \big( \textup{I}_1^{N_1, N_2, N_3} + \textup{I}_2^{N_1, N_2, N_3} \big) dx dt \bigg| \\
&\les N_1^{- \frac 32 + 20 \ta} \| \Pb_{N_1} u_1 \|_{V^2 L^2} \| \Pb_{N_2} u_2 \|_{V^2 L^2} \| \Pb_{\les 1} u_3 \|_{V^2 L^2}
\end{split}
\label{tridd4-1}
\end{align}

\noi
and
\begin{align}
\begin{split}
\bigg| &\sum_{N_3 \sim 1} \int_0^T \int_\R \big( \textup{II}_1^{N_1, N_2, N_3} + \textup{II}_2^{N_1, N_2, N_3} + \textup{II}_3^{N_1, N_2, N_3} \big) dx dt \bigg| \\
&\les N_1^{- \frac 32 + 20 \ta}  h N_1^{\frac 52} (1 + h N_1^{\frac 52}) \| \Pb_{N_1} u_1 \|_{V^2 L^2} \| \Pb_{N_2} u_2 \|_{V^2 L^2} \| \Pb_{\les 1} u_3 \|_{V^2 L^2}
\end{split}
\label{tridd4-2}
\end{align}

\noi
for any $\ta > 0$ sufficiently small. 
By interpolating \eqref{tridd4-1} and \eqref{tridd4-2} along with the fact that $0 < \frac 35 - 8 \ta + \frac 45 s < 1$ given $- \frac 34 < s \leq 0$ and $0 < \ta \ll 4s + 3$ and using Lemma~\ref{LEM:UVemb}, we get the following contribution in this case:
\begin{align*}
h^{\frac{3}{5} - 8 \ta + \frac 45 s} N_1^s N_2^s \| \Pb_{N_1} u_1 \|_{U^2 L^2} \| \Pb_{N_2} u_2 \|_{U^2 L^2}  \| \Pb_{\les 1} u_3 \|_{V^2 L^2}.
\end{align*}

\noi
We can then sum up the dyadic scales using the Cauchy-Schwarz inequality in $N_1 \sim N_2$ to obtain the desired bound \eqref{trid2_goal}.

\smallskip \noi
{\bf Case 5:} $N_1 \sim N_3 \gg N_2 \gg 1$.

In this case, we use Lemma~\ref{LEM:tri_gen}~(ii), the fact that $\frac{2}{h} |\sin (\frac{h \xi}{2})| \leq |\xi|$, and Lemma~\ref{LEM:V2diff}~(i) and (ii) to obtain
\begin{align}
\begin{split}
\bigg| &\int_0^T \int_\R \big( \textup{I}_1^{N_1, N_2, N_3} + \textup{I}_2^{N_1, N_2, N_3} \big) dx dt \bigg| \\
&\les N_1^{- \frac 12 + 10 \ta} N_2^{-1} \| \Pb_{N_1} u_1 \|_{V^2 L^2} \| \Pb_{N_2} u_2 \|_{V^2 L^2} \| \Pb_{N_3} u_3 \|_{V^2 L^2}
\end{split}
\label{tridd5-1}
\end{align}

\noi
and
\begin{align}
\begin{split}
\bigg| &\int_0^T \int_\R \big( \textup{II}_1^{N_1, N_2, N_3} + \textup{II}_2^{N_1, N_2, N_3} + \textup{II}_3^{N_1, N_2, N_3} \big) dx dt \bigg| \\
&\les N_1^{- \frac 12 + 10 \ta} N_2^{-1} h N_1^{\frac 52} (1 + h N_1^{\frac 52}) \| \Pb_{N_1} u_1 \|_{V^2 L^2} \| \Pb_{N_2} u_2 \|_{V^2 L^2} \| \Pb_{N_3} u_3 \|_{V^2 L^2}
\end{split}
\label{tridd5-2}
\end{align}

\noi
for any $\ta > 0$ sufficiently small. 
By interpolating \eqref{tridd5-1} and \eqref{tridd5-2} along with the fact that $0 < \frac 15 - 4 \ta + \frac 25 \gamma < 1$ given $0 < \gamma \leq 1$ and $\ta > 0$ sufficiently small and using Lemma~\ref{LEM:UVemb}, we get the following contribution in this case:
\begin{align*}
h^{\frac 15 - 4 \ta + \frac 25 \gamma} N_1^s N_2^{-1} N_3^{-s + \gamma} \| \Pb_{N_1} u_1 \|_{U^2 L^2} \| \Pb_{N_2} u_2 \|_{U^2 L^2}  \| \Pb_{N_3} u_3 \|_{V^2 L^2}.
\end{align*}

\noi
We can then sum up the dyadic scales using the Cauchy-Schwarz inequalities first in $N_2$ and then in $N_1 \sim N_3$ to obtain the desired bound \eqref{trid2_goal}.

\smallskip \noi
{\bf Case 6:} $N_1 \sim N_3 \gg N_2 \sim 1$.

In this case, we use Lemma~\ref{LEM:tri_gen}~(iv), the fact that $\frac{2}{h} |\sin (\frac{h \xi}{2})| \leq |\xi|$, and Lemma~\ref{LEM:V2diff}~(iii) and (iv) to obtain
\begin{align}
\begin{split}
\bigg| &\int_0^T \int_\R \sum_{N_2 \sim 1} \big( \textup{I}_1^{N_1, N_2, N_3} + \textup{I}_2^{N_1, N_2, N_3} \big) dx dt \bigg| \\
&\les  \| \Pb_{N_1} u_1 \|_{U^2 L^2} \| \Pb_{\les 1} u_2 \|_{U^2 L^2} \| \Pb_{N_3} u_3 \|_{V^2 L^2}
\end{split}
\label{tridd6-1}
\end{align}

\noi
and
\begin{align}
\begin{split}
\bigg| &\int_0^T \int_\R \sum_{N_2 \sim 1} \big( \textup{II}_1^{N_1, N_2, N_3} + \textup{II}_2^{N_1, N_2, N_3} + \textup{II}_3^{N_1, N_2, N_3} \big) dx dt \bigg| \\
&\les  h N_1^{\frac 52} (1 + h^2 N_1^5)^2 \| \Pb_{N_1} u_1 \|_{U^2 L^2} \| \Pb_{\les 1} u_2 \|_{U^2 L^2} \| \Pb_{N_3} u_3 \|_{V^2 L^2}
\end{split}
\label{tridd6-2}
\end{align}

\noi
for any $\ta > 0$ sufficiently small. 
By interpolating \eqref{tridd6-1} and \eqref{tridd6-2} given $0 < \gamma \leq 1$, we get the following contribution in this case:
\begin{align*}
h^{\frac 25 \gamma} N_1^s N_3^{-s + \gamma} \| \Pb_{N_1} u_1 \|_{U^2 L^2} \| \Pb_{\les 1} u_2 \|_{U^2 L^2}  \| \Pb_{N_3} u_3 \|_{V^2 L^2}.
\end{align*}

\noi
We can then sum up the dyadic scales using the Cauchy-Schwarz inequality in $N_1 \sim N_3$ to obtain the desired bound \eqref{trid2_goal}.

\smallskip \noi
{\bf Case 7:} $N_2 \sim N_3 \gg N_1 \gg 1$.

This case is symmetric to Case 5, and so we omit details.

\smallskip \noi
{\bf Case 8:} $N_2 \sim N_3 \gg N_1 \sim 1$.

This case is symmetric to Case 6, and so we omit details.
\end{proof}

Finally, we consider the term involving a high-frequency projection.

\begin{lemma}
\label{LEM:trid3}
Let $0 < h \leq 1$ and $0 < T \leq 1$. Let $- \frac 34 < s \leq 0$ and $\gamma \geq 0$. Then, we have
\begin{align}
\bigg| \int_0^T \int_\R (e^{\pm \frac{t}{24} \dx^3} u_1) (e^{\pm \frac{t}{24} \dx^3} u_2) (\dx e^{\pm \frac{t}{24} \dx^3} P_{> \frac{\pi}{h}} u_3) dx dt \les h^{\gamma} \| u_1 \|_{X_T^{s}} \| u_2 \|_{X_T^{s}} \| u_3 \|_{Y_T^{-s + \gamma}} .
\label{trid3_goal}
\end{align}
\end{lemma}

\begin{proof}
We use Lemma~\ref{LEM:tri} with $s_1 = s_2 = s$ and Lemma~\ref{LEM:PRuUV} to obtain
\begin{align*}
\text{LHS of \eqref{trid3_goal}} \les \| u_1 \|_{X_T^{s}} \| u_2 \|_{X_T^s} \| P_{> \frac{\pi}{h}} u_3 \|_{Y_T^{- s}} 
\les h^\gamma \| u_1 \|_{X_T^{s}} \| u_2 \|_{X_T^s} \| u_3 \|_{Y_T^{-s + \gamma}} ,
\end{align*}

\noi
giving the desired estimate.
\end{proof}

\subsection{Estimates on remainder terms}
\label{SUB:rem}

In this subsection, we show estimates on the remainder terms in $F (w_\pm^h)$ defined in \eqref{FR} appearing on the right-hand side of \eqref{whDuh}.
Again, our task is to create some positive powers of $h$ by losing some corresponding number of derivatives.

In the following five lemmas, we prove trilinear estimates that correspond to the first five terms in $F (w_\pm^h)$. The procedure of showing the trilinear estimates is similar to those performed in previous subsections, especially Lemma~\ref{LEM:tri_gen}, but the computations are different and are sometimes even more technical. Thus, for simplicity, we choose to omit some intermediate steps and focus only on the essential parts of the proof.

We first show the following estimate involving shifted frequency by $\cos (\frac{2 \pi}{h} \cdot)$.

\begin{lemma}
\label{LEM:rem1}
Let $0 < h \leq 1$ and $0 < T \leq 1$. 
Then, for any $- \frac 34 < s \leq 0$ and $0 \leq \gamma \leq 1 + s$, we have
\begin{align}
\begin{split}
\bigg| &\int_0^T \int_\R \big( e^{\pm \frac{t}{h^2} \nb_h} P_{\leq \frac{\pi}{h}} u_1 \big) \big( e^{\pm \frac{t}{h^2} \nb_h} P_{\leq \frac{\pi}{h}} u_2 \big) \big( \cos (\tfrac{2 \pi}{h} \cdot) \nb_h e^{\pm \frac{t}{h^2} \nb_h} P_{\leq \frac{\pi}{h}} u_3 \big) dx dt \bigg| \\
&\les h^{\gamma} T^{\frac 14} \| u_1 \|_{X_T^{s}} \| u_2 \|_{X_T^{s}} \| u_3 \|_{Y_T^{- s + \gamma}} .
\end{split}
\label{rem1_goal}
\end{align}
\end{lemma}
\begin{proof}
Note that $\cos (\frac{2 \pi}{h} x) = \frac 12 (e^{i \frac{2 \pi}{h} x} + e^{- i \frac{2 \pi}{h} x})$ given any $x \in \R$, and without loss of generality we only work with the $e^{i \frac{2 \pi}{h} x}$ part.
As usual, we only work with the ``$+$'' signs in \eqref{rem1_goal} and we ignore the subscripts $T$ on the right-hand side of \eqref{rem1_goal}.
Using a slight manipulation on Fourier multiplier operators and the Littlewood-Paley decomposition \eqref{PNdecomp}, we see that the main goal is to estimate
\begin{align}
\begin{split}
\int_\R &\int_\R \big( e^{\frac{t}{h^2} \nb_h} \mathbf{P}_{N_1} P_{\leq \frac{\pi}{h}} \ind_{[0, T)} u_1 \big) \big( e^{\frac{t}{h^2} \nb_h} \mathbf{P}_{N_2} P_{\leq \frac{\pi}{h}} \ind_{[0, T)} u_2 \big) \\
&\times \big( e^{i \frac{2 \pi}{h} \cdot} \nb_h e^{\frac{t}{h^2} \nb_h} \mathbf{P}_{N_3} P_{\leq \frac{\pi}{h}} \ind_{[0, T)} u_3 \big)   dx dt 
\end{split}
\label{rem1-0}
\end{align}

\noi
with $N_1, N_2, N_3 \geq 1$ dyadic. Note that we must have $N_1, N_2, N_3 \les \frac{1}{h}$ to get a non-trivial contribution.
Let $\xi_j$ be the spatial frequency of $u_j$, $j = 1, 2, 3$. Then, we must have $\xi_1 + \xi_2 + \xi_3 + \frac{2 \pi}{h} = 0$. Since $|\xi_j| \leq \frac{\pi}{h}$, we know that at least two $N_j$'s satisfy $N_j \sim \frac{1}{h}$.

We now consider the following cases.

\smallskip \noi
{\bf Case 1:} $N_1 \sim N_2 \sim \frac 1h \ges N_3 \gg 1$. 

In this case, we exploit smoothing from the maximum modulation.
Let $(\tau_j, \xi_j)$ be the temporal-spatial frequency of $e^{\frac{t}{h^2} \nb_h} P_{\leq \frac{\pi}{h}} \ind_{[0, T)} u_j$, $j = 1, 2, 3$. 
Then, we must have $\tau_1 + \tau_2 + \tau_3 = 0$, $\xi_1 + \xi_2 + \xi_3 + \frac{2 \pi}{h} = 0$, $|\xi_j| \sim N_j$, and $|\xi_j| \leq \frac{\pi}{h}$, $j = 1, 2, 3$. Note that each $\ind_{[0, T)} u_j$ has temporal frequency $\tau_j - \frac{2}{h^3} \sin (\frac{h \xi_j}{2})$ and we have
\begin{align*}
\max_{j = 1, 2, 3} \big| \tau_j - \tfrac{2}{h^3} \sin (\tfrac{h \xi_j}{2}) \big| 
&\ges \big| ( \tau_1 - \tfrac{2}{h^3} \sin (\tfrac{h \xi_1}{2}) ) + ( \tau_2 - \tfrac{2}{h^3} \sin (\tfrac{h \xi_2}{2}) ) + ( \tau_3 - \tfrac{2}{h^3} \sin (\tfrac{h \xi_3}{2}) ) \big| \\
&= \tfrac{2}{h^3} \big| - \sin (\tfrac{h \xi_1}{2}) - \sin (\tfrac{h \xi_2}{2}) - \sin (\tfrac{h (\xi_1 + \xi_2)}{2}) \big| \\
&= \tfrac{8}{h^3} \big| \cos (\tfrac{h \xi_1}{4}) \cos (\tfrac{h \xi_2}{4}) \cos (\tfrac{h \xi_3}{4}) \big| \\
&\sim h^{-3} .
\end{align*} 

\noi
Thus, by letting $M \in 2^\Z$ be such that $M \sim c h^{-3}$ for some small $c > 0$, we know that in order to get a non-trivial contribution, at least one of $u_j$'s must be equipped with the projector $\Qb_{>M}$. 

If $\Qb_{> M}$ hits $u_3$, then we use the Strichartz estimate in Corollary~\ref{COR:linU}~(i) on $u_2$, Lemma~\ref{LEM:QM}~(i) with $M \sim h^{-3}$ on $u_3$, and the fact that $|\frac{2}{h} \sin (\frac{h \xi}{2})| \leq \min (\frac{2}{h}, |\xi|) \les h^{- 1 - s + \gamma} |\xi|^{- s + \gamma}$ given $0 \leq \gamma \leq 1 + s$ to obtain
\begin{align*}
| \text{\eqref{rem1-0}} | \les h^{\frac 12 - s + \gamma} T^{\frac 14} N_1^{- \frac 14 - 2s} N_1^s \| \Pb_{N_1} u_1 \|_{U^2 L^2} N_2^s \| \Pb_{N_2} u_2 \|_{U^2 L^2} N_3^{- s + \gamma} \| \Pb_{N_3} u_3 \|_{V^2 L^2} .
\end{align*}

\noi
Since $h \sim N_1^{-1}$, we have $h^{\frac 12 - s} N_1^{- \frac 14 - 2 s} \sim N_1^{- \frac 32 - 2 s}$, which is summable given $- \frac 34 < s \leq 0$. 
The case when $\Qb_{> M}$ hits $u_1$ or when $\Qb_{> M}$ hits $u_2$ (for which we apply the Strichartz estimate on $u_1$) holds from a similar manner.

\smallskip \noi
{\bf Case 2:} $N_1 \sim N_2 \sim \frac 1h$ and $N_3 \sim 1$.

In this case, we apply a further dyadic decomposition $\Pb_{\leq 1} = \sum_{N_3' \leq 1} \Pb_{N_3'}$ and use the same steps as those in Case 1, but just with $|\frac{2}{h} \sin (\frac{h \xi}{2})| \les N_3'$ given $|\xi| \sim N_3'$ for $u_3$ so that the summation on $N_3' \leq 1$ becomes summable.

\smallskip \noi
{\bf Case 3:} $N_1 \sim N_3 \sim \frac{1}{h}\gg N_2 \gg 1$.

In this case, we exploit smoothing from the maximum modulation as in Case 1. 
If $\Qb_{> M}$ hits $u_3$, then we use the Strichartz estimate in Corollary~\ref{COR:linU}~(i) on $u_2$, Lemma~\ref{LEM:QM}~(i) with $M \sim h^{-3}$ on $u_3$, and the fact that $|\frac{2}{h} \sin (\frac{h \xi}{2})| \leq \min (\frac{2}{h}, |\xi|) \les h^{- 1 + \gamma} |\xi|^{\gamma}$ to obtain
\begin{align*}
| \text{\eqref{rem1-0}} | \les h^{\frac 12 + \gamma} T^{\frac 14} N_2^{- \frac 14 - s} N_1^s \| \Pb_{N_1} u_1 \|_{U^2 L^2} N_2^s \| \Pb_{N_2} u_2 \|_{U^2 L^2} N_3^{-s + \gamma} \| \Pb_{N_3} u_3 \|_{V^2 L^2} .
\end{align*}

\noi
Since $h \sim N_1^{-1}$, we have $h^{\frac 12} N_2^{- \frac 14 - s} \sim N_1^{- \frac 12} N_2^{- \frac 14 - s}$, which is summable given $-\frac 34 < s \leq 0$. The case when $\Qb_{> M}$ hits $u_1$ or when $\Qb_{> M}$ hits $u_2$ (for which we apply the Strichartz estimate on $u_1$) holds from a similar manner.

\smallskip \noi
{\bf Case 4:} $N_1 \sim N_3 \sim \frac 1h$ and $N_2 \sim 1$.

In this case, we use the local smoothing estimate in Corollary~\ref{COR:linU}~(iii) on $u_1$, the maximal function estimate in Corollary~\ref{COR:linU}~(ii) on $u_2$, and the fact that $|\frac{2}{h} \sin (\frac{h \xi}{2})| \leq \min (\frac{2}{h}, |\xi|) \les h^{- 1 + \gamma} |\xi|^{\gamma}$ to obtain
\begin{align*}
| \text{\eqref{rem1-0}} | \les h^{-1 + \gamma} T^{\frac 12} N_1^{-1} N_1^s \| \Pb_{N_1} u_1 \|_{U^2 L^2}  \| \Pb_{\les 1} u_2 \|_{U^2 L^2} N_3^{-s + \gamma} \| \Pb_{N_3} u_3 \|_{V^2 L^2} .
\end{align*}

\noi
By taking $h^{-1} N_1^{-1} \sim 1$, we obtain the desired estimate.

\smallskip \noi
{\bf Case 5:} $N_2 \sim N_3 \sim \frac 1h \gg N_1 \gg 1$.

This case is symmetric to Case 3, and so we omit details.

\smallskip \noi
{\bf Case 6:} $N_2 \sim N_3 \sim \frac 1h$ and $N_1 \sim 1$.

This case is symmetric to Case 4, and so we omit details.
\end{proof}

We now look at the term involving modulations with opposite signs.

\begin{lemma}
\label{LEM:rem2}
Let $0 < h \leq 1$ and $0 < T \leq 1$.
Then, for any $- \frac 34 < s \leq 0$ and $0 \leq \gamma \leq 1 + s$, we have
\begin{align*}
\bigg| &\int_0^T \int_\R \big( e^{\pm \frac{t}{h^2} \nb_h} P_{\leq \frac{\pi}{h}} u_1 \big) \big( e^{\mp \frac{t}{h^2} \nb_h} P_{\leq \frac{\pi}{h}} u_2 \big) \big( \nb_h e^{\pm \frac{t}{h^2} \nb_h} P_{\leq \frac{\pi}{h}} u_3 \big) dx dt \bigg| \\
&\les h^\gamma T^{\frac 14} \| u_1 \|_{X_T^{s}} \| u_2 \|_{X_T^{s}} \| u_3 \|_{Y_T^{- s + \gamma}} .
\end{align*}
\end{lemma}
\begin{proof}
As before, we mainly look at
\begin{align}
\begin{split}
\int_\R &\int_\R \big( e^{\frac{t}{h^2} \nb_h} \Pb_{N_1} P_{\leq \frac{\pi}{h}} \ind_{[0, T)} u_1 \big) \big( e^{- \frac{t}{h^2} \nb_h} \Pb_{N_2} P_{\leq \frac{\pi}{h}} \ind_{[0, T)} u_2 \big) \\
&\quad \times \big( \nb_h e^{\frac{t}{h^2} \nb_h}  \Pb_{N_3} P_{\leq \frac{\pi}{h}} \ind_{[0, T)} u_3 \big)   dx dt 
\end{split}
\label{rem2-0}
\end{align}

\noi
with $N_1, N_2, N_3 \geq 1$ dyadic. Note that we must have $N_1, N_2, N_3 \les \frac{1}{h}$ to get a non-trivial contribution.

Since the largest two frequencies must be comparable, we consider the following cases.

\smallskip \noi
{\bf Case 1:} $N_1, N_2, N_3 \sim 1$.

In this case, by using the bilinear estimate in Corollary~\ref{COR:bilinU}~(iii) on the product of $u_1$ and $u_2$, we get 
\begin{align*}
| \eqref{rem2-0} | \les h T^{\frac 12} \| \Pb_{\les 1} u_1 \|_{U^2 L^2} \| \Pb_{\les 1} u_2 \|_{U^2 L^2} \| \Pb_{\les 1} u_3 \|_{V^2 L^2} ,
\end{align*}

\noi
which gives the desired estimate.

\smallskip \noi
{\bf Case 2:} $N_1 \sim N_2 \ges N_3 \gg 1$.

In this case, we exploit smoothing from the maximum modulation.  
Let $(\tau_j, \xi_j)$ be the temporal-spatial frequency of $e^{\frac{t}{h^2} \nb_h} P_{\leq \frac{\pi}{h}} \ind_{[0, T)} u_j$, $j = 1, 3$, and let $(\tau_2, \xi_2)$ be the temporal-spatial frequency of $e^{- \frac{t}{h^2} \nb_h} P_{\leq \frac{\pi}{h}} \ind_{[0, T)} u_2$.
Then, we must have $\tau_1 + \tau_2 + \tau_3 = 0$, $\xi_1 + \xi_2 + \xi_3 = 0$, $|\xi_j| \sim N_j$ for $j = 1, 2, 3$, and $|\xi_j| \leq \frac{\pi}{h}$ for $j = 1, 2, 3$. Note that $\ind_{[0, T)} u_1$ has temporal frequency $\tau_1 - \frac{2}{h^3} \sin (\frac{h \xi_1}{2})$, $\ind_{[0, T)} u_2$ has temporal frequency $\tau_2 + \frac{2}{h^3} \sin (\frac{h \xi_2}{2})$, and $\ind_{[0, T)} u_3$ has temporal frequency $\tau_3 - \frac{2}{h^3} \sin (\frac{h \xi_3}{2})$. Then, we have
\begin{align*}
&\max \big( \big| \tau_1 - \tfrac{2}{h^3} \sin (\tfrac{h \xi_1}{2}) \big|, \big| \tau_2 + \tfrac{2}{h^3} \sin (\tfrac{h \xi_2}{2}) \big|, \big| \tau_3 - \tfrac{2}{h^3} \sin (\tfrac{h \xi_3}{2}) \big| \big)  \\
&\ges \big| ( \tau_1 - \tfrac{2}{h^3} \sin (\tfrac{h \xi_1}{2}) ) + ( \tau_2 + \tfrac{2}{h^3} \sin (\tfrac{h \xi_2}{2}) ) + ( \tau_3 - \tfrac{2}{h^3} \sin (\tfrac{h \xi_3}{2}) ) \big| \\
&= \tfrac{2}{h^3} \big| - \sin (\tfrac{h \xi_1}{2}) + \sin (\tfrac{h \xi_2}{2}) + \sin (\tfrac{h (\xi_1 + \xi_2)}{2}) \big| \\
&= \tfrac{8}{h^3} \big| \cos (\tfrac{h \xi_1}{4}) \sin (\tfrac{h \xi_2}{4}) \cos (\tfrac{h \xi_3}{4}) \big| \\
&\sim h^{- 2} N_2 .
\end{align*} 

\noi
Thus, by letting $M \in 2^\Z$ be such that $M \sim c h^{-2} N_2$ for some small $c > 0$, we know that in order to get a non-trivial contribution, at least one of $u_j$'s must be equipped with the projector $\Qb_{>M}$.

If $\Qb_{>M}$ hits $u_3$, we use the Strichartz estimate in Corollary~\ref{COR:linU}~(i) on $u_2$, Lemma~\ref{LEM:QM}~(i) with $M \sim h^{-2} N_2$ on $u_3$, and the fact that $| \frac{2}{h} \sin (\frac{h \xi}{2}) | \leq \min (\frac{2}{h}, |\xi|) \les h^{- 1 - s + \gamma} |\xi|^{-s + \gamma}$ given $0 \leq \gamma \leq 1 + s$ to obtain
\begin{align*}
|\text{\eqref{rem2-0}}| \les h^{- s + \gamma} T^{\frac 14} N_1^{- \frac 34 - 2s} N_1^s \| \Pb_{N_1} u_1 \|_{U^2 L^2} N_2^s \| \Pb_{N_2} u_2 \|_{U^2 L^2} N_3^{-s + \gamma} \| \Pb_{N_3} u_3 \|_{V^2 L^2} .
\end{align*}

\noi
Since $h \les N_1^{-1}$, we have $h^{-s} N_1^{- \frac 34 - 2s} \les N_1^{- \frac 34 - s}$, which is summable given $- \frac 34 < s \leq 0$. The case when $\Qb_{> M}$ hits $u_1$ or when $\Qb_{> M}$ hits $u_2$ (for which we apply the Strichartz estimate on $u_1$) holds from a similar manner.

\smallskip \noi
{\bf Case 3:} $N_1 \sim N_2 \gg N_3 \sim 1$.

In this case, we apply a further dyadic decomposition $\Pb_{\leq 1} = \sum_{N_3' \leq 1} \Pb_{N_3'}$ and use the same steps as those in Case 2, but just with $|\frac{2}{h} \sin (\frac{h \xi}{2})| \les N_3'$ given $|\xi| \sim N_3'$ for $u_3$ so that the summation on $N_3' \leq 1$ becomes summable.

\smallskip \noi
{\bf Case 4:} $N_1 \sim N_3 \gg N_2 \gg 1$.

In this case, we exploit smoothing from the maximum modulation as in Case 2.
If $\Qb_{>M}$ hits $u_3$, we use the Strichartz estimate in Corollary~\ref{COR:linU}~(i) on $u_2$, Lemma~\ref{LEM:QM}~(i) with $M \sim h^{-2} N_2$ on $u_3$, and the fact that $| \frac{2}{h} \sin (\frac{h \xi}{2}) | \leq \min (\frac{2}{h}, |\xi|) \les h^{- 1 + \gamma} |\xi|^\gamma$ to obtain
\begin{align*}
|\text{\eqref{rem2-0}}| \les h^\gamma T^{\frac 14} N_2^{- \frac 34 - s} N_1^s \| \Pb_{N_1} u_1 \|_{U^2 L^2} N_2^s \| \Pb_{N_2} u_2 \|_{U^2 L^2} N_3^{-s + \gamma} \| \Pb_{N_3} u_3 \|_{V^2 L^2} .
\end{align*}

\noi
Since $- \frac 34 < s \leq 0$, we obtain the desired estimate by summing up  $N_2$ and using the Cauchy-Schwarz inequality on $N_1 \sim N_3 \gg 1$. The case when $\Qb_{> M}$ hits $u_1$ or when $\Qb_{> M}$ hits $u_2$ (for which we apply the Strichartz estimate on $u_1$) holds from a similar manner.

\smallskip \noi
{\bf Case 5:} $N_1 \sim N_3 \gg N_2 \sim 1$.

In this case, we use the bilinear estimate in Corollary~\ref{COR:bilinU}~(iii) on the product of $u_1$ and $u_2$ and the fact that $|\frac{2}{h} \sin (\frac{h \xi}{2})| \leq \min (\frac{2}{h}, |\xi|) \les h^{- 1 + \gamma} |\xi|^\gamma$ to obtain
\begin{align*}
| \eqref{rem2-0} | \les h^\gamma T^{\frac 12} N_1^s \| \Pb_{N_1} u_1 \|_{U^2 L^2} \| \Pb_{\les 1} u_2 \|_{U^2 L^2} N_3^{-s + \gamma} \| \Pb_{N_3} u_3 \|_{V^2 L^2} ,
\end{align*}

\noi
and so we can sum up $N_1 \sim N_3 \gg 1$ using the Cauchy-Schwarz inequality to obtain the desired estimate.

\smallskip \noi
{\bf Case 6:} $N_2 \sim N_3 \gg N_1 \gg 1$.

This case is similar to (and slightly easier than) Case 4, and so we omit details.

\smallskip \noi
{\bf Case 7:} $N_2 \sim N_3 \gg N_1 \sim 1$.

This case is symmetric to Case 5, and so we omit details.
\end{proof}

We then look at a term involving both shifted frequency by $\cos (\frac{2 \pi}{h} \cdot)$ and modulations with opposite signs.

\begin{lemma}
\label{LEM:rem3}
Let $0 < h \leq 1$ and $0 < T \leq 1$.
Then, for any $-\frac 34 < s \leq 0$, $0 \leq \gamma \leq \frac 12$ satisfying $\gamma < \frac 32 + 2s$, and $\ta > 0$ sufficiently small, we have
\begin{align*}
\bigg| &\int_0^T \int_\R \big( e^{\pm \frac{t}{h^2} \nb_h} P_{\leq \frac{\pi}{h}} u_1 \big) \big( e^{\mp \frac{t}{h^2} \nb_h} P_{\leq \frac{\pi}{h}} u_2 \big) \big( \cos (\tfrac{2 \pi}{h} \cdot) \nb_h e^{\pm \frac{t}{h^2} \nb_h} P_{\leq \frac{\pi}{h}} u_3 \big) dx dt \bigg| \\
&\les h^\gamma T^{\ta} \| u_1 \|_{X_T^{s}} \| u_2 \|_{X_T^{s}} \| u_3 \|_{Y_T^{-s + \gamma}} .
\end{align*}
\end{lemma}
\begin{proof}
As before, we mainly look at
\begin{align}
\begin{split}
\int_\R &\int_\R \big( e^{\frac{t}{h^2} \nb_h} \Pb_{N_1} P_{\leq \frac{\pi}{h}} \ind_{[0, T)} u_1 \big) \big( e^{- \frac{t}{h^2} \nb_h} \Pb_{N_2} P_{\leq \frac{\pi}{h}} \ind_{[0, T)} u_2 \big) \\
&\quad \times \big( e^{i \frac{2 \pi}{h} \cdot} \nb_h e^{\frac{t}{h^2} \nb_h}  \Pb_{N_3} P_{\leq \frac{\pi}{h}} \ind_{[0, T)} u_3 \big)   dx dt 
\end{split}
\label{rem3-0}
\end{align}

\noi
with $N_1, N_2, N_3 \geq 1$ dyadic. Note that we must have $N_1, N_2, N_3 \les \frac{1}{h}$ to get a non-trivial contribution.
Let $\xi_j$ be the spatial frequency of $u_j$, $j = 1, 2, 3$. Then, we must have $\xi_1 + \xi_2 = \xi_3 + \frac{2 \pi}{h}$. Since $|\xi_j| \leq \frac{\pi}{h}$, we know that at least two $N_j$'s satisfy $N_j \sim \frac{1}{h}$.

\smallskip \noi
{\bf Case 1:} $N_1 \sim N_2 \sim N_3 \sim \frac 1h$.

In this case, we exploit smoothing from the maximum modulation. 
Let $(\tau_j, \xi_j)$ be the temporal-spatial frequency of $e^{\frac{t}{h^2} \nb_h} P_{\leq \frac{\pi}{h}} \ind_{[0, T)} u_j$, $j = 1, 3$, and let $(\tau_2, \xi_2)$ be the temporal-spatial frequency of $e^{- \frac{t}{h^2} \nb_h} P_{\leq \frac{\pi}{h}} \ind_{[0, T)} u_2$.
Then, we must have $\tau_1 + \tau_2 + \tau_3 = 0$, $\xi_1 + \xi_2 + \xi_3 + \frac{2 \pi}{h} = 0$, $|\xi_j| \sim N_j$ for $j = 1, 2, 3$, and $|\xi_j| \leq \frac{\pi}{h}$ for $j = 1, 2, 3$. Note that $\ind_{[0, T)} u_1$ has temporal frequency $\tau_1 - \frac{2}{h^3} \sin (\frac{h \xi_1}{2})$, $\ind_{[0, T)} u_2$ has temporal frequency $\tau_2 + \frac{2}{h^3} \sin (\frac{h \xi_2}{2})$, and $\ind_{[0, T)} u_3$ has temporal frequency $\tau_3 - \frac{2}{h^3} \sin (\frac{h \xi_3}{2})$. Then, we have
\begin{align*}
&\max \big( \big| \tau_1 - \tfrac{2}{h^3} \sin (\tfrac{h \xi_1}{2}) \big|, \big| \tau_2 + \tfrac{2}{h^3} \sin (\tfrac{h \xi_2}{2}) \big|, \big| \tau_3 - \tfrac{2}{h^3} \sin (\tfrac{h \xi_3}{2}) \big| \big)  \\
&\ges \big| ( \tau_1 - \tfrac{2}{h^3} \sin (\tfrac{h \xi_1}{2}) ) + ( \tau_2 + \tfrac{2}{h^3} \sin (\tfrac{h \xi_2}{2}) ) + ( \tau_3 - \tfrac{2}{h^3} \sin (\tfrac{h \xi_3}{2}) ) \big| \\
&= \tfrac{2}{h^3} \big| - \sin (\tfrac{h \xi_1}{2}) + \sin (\tfrac{h \xi_2}{2}) - \sin (\tfrac{h (\xi_1 + \xi_2)}{2}) \big| \\
&= \tfrac{8}{h^3} \big| \sin (\tfrac{h \xi_1}{4}) \cos (\tfrac{h \xi_2}{4}) \sin (\tfrac{h \xi_3}{4}) \big| \\
&\sim h^{-1} N_1 N_3 .
\end{align*} 

\noi
Thus, by letting $M \in 2^\Z$ be such that $M \sim c h^{-1} N_1 N_3$ for some small $c > 0$, we know that in order to get a non-trivial contribution, at least one of $u_j$'s must be equipped with the projector $\Qb_{> M}$.


If $\Qb_{> M}$ hits $u_3$, we use the Strichartz estimate in Corollary~\ref{COR:linU}~(i) on $u_2$,  Lemma~\ref{LEM:QM}~(i) with $M \sim h^{-1} N_1 N_3$ on $u_3$, and the fact that $|\frac{2}{h} \sin (\frac{h \xi}{2})| \leq \min (\frac{2}{h}, |\xi|) \les h^{-1 + \gamma} |\xi|^{\gamma}$ to obtain
\begin{align*}
|\eqref{rem3-0}| \les h^{- \frac 12 + \gamma} T^{\frac 14} N_1^{- \frac 54 - s} 
N_1^s \| \Pb_{N_1} u_1 \|_{U^2 L^2} N_2^s \| \Pb_{N_2} u_2 \|_{U^2 L^2} N_3^{-s + \gamma} \| \Pb_{N_3} u_3 \|_{V^2 L^2} .
\end{align*}

\noi
Since $h \sim N_1^{-1}$, we have $h^{- \frac 12} N_1^{- \frac 54 - s} \sim N_1^{- \frac 34 - s}$, which is valid given $- \frac 34 < s \leq 0$.
The case when $\Qb_{> M}$ hits $u_1$ or when $\Qb_{> M}$ hits $u_2$ (for which we apply the Strichartz estimate on $u_1$) holds from a similar manner, and so we omit details.

\smallskip \noi
{\bf Case 2:} $N_1 \sim N_2 \sim \frac{1}{h} \gg N_3 \gg 1$.

In this case, we exploit smoothing from the maximum modulation as in Case 1. 
If $\Qb_{> M}$ hits $u_2$, we use the bilinear estimate in Corollary~\ref{COR:bilinU}~(i) on the product of $u_1$ and $u_3$,  Lemma~\ref{LEM:QM}~(i) with $M \sim h^{-1} N_1 N_3$ on $u_2$, and the fact that $|\frac 2h \sin (\frac{h \xi}{2})| \leq \min (\frac 2h, |\xi|) \les h^{- \frac 12 + \gamma} |\xi|^{\frac 12 + \gamma}$ given $0 \leq \gamma < \frac 12$ to obtain
\begin{align*}
|\eqref{rem3-0}| \les h^\gamma N_1^{- \frac 32 - 2 s} N_3^{s} N_1^s \| \Pb_{N_1} u_1 \|_{U^2 L^2} N_2^s \| \Pb_{N_2} u_2 \|_{U^2 L^2} N_3^{-s + \gamma} \| \Pb_{N_3} u_3 \|_{U^2 L^2} .
\end{align*}

\noi
The coefficient $N_1^{- \frac 32 - 2 s} N_3^s$ is summable given $- \frac 34 < s \leq 0$.

If $\Qb_{> M}$ hits $u_1$, we use the bilinear estimate in Corollary~\ref{COR:bilinU}~(iii) on the product of $u_2$ and $u_3$, Lemma~\ref{LEM:QM}~(i) with $M \sim h^{-1} N_1 N_3$ on $u_1$, and the fact that $|\frac 2h \sin (\frac{h \xi}{2})| \leq \min (\frac 2h, |\xi|) \les h^{- \frac 12 + \gamma} |\xi|^{\frac 12 + \gamma}$ to obtain
\begin{align*}
|\eqref{rem3-0}| \les h^{1 + \gamma} N_1^{- \frac 12 - 2 s} N_3^{s} N_1^s \| \Pb_{N_1} u_1 \|_{U^2 L^2} N_2^s \| \Pb_{N_2} u_2 \|_{U^2 L^2} N_3^{-s + \gamma} \| \Pb_{N_3} u_3 \|_{U^2 L^2} .
\end{align*}

\noi
Since $h \sim N_1^{-1}$, we have $h N_1^{- \frac 12 - 2 s} N_3^s \sim N_1^{- \frac 32 - 2 s} N_3^s$, which is summable given $- \frac 34 < s \leq 0$.

If $\Qb_{> M}$ hits $u_3$, we consider the following two scenarios. If $\frac{9 \pi}{10 h} \leq \xi_1, \xi_2 \leq \frac{\pi}{h}$, since $\xi_1 + \xi_2 = \xi_3 + \frac{2 \pi}{h}$ with $|\xi_3| \sim N_3$, we use the bilinear estimate in Corollary~\ref{COR:bilinU}~(iv) on the product of $u_1$ and $u_2$, Lemma~\ref{LEM:QM}~(i) with $M \sim h^{-1} N_1 N_3$ on $u_3$, and the fact that $|\frac 2h \sin (\frac{h \xi}{2})| \leq |\xi|$ to obtain
\begin{align}
|\eqref{rem3-0}| \les h N_1^{- \frac 12 - 2 s} N_3^{s - \gamma} N_1^s \| \Pb_{N_1} u_1 \|_{U^2 L^2} N_2^s \| \Pb_{N_2} u_2 \|_{U^2 L^2} N_3^{-s + \gamma} \| \Pb_{N_3} u_3 \|_{V^2 L^2} .
\label{rem3-1}
\end{align} 

\noi
Since $h \sim N_1^{-1}$, we have $h N_1^{- \frac 12 - 2 s} N_3^{s - \gamma} \sim h^\gamma N_1^{- \frac 32 - 2s + \gamma} N_3^{s - \gamma}$, which is summable given $- \frac 34 < s \leq 0$ and $0 \leq \gamma < \frac 32 + 2s$. 
If $\xi_1 \leq \frac{9 \pi}{10 h}$ or $\xi_2 \leq \frac{9 \pi}{10 h}$ (it is impossible to have $- \frac{\pi}{h} \leq \xi_1, \xi_2 \leq - \frac{9 \pi}{10 h}$ due to $\xi_1 + \xi_2 = \xi_3 + \frac{2 \pi}{h}$ and $|\xi_3| \leq \frac{\pi}{h}$), we use the bilinear estimate in Corollary~\ref{COR:bilinU}~(iii) on the product of $u_1$ and $u_2$,  Lemma~\ref{LEM:QM}~(i) with $M \sim h^{-1} N_1 N_3$ on $u_3$, and the fact that $|\frac 2h \sin (\frac{h \xi}{2})| \leq \min (\frac 2h, |\xi|) \les h^{- \frac 12 + \gamma} |\xi|^{\frac 12 + \gamma}$ to obtain
\begin{align*}
|\eqref{rem3-0}| \les h^{1 + \gamma} N_1^{- \frac 12 - 2 s} N_3^{s} N_1^s \| \Pb_{N_1} u_1 \|_{U^2 L^2} N_2^s \| \Pb_{N_2} u_2 \|_{U^2 L^2} N_3^{-s + \gamma} \| \Pb_{N_3} u_3 \|_{V^2 L^2} .
\end{align*} 

\noi
Since $h \sim N_1^{-1}$, we have $h N_1^{- \frac 12 - 2 s} N_3^s \sim N_1^{- \frac 32 - 2 s} N_3^s$, which is summable given $- \frac 34 < s \leq 0$. 

In all of the above situations, we need to use the interpolation in Lemma~\ref{LEM:interp} with a trivial estimate to create a factor $T^\ta$ and put a $V^2 L^2$-norm for $u_3$.

\smallskip \noi
{\bf Case 3:} $N_1 \sim N_2 \sim \frac 1h$ and $N_3 \sim 1$.

In this case, we apply a further dyadic decomposition $\Pb_{\leq 1} = \sum_{N_3' \leq 1} \Pb_{N_3'}$ and use the same steps as those in Case 2, but just with $|\frac{2}{h} \sin (\frac{h \xi}{2})| \les N_3'$ given $|\xi| \sim N_3'$ for $u_3$ so that the summation on $N_3' \leq 1$ becomes summable. 
The only issue is \eqref{rem3-1} as no factor of $N_3'$ appears even with $|\frac{2}{h} \sin (\frac{h \xi}{2})| \les N_3'$. To deal with this term, we simply interpolate the estimate \eqref{rem3-1} with a trivial estimate with a factor $(N_3')^a$ for some $a > 0$ to obtain the desired estimate (see the proof of \eqref{help}).

\smallskip \noi
{\bf Case 4:} $N_1 \sim N_3 \sim \frac{1}{h}\gg N_2 \gg 1$.

In this case, we exploit smoothing from the maximum modulation as in Case 1. 
If $\Qb_{> M}$ hits $u_3$, then we use the Strichartz estimate in Corollary~\ref{COR:linU}~(i) on $u_2$, Lemma~\ref{LEM:QM}~(i) with $M \sim h^{-1} N_1 N_3$ on $u_3$, and the fact that $|\frac{2}{h} \sin (\frac{h \xi}{2})| \leq \min (\frac{2}{h}, |\xi|) \les h^{- 1 + \gamma} |\xi|^{\gamma}$ to obtain
\begin{align*}
| \text{\eqref{rem3-0}} | \les h^{- \frac 12 + \gamma} T^{\frac 14} N_1^{-1} N_2^{- \frac 14 - s} N_1^s \| \Pb_{N_1} u_1 \|_{U^2 L^2} N_2^s \| \Pb_{N_2} u_2 \|_{U^2 L^2} N_3^{-s + \gamma} \| \Pb_{N_3} u_3 \|_{V^2 L^2} .
\end{align*}

\noi
Since $h \sim N_1^{-1}$, we have $h^{-  \frac 12} N_1^{-1} N_2^{- \frac 14 - s} \sim N_1^{- \frac 12} N_2^{- \frac 14 - s} \ll N_2^{- \frac 34 - s}$, which is summable given $-\frac 34 < s \leq 0$. The case when $\Qb_{> M}$ hits $u_1$ or when $\Qb_{> M}$ hits $u_2$ (for which we apply the Strichartz estimate on $u_1$) holds from a similar manner.

\smallskip \noi
{\bf Case 5:} $N_1 \sim N_3 \sim \frac 1h$ and $N_2 \sim 1$.

In this case, we use the local smoothing estimate in Corollary~\ref{COR:linU}~(iii) on $u_1$, the maximal function estimate in Corollary~\ref{COR:linU}~(ii) on $u_2$, and the fact that $|\frac{2}{h} \sin (\frac{h \xi}{2})| \leq \min (\frac{2}{h}, |\xi|) \les h^{- 1 + \gamma} |\xi|^{\gamma}$ to obtain
\begin{align*}
| \text{\eqref{rem3-0}} | \les h^{-1 + \gamma} T^{\frac 12} N_1^{-1} N_1^s \| \Pb_{N_1} u_1 \|_{U^2 L^2}  \| \Pb_{\les 1} u_2 \|_{U^2 L^2} N_3^{-s + \gamma} \| \Pb_{N_3} u_3 \|_{V^2 L^2} .
\end{align*}

\noi
By taking $h^{-1} N_1^{-1} \sim 1$, we obtain the desired estimate.

\smallskip \noi
{\bf Case 6:} $N_2 \sim N_3 \sim \frac 1h \gg N_1 \gg 1$.

This case follows from the same steps as those in Case 4 with minor modifications on the computations, and so we omit details.

\smallskip \noi
{\bf Case 7:} $N_2 \sim N_3 \sim \frac 1h$ and $N_1 \sim 1$.

This case is symmetric to Case 5, and so we omit details.
\end{proof}

The following term, compared to the one in Lemma~\ref{LEM:rem2}, involves a modulation with an opposite sign at a different place.

\begin{lemma}
\label{LEM:rem4}
Let $0 < h \leq 1$ and $0 < T \leq 1$.
Then, for any $- \frac 34 < s \leq 0$, $0 \leq \gamma \leq \frac 12$ satisfying $\gamma < \frac 32 + 2s$, and $\ta > 0$ sufficiently small, we have
\begin{align*}
\bigg| &\int_0^T \int_\R \big( e^{\pm \frac{t}{h^2} \nb_h} P_{\leq \frac{\pi}{h}} u_1 \big) \big( e^{\pm \frac{t}{h^2} \nb_h} P_{\leq \frac{\pi}{h}} u_2 \big) \big( \nb_h e^{\mp \frac{t}{h^2} \nb_h} P_{\leq \frac{\pi}{h}} u_3 \big) dx dt \bigg| \\
&\les h^\gamma T^{\ta} \| u_1 \|_{X_T^{s}} \| u_2 \|_{X_T^{s}} \| u_3 \|_{Y_T^{- s + \gamma}} .
\end{align*}
\end{lemma}
\begin{proof}
As before, we mainly look at
\begin{align}
\begin{split}
\int_\R &\int_\R \big( e^{\frac{t}{h^2} \nb_h} \Pb_{N_1} P_{\leq \frac{\pi}{h}} \ind_{[0, T)} u_1 \big) \big( e^{\frac{t}{h^2} \nb_h} \Pb_{N_2} P_{\leq \frac{\pi}{h}} \ind_{[0, T)} u_2 \big) \\
&\quad \times \big( \nb_h e^{- \frac{t}{h^2} \nb_h}  \Pb_{N_3} P_{\leq \frac{\pi}{h}} \ind_{[0, T)} u_3 \big)   dx dt 
\end{split}
\label{rem4-0}
\end{align}


\noi
with $N_1, N_2, N_3 \geq 1$ dyadic. Note that we must have $N_1, N_2, N_3 \les \frac{1}{h}$ to get a non-trivial contribution.

Since the largest two frequencies are comparable, we consider the following cases.

\smallskip \noi
{\bf Case 1:} $N_1, N_2, N_3 \sim 1$.

In this case, by using the bilinear estimate in Corollary~\ref{COR:bilinU}~(iii) on the product of $u_1$ and $u_3$, we get 
\begin{align*}
| \eqref{rem4-0} | \les h T^{\frac 12} \| \Pb_{\les 1} u_1 \|_{U^2 L^2} \| \Pb_{\les 1} u_2 \|_{U^2 L^2} \| \Pb_{\les 1} u_3 \|_{V^2 L^2} ,
\end{align*}

\noi
which gives the desired estimate.

\smallskip \noi
{\bf Case 2:} $N_1 \sim N_2 \sim N_3 \gg 1$.

In this case, we exploit smoothing from the maximum modulation.
Let $(\tau_j, \xi_j)$ be the temporal-spatial frequency of $e^{\frac{t}{h^2} \nb_h} P_{\leq \frac{\pi}{h}} \ind_{[0, T)} u_j$, $j = 1, 2$, and let $(\tau_3, \xi_3)$ be the temporal-spatial frequency of $e^{- \frac{t}{h^2} \nb_h} P_{\leq \frac{\pi}{h}} \ind_{[0, T)} u_3$.
Then, we must have $\tau_1 + \tau_2 + \tau_3 = 0$, $\xi_1 + \xi_2 + \xi_3 = 0$, $|\xi_j| \sim N_j$ for $j = 1, 2, 3$, and $|\xi_j| \leq \frac{\pi}{h}$ for $j = 1, 2, 3$. Note that $\ind_{[0, T)} u_1$ has temporal frequency $\tau_1 - \frac{2}{h^3} \sin (\frac{h \xi_1}{2})$, $\ind_{[0, T)} u_2$ has temporal frequency $\tau_2 - \frac{2}{h^3} \sin (\frac{h \xi_2}{2})$, and $\ind_{[0, T)} u_3$ has temporal frequency $\tau_3 + \frac{2}{h^3} \sin (\frac{h \xi_3}{2})$. Then, we have
\begin{align*}
&\max \big( \big| \tau_1 - \tfrac{2}{h^3} \sin (\tfrac{h \xi_1}{2}) \big|, \big| \tau_2 - \tfrac{2}{h^3} \sin (\tfrac{h \xi_2}{2}) \big|, \big| \tau_3 + \tfrac{2}{h^3} \sin (\tfrac{h \xi_3}{2}) \big| \big)  \\
&\ges \big| ( \tau_1 - \tfrac{2}{h^3} \sin (\tfrac{h \xi_1}{2}) ) + ( \tau_2 - \tfrac{2}{h^3} \sin (\tfrac{h \xi_2}{2}) ) + ( \tau_3 + \tfrac{2}{h^3} \sin (\tfrac{h \xi_3}{2}) ) \big| \\
&= \tfrac{2}{h^3} \big| - \sin (\tfrac{h \xi_1}{2}) - \sin (\tfrac{h \xi_2}{2}) - \sin (\tfrac{h (\xi_1 + \xi_2)}{2}) \big| \\
&= \tfrac{8}{h^3} \big| \cos (\tfrac{h \xi_1}{4}) \cos (\tfrac{h \xi_2}{4}) \sin (\tfrac{h \xi_3}{4}) \big| \\
&\sim h^{-2} N_3 .
\end{align*} 

\noi
Thus, by letting $M \in 2^\Z$ be such that $M \sim c h^{-2} N_3$ for some small $c > 0$, we know that in order to get a non-trivial contribution, at least one of $u_j$'s must be equipped with the projector $\Qb_{> M}$.

If $\Qb_{>M}$ hits $u_3$, we use the Strichartz estimate in Corollary~\ref{COR:linU}~(i) on $u_2$, Lemma~\ref{LEM:QM}~(i) with $M \sim h^{-2} N_3$ on $u_3$, and the fact that $| \frac{2}{h} \sin (\frac{h \xi}{2}) | \leq \min (\frac{2}{h}, |\xi|) \les h^{- 1 + \gamma} |\xi|^{\gamma}$ to obtain
\begin{align*}
|\text{\eqref{rem4-0}}| \les h^\gamma T^{\frac 14} N_1^{- \frac 34 - s} N_1^s \| \Pb_{N_1} u_1 \|_{U^2 L^2} N_2^s \| \Pb_{N_2} u_2 \|_{U^2 L^2} N_3^{-s + \gamma} \| \Pb_{N_3} u_3 \|_{V^2 L^2} .
\end{align*}

\noi
The coefficient $N_1^{- \frac 34 - s}$ is summable given $- \frac 34 < s \leq 0$. The case when $\Qb_{> M}$ hits $u_1$ or when $\Qb_{> M}$ hits $u_2$ (for which we apply the Strichartz estimate on $u_1$) holds from a similar manner.

\smallskip \noi
{\bf Case 3:} $N_1 \sim N_2 \gg N_3 \gg 1$.

In this case, we exploit smoothing from the maximum modulation as in Case 2. 
If $\Qb_{> M}$ hits $u_3$, we use the bilinear estimate in Corollary~\ref{COR:bilinU}~(ii) on the product of $u_1$ and $u_2$, Lemma~\ref{LEM:QM}~(i) with $M \sim h^{-2} N_3$ on $u_3$, and the fact that $| \frac{2}{h} \sin (\frac{h \xi}{2}) | \leq |\xi|$ to obtain
\begin{align}
|\eqref{rem4-0}| \les h N_1^{- \frac 12 - 2s} N_3^{s - \gamma}  N_1^s \| \Pb_{N_1} u_1 \|_{U^2 L^2} N_2^s \| \Pb_{N_2} u_2 \|_{U^2 L^2} N_3^{-s + \gamma} \| \Pb_{N_3} u_3 \|_{V^2 L^2} .
\label{rem4-1}
\end{align}

\noi
Since $h \les N_1^{-1}$, we have $h N_1^{- \frac 12 - 2s} N_3^{s - \gamma} \les h^\gamma N_1^{- \frac 32 - 2 s + \gamma} N_3^{s - \gamma}$, which is summable given $- \frac 34 < s \leq 0$ and $0 \leq \gamma < \frac 32 + 2s$.
If $\Qb_{> M}$ hits $u_2$, we use the bilinear estimate in Corollary~\ref{COR:bilinU}~(iii) on the product of $u_1$ and $u_3$, Lemma~\ref{LEM:QM}~(i) with $M \sim h^{-2} N_3$ on $u_2$, and the fact that $| \frac{2}{h} \sin (\frac{h \xi}{2}) | \leq \min (\frac{2}{h}, |\xi|) \les h^{- \frac 12 + \gamma} |\xi|^{\frac 12 + \gamma}$ given $0 \leq \gamma < \frac 12$ to obtain
\begin{align*}
|\eqref{rem4-0}| \les h^{\frac 32 + \gamma} N_1^{- 2s} N_3^{s}  N_1^s \| \Pb_{N_1} u_1 \|_{U^2 L^2} N_2^s \| \Pb_{N_2} u_2 \|_{U^2 L^2} N_3^{-s + \gamma} \| \Pb_{N_3} u_3 \|_{U^2 L^2} .
\end{align*}

\noi
Since $h \les N_1^{-1}$, we have $h^{\frac 32} N_1^{ - 2s} N_3^{s} \les N_1^{- \frac 32 - 2 s} N_3^{s}$, which is summable given $- \frac 34 < s \leq 0$. The case when $\Qb_{> M}$ hits $u_1$ follows from the same way.

In all of the above situations, we need to use the interpolation in Lemma~\ref{LEM:interp} with a trivial estimate to create a factor $T^\ta$ and put a $V^2 L^2$-norm for $u_3$.

\smallskip \noi
{\bf Case 4:} $N_1 \sim N_2 \gg N_3 \sim 1$.

In this case, we apply a further dyadic decomposition $\Pb_{\leq 1} = \sum_{N_3' \leq 1} \Pb_{N_3'}$ and use the same steps as those in Case 3, but just with $|\frac{2}{h} \sin (\frac{h \xi}{2})| \les N_3'$ given $|\xi| \sim N_3'$ for $u_3$ so that the summation on $N_3' \leq 1$ becomes summable. 
The only issue is \eqref{rem4-1} as no factor of $N_3$ appears even with $|\frac{2}{h} \sin (\frac{h \xi}{2})| \les N_3'$. To deal with this term, we simply interpolate the estimate \eqref{rem4-1} with a trivial estimate with a factor $(N_3')^a$ for some $a > 0$ to obtain the desired estimate (see the proof of \eqref{help}).

\smallskip \noi
{\bf Case 5:} $N_1 \sim N_3 \gg N_2 \gg 1$.

In this case, we exploit smoothing from the maximum modulation as in Case 2.
If $\Qb_{>M}$ hits $u_3$, we use the Strichartz estimate in Corollary~\ref{COR:linU}~(i) on $u_2$, Lemma~\ref{LEM:QM}~(i) with $M \sim h^{-2} N_3$ on $u_3$, and the fact that $| \frac{2}{h} \sin (\frac{h \xi}{2}) | \leq \min (\frac{2}{h}, |\xi|) \les h^{- 1 + \gamma} |\xi|^\gamma$ to obtain
\begin{align*}
|\text{\eqref{rem4-0}}| \les h^\gamma T^{\frac 14} N_2^{- \frac 14 - s} N_3^{- \frac 12} N_1^s \| \Pb_{N_1} u_1 \|_{U^2 L^2} N_2^s \| \Pb_{N_2} u_2 \|_{U^2 L^2} N_3^{-s + \gamma} \| \Pb_{N_3} u_3 \|_{V^2 L^2} .
\end{align*}

\noi
Since $N_2^{- \frac 14 - s} N_3^{- \frac 12}\les N_2^{- \frac 34 - s}$ and $- \frac 34 < s \leq 0$, we obtain the desired estimate by summing up $N_2$ and using the Cauchy-Schwarz inequality on $N_1 \sim N_3 \gg 1$. The case when $\Qb_{> M}$ hits $u_1$ or when $\Qb_{> M}$ hits $u_2$ (for which we apply the Strichartz estimate on $u_1$) holds from a similar manner.

\smallskip \noi
{\bf Case 6:} $N_1 \sim N_3 \gg N_2 \sim 1$.

In this case, we again exploit smoothing from the maximum modulation as in Case 2.
If $\Qb_{>M}$ hits $u_3$, we use the Strichartz estimate in Corollary~\ref{COR:linU}~(i) on $u_1$, Lemma~\ref{LEM:QM}~(i) with $M \sim h^{-2} N_3$ on $u_3$, and the fact that $| \frac{2}{h} \sin (\frac{h \xi}{2}) | \leq \min (\frac{2}{h}, |\xi|) \les h^{- 1 + \gamma} |\xi|^\gamma$ to obtain
\begin{align*}
| \text{\eqref{rem4-0}} | \les h^\gamma T^{\frac 14} N_1^{- \frac 34} N_1^s \| \Pb_{N_1} u_1 \|_{U^2 L^2}  \| \Pb_{\les 1} u_2 \|_{U^2 L^2} N_3^{-s + \gamma} \| \Pb_{N_3} u_3 \|_{V^2 L^2} .
\end{align*}

\noi
We obtain the desired estimate using the Cauchy-Schwarz inequality on $N_1 \sim N_3 \gg 1$. The case when $\Qb_{> M}$ hits $u_2$ or when $\Qb_{> M}$ hits $u_1$ (for which we apply the Strichartz estimate on $u_3$ and then interpolate the resulting estimate with a trivial estimate using Lemma~\ref{LEM:interp} in order to put a $V^2 L^2$-norm for $u_3$) holds from a similar manner.

\smallskip \noi
{\bf Case 7:} $N_2 \sim N_3 \gg N_1 \gg 1$.

This case is symmetric to Case 5, and so we omit details.

\smallskip \noi
{\bf Case 8:} $N_2 \sim N_3 \gg N_1 \sim 1$.

This case is symmetric to Case 6, and so we omit details.
\end{proof}

We are left with the following term involving both shifted frequency by $\cos (\frac{2 \pi}{h} \cdot)$ and modulations with opposite signs.

\begin{lemma}
\label{LEM:rem5}
Let $0 < h \leq 1$ and $0 < T \leq 1$.
Then, for any $- \frac 34 < s \leq 0$, $0 \leq \gamma \leq 1 + s$, we have
\begin{align*}
\bigg| &\int_0^T \int_\R \big( e^{\pm \frac{t}{h^2} \nb_h} P_{\leq \frac{\pi}{h}} u_1 \big) \big( e^{\pm \frac{t}{h^2} \nb_h} P_{\leq \frac{\pi}{h}} u_2 \big) \big( \cos (\tfrac{2 \pi}{h} \cdot) \nb_h e^{\mp \frac{t}{h^2} \nb_h} P_{\leq \frac{\pi}{h}} u_3 \big) dx dt \bigg| \\
&\les h^\gamma T^{\frac 14} \| u_1 \|_{X_T^{s}} \| u_2 \|_{X_T^{s}} \| u_3 \|_{Y_T^{- s + \gamma}} .
\end{align*}
\end{lemma}
\begin{proof}
As before, we mainly look at
\begin{align}
\begin{split}
\int_\R &\int_\R \big( e^{\frac{t}{h^2} \nb_h} \Pb_{N_1} P_{\leq \frac{\pi}{h}} \ind_{[0, T)} u_1 \big) \big( e^{\frac{t}{h^2} \nb_h} \Pb_{N_2} P_{\leq \frac{\pi}{h}} \ind_{[0, T)} u_2 \big) \\
&\quad \times \big( e^{i \frac{2 \pi}{h} \cdot} \nb_h e^{- \frac{t}{h^2} \nb_h}  \Pb_{N_3} P_{\leq \frac{\pi}{h}} \ind_{[0, T)} u_3 \big)   dx dt 
\end{split}
\label{rem5-0}
\end{align}


\noi
with $N_1, N_2, N_3 \geq 1$ dyadic. Note that we must have $N_1, N_2, N_3 \les \frac{1}{h}$ to get a non-trivial contribution.

Let $\xi_j$ be the spatial frequency of $u_j$, $j = 1, 2, 3$. Then, we must have $\xi_1 + \xi_2 = \xi_3 + \frac{2 \pi}{h}$. Since $|\xi_j| \leq \frac{\pi}{h}$, we know that at least two $N_j$'s satisfy $N_j \sim \frac{1}{h}$.

We now consider the following cases.

\smallskip \noi
{\bf Case 1:} $N_1 \sim N_2 \sim \frac{1}{h} \ges N_3 \gg 1$.

In this case, we exploit smoothing from the maximum modulation.
Let $(\tau_j, \xi_j)$ be the temporal-spatial frequency of $e^{\frac{t}{h^2} \nb_h} P_{\leq \frac{\pi}{h}} \ind_{[0, T)} u_j$, $j = 1, 2$, and let $(\tau_3, \xi_3)$ be the temporal-spatial frequency of $e^{- \frac{t}{h^2} \nb_h} P_{\leq \frac{\pi}{h}} \ind_{[0, T)} u_3$.
Then, we must have $\tau_1 + \tau_2 + \tau_3 = 0$, $\xi_1 + \xi_2 + \xi_3 + \frac{2 \pi}{h} = 0$, $|\xi_j| \sim N_j$ for $j = 1, 2, 3$, and $|\xi_j| \leq \frac{\pi}{h}$ for $j = 1, 2, 3$. Note that $\ind_{[0, T)} u_1$ has temporal frequency $\tau_1 - \frac{2}{h^3} \sin (\frac{h \xi_1}{2})$, $\ind_{[0, T)} u_2$ has temporal frequency $\tau_2 - \frac{2}{h^3} \sin (\frac{h \xi_2}{2})$, and $\ind_{[0, T)} u_3$ has temporal frequency $\tau_3 + \frac{2}{h^3} \sin (\frac{h \xi_3}{2})$. Then, we have
\begin{align*}
&\max \big( \big| \tau_1 - \tfrac{2}{h^3} \sin (\tfrac{h \xi_1}{2}) \big|, \big| \tau_2 - \tfrac{2}{h^3} \sin (\tfrac{h \xi_2}{2}) \big|, \big| \tau_3 + \tfrac{2}{h^3} \sin (\tfrac{h \xi_3}{2}) \big| \big)  \\
&\ges \big| ( \tau_1 - \tfrac{2}{h^3} \sin (\tfrac{h \xi_1}{2}) ) + ( \tau_2 - \tfrac{2}{h^3} \sin (\tfrac{h \xi_2}{2}) ) + ( \tau_3 + \tfrac{2}{h^3} \sin (\tfrac{h \xi_3}{2}) ) \big| \\
&= \tfrac{2}{h^3} \big| - \sin (\tfrac{h \xi_1}{2}) - \sin (\tfrac{h \xi_2}{2}) + \sin (\tfrac{h (\xi_1 + \xi_2)}{2}) \big| \\
&= \tfrac{8}{h^3} \big| \sin (\tfrac{h \xi_1}{4}) \sin (\tfrac{h \xi_2}{4}) \cos (\tfrac{h \xi_3}{4}) \big| \\
&\sim h^{-1} N_1 N_2 ,
\end{align*} 

\noi
Thus, by letting $M \in 2^\Z$ be such that $M \sim c h^{-1} N_1 N_2$ for some small $c > 0$, we know that in order to get a non-trivial contribution, at least one of $u_j$'s must be equipped with the projector $\Qb_{>M}$. 

If $\Qb_{> M}$ hits $u_3$, we use the Strichartz estimate in Corollary~\ref{COR:linU}~(i) on $u_2$, Lemma~\ref{LEM:QM}~(i) with $M \sim h^{-1} N_1 N_2$ on $u_3$, and the fact that $|\frac{2}{h} \sin (\frac{h \xi}{2})| \leq \min (\frac{2}{h}, |\xi|) \les h^{- 1 - s + \gamma} |\xi|^{- s + \gamma}$ given $0 \leq \gamma \leq 1 + s$ to obtain
\begin{align*}
| \text{\eqref{rem5-0}} | \les h^{- \frac 12 - s + \gamma} T^{\frac 14} N_1^{- \frac 54 - 2s} N_1^s \| \Pb_{N_1} u_1 \|_{U^2 L^2} N_2^s \| \Pb_{N_2} u_2 \|_{U^2 L^2} N_3^{- s + \gamma} \| \Pb_{N_3} u_3 \|_{V^2 L^2} .
\end{align*}

\noi
Since $h \sim N_1^{-1}$, we have $h^{- \frac 12 - s} N_1^{- \frac 54 - 2 s} \sim N_1^{- \frac 34 - s}$, which is summable given $- \frac 34 < s \leq 0$. 
The case when $\Qb_{> M}$ hits $u_1$ or when $\Qb_{> M}$ hits $u_2$ (for which we apply the Strichartz estimate on $u_1$) holds from a similar manner.

\smallskip \noi
{\bf Case 2:} $N_1 \sim N_2 \sim \frac 1h$ and $N_3 \sim 1$.

In this case, we apply a further dyadic decomposition $\Pb_{\leq 1} = \sum_{N_3' \leq 1} \Pb_{N_3'}$ and use the same steps as those in Case 1, but just with $|\frac{2}{h} \sin (\frac{h \xi}{2})| \les N_3'$ given $|\xi| \sim N_3'$ for $u_3$ so that the summation on $N_3' \leq 1$ becomes summable.

\smallskip \noi
{\bf Case 3:} $N_1 \sim N_3 \sim \frac{1}{h}\gg N_2 \gg 1$.

In this case, we exploit smoothing from the maximum modulation as in Case 1. 
If $\Qb_{> M}$ hits $u_3$, then we use the Strichartz estimate in Corollary~\ref{COR:linU}~(i) on $u_2$, Lemma~\ref{LEM:QM}~(i) with $M \sim h^{-1} N_1 N_2$ on $u_3$, and the fact that $|\frac{2}{h} \sin (\frac{h \xi}{2})| \leq \min (\frac{2}{h}, |\xi|) \les h^{- 1 + \gamma} |\xi|^{\gamma}$ to obtain
\begin{align*}
| \text{\eqref{rem5-0}} | \les h^{- \frac 12 + \gamma} T^{\frac 14} N_1^{- \frac 12} N_2^{- \frac 34 - s} N_1^s \| \Pb_{N_1} u_1 \|_{U^2 L^2} N_2^s \| \Pb_{N_2} u_2 \|_{U^2 L^2} N_3^{-s + \gamma} \| \Pb_{N_3} u_3 \|_{V^2 L^2} .
\end{align*}

\noi
Since $h \sim N_1^{-1}$, we have $h^{- \frac 12} N_1^{- \frac 12} N_2^{- \frac 34 - s} \sim N_2^{- \frac 34 - s}$, which is summable given $-\frac 34 < s \leq 0$. The case when $\Qb_{> M}$ hits $u_1$ or when $\Qb_{> M}$ hits $u_2$ (for which we apply the Strichartz estimate on $u_1$) holds from a similar manner.

\smallskip \noi
{\bf Case 4:} $N_1 \sim N_3 \sim \frac 1h$ and $N_2 \sim 1$.

In this case, we use the local smoothing estimate in Corollary~\ref{COR:linU}~(iii) on $u_1$, the maximal function estimate in Corollary~\ref{COR:linU}~(ii) on $u_2$, and the fact that $|\frac{2}{h} \sin (\frac{h \xi}{2})| \leq \min (\frac{2}{h}, |\xi|) \les h^{- 1 + \gamma} |\xi|^{\gamma}$ to obtain
\begin{align*}
| \text{\eqref{rem5-0}} | \les h^{-1 + \gamma} T^{\frac 12} N_1^{-1} N_1^s \| \Pb_{N_1} u_1 \|_{U^2 L^2}  \| \Pb_{\les 1} u_2 \|_{U^2 L^2} N_3^{-s + \gamma} \| \Pb_{N_3} u_3 \|_{V^2 L^2} .
\end{align*}

\noi
By taking $h^{-1} N_1^{-1} \sim 1$, we obtain the desired estimate.

\smallskip \noi
{\bf Case 5:} $N_2 \sim N_3 \sim \frac 1h \gg N_1 \gg 1$.

This case is symmetric to Case 3, and so we omit details.

\smallskip \noi
{\bf Case 6:} $N_2 \sim N_3 \sim \frac 1h$ and $N_1 \sim 1$.

This case is symmetric to Case 4, and so we omit details.
\end{proof}

Finally, we consider the last object in $F(w_\pm^h)$ defined in \eqref{FR} involving the remainder term $\mathcal{R}$. 

\begin{lemma}
\label{LEM:rem6}
Let $0 < h \leq 1$, $0 < T \leq 1$. Let $u_\pm$, $u_{\pm, 1}$, and $u_{\pm, 2}$ be functions defined on $\R_+ \times \R$ and $v^h$, $v^h_1$, and $v^h_2$ be functions defined on $\R_+ \times h \Z$ satisfying 
\begin{align}
\begin{split}
\EE v^h (t) &= e^{- \frac{t}{h^2} \nb_h} P_{\leq \frac{\pi}{h}} u_+ (t) + e^{\frac{t}{h^2} \nb_h} P_{\leq \frac{\pi}{h}} u_- (t), \\
\EE v_1^h (t) &= e^{- \frac{t}{h^2} \nb_h} P_{\leq \frac{\pi}{h}} u_{+, 1} (t) + e^{\frac{t}{h^2} \nb_h} P_{\leq \frac{\pi}{h}} u_{-, 1} (t), \\
\EE v_2^h (t) &= e^{- \frac{t}{h^2} \nb_h} P_{\leq \frac{\pi}{h}} u_{+, 2} (t) + e^{\frac{t}{h^2} \nb_h} P_{\leq \frac{\pi}{h}} u_{-, 2} (t).
\end{split}
\label{vdecomp}
\end{align}

\noi
Let $u$ be a function defined on $\R_+ \times \R$. Then, for any $- \frac 34 < s \leq 0$ and $0 \leq \gamma \leq \frac 34 + s$, we have
\begin{align}
\begin{split}
\bigg| &\int_0^T \int_\R h^2 \mathcal{E} \big( \mathcal{R} (v^h (t)) \big) \big( \nb_h e^{\pm \frac{t}{h^2} \nb_h} P_{\leq \frac{\pi}{h}} u \big) dx dt \bigg| \\
&\les h^\gamma T^{\frac 12} \big( \| u_+ \|_{X_T^s}^3 + \| u_- \|_{X_T^s}^3 \big) \| u \|_{Y_T^{- s + \gamma}} \sup_{|R| \les h^{\frac 32 + s} (\| u_+ \|_{L_T^\infty H^s (\R)} + \| u_- \|_{L_T^\infty H^s (\R)})} |V^{(4)} (R)| 
\end{split}
\label{rem6_goal1}
\end{align}

\noi
and
\begin{align}
\begin{split}
&\bigg| \int_0^T \int_\R h^2 \Big( \mathcal{E} \big( \mathcal{R} (v^h_1 (t)) \big) - \mathcal{E} \big( \mathcal{R} (v^h_2 (t)) \big) \Big) \big( \nb_h e^{\pm \frac{t}{h^2} \nb_h} P_{\leq \frac{\pi}{h}} u \big) dx dt \bigg| \\
&\les T^{\frac 12} \big( \| u_{+, 1} - u_{+, 2} \|_{X_T^s} + \| u_{-, 1} - u_{-, 2} \|_{X_T^s} \big) \\
&\quad \times \big( 1 + \| u_{+, 1} \|_{X_T^s}^3 + \| u_{+, 2} \|_{X_T^s}^3 + \| u_{-, 1} \|_{X_T^s}^3 + \| u_{-, 2} \|_{X_T^s}^3 \big) \| u \|_{Y_T^{-s}} \\
&\quad \times \sup_{|R| \les h^{\frac 32 + s} \sum_{j = 1}^2 ( \| u_{+, j} \|_{L_T^\infty H^s (\R)} + \| u_{-, j} \|_{L_T^\infty H^s (\R)} )} \big( |V^{(4)} (R)| + |V^{(5)} (R)| \big) .
\end{split}
\label{rem6_goal2}
\end{align}
\end{lemma}
\begin{proof}
We first consider \eqref{rem6_goal1}. As usual, we ignore the subscripts $T$ on the right-hand side of \eqref{rem6_goal1}. 
By using \eqref{defR} and applying Lemma~\ref{LEM:Eprod} twice, we get
\begin{align*}
\text{LHS of \eqref{rem6_goal1}} = h^2 \bigg| &\int_0^T \int_\R P_{\leq \frac{\pi}{h}} \bigg( (1 + 2 \cos (\tfrac{2 \pi}{h} \cdot)) P_{\leq \frac{\pi}{h}} \big( (1 + 2 \cos (\tfrac{2 \pi}{h} \cdot) ) \EE v^h \EE v^h \big) \\
&\quad \times  \EE \bigg( v^h  \int_0^1 (1 - \ta)^2 V^{(4)} (\ta h^2 v^h) d \ta \bigg) \bigg)  \big( \nb_h e^{\pm \frac{t}{h^2} \nb_h} P_{\leq \frac{\pi}{h}} u \big) dx dt \bigg| .
\end{align*}

\noi
By Lemma~\ref{LEM:Hs_equi}, H\"older's inequality, and Lemma~\ref{LEM:embh}, we have
\begin{align*}
\bigg\| \EE \bigg( v^h \int_0^1 (1 - \ta)^2 V^{(4)} (\ta h^2 v^h) d \ta \bigg) \bigg\|_{L^2 (\R)} &\leq \| v^h \|_{L^2 (h \Z)} \bigg\| \int_0^1 (1 - \ta)^2 V^{(4)} (\ta h^2 v^h) d \ta \bigg\|_{L^\infty (h \Z)} \\
&\leq \| \EE v^h \|_{L^2 (\R)} \sup_{0 \leq \ta \leq 1} \| V^{(4)} (\ta h^2 v^h) \|_{L^\infty (h \Z)} \\
&\leq \| \EE v^h \|_{L^2 (\R)} \sup_{|R| \leq h^{\frac 32} \| v^h \|_{L_T^\infty L^2 (h \Z)}} |V^{(4)} (R)| 
\end{align*}

\noi
for any $0 \leq \ta \leq 1$, so that by H\"older's inequalities, we get
\begin{align}
\begin{split}
\text{LHS of \eqref{rem6_goal1}} &\les h^2 T^{\frac 12} \| \EE v^h \EE v^h \|_{L_T^4 L_x^2} \bigg\| \EE \bigg( v^h \int_0^1 (1 - \ta)^2 V^{(4)} (\ta h^2 v^h) d \ta \bigg) \bigg\|_{L_T^\infty L_x^2} \\
&\quad \times \big\| \nb_h e^{\pm \frac{t}{h^2} \nb_h} P_{\leq \frac{\pi}{h}} u \big\|_{L_T^4 L_x^\infty} \\
&\les h^2 T^{\frac 12} \| \EE v^h \|_{L_T^4 L_x^\infty} \| \EE v^h \|_{L_T^\infty L_x^2}^2 \big\| \nb_h e^{\pm \frac{t}{h^2} \nb_h} P_{\leq \frac{\pi}{h}} u \big\|_{L_T^4 L_x^\infty} \\
&\quad \times \sup_{|R| \leq h^{\frac 32} \| v^h \|_{L_T^\infty L^2 (h \Z)}} |V^{(4)} (R)| .
\end{split}
\label{rem6-1}
\end{align}

\noi
Note that by using the dyadic decomposition \eqref{PNdecomp}, Sobolev's embedding, the Strichartz estimate in Corollary~\ref{COR:linU}~(i), \eqref{V2Ys}, Lemma~\ref{LEM:UVemb}, and the Cauchy-Schwarz inequality in $N$,  we have
\begin{align}
\begin{split}
\big\| e^{\mp \frac{t}{h^2} \nb_h} P_{\leq \frac{\pi}{h}} u_\pm \big\|_{L_T^4 L_x^\infty} 
&\leq \big\| e^{\mp \frac{t}{h^2} \nb_h} \Pb_{\leq 1} P_{\leq \frac{\pi}{h}} u_\pm \big\|_{L_T^4 L_x^\infty} + \sum_{N \in 2^\Z}  \big\| e^{\mp \frac{t}{h^2} \nb_h} \Pb_{N} P_{\leq \frac{\pi}{h}} u_\pm \big\|_{L_T^4 L_x^\infty} \\
&\les T^{\frac 14} \| u_\pm \|_{L_T^\infty H_x^s} + \sum_{\substack{N \in 2^\Z \\ 2 \leq N \les \frac{\pi}{h}}} N^{- \frac 14} \| \Pb_N u_\pm \|_{U^4 L^2} \\
&\les \| u_\pm \|_{X^s} + h^{\min (\frac 14 + s, 0)} \sum_{\substack{N \in 2^\Z \\ 2 \leq N \les \frac{\pi}{h}}} N^{s} \| \Pb_N u_\pm \|_{U^2 L^2} \\
&\les h^{\min (\frac 14 + s, 0)} |\log h|^{\frac 12} \| u_\pm \|_{X^s} .
\end{split}
\label{L4t1}
\end{align}

\noi
Using similar steps along with $|\frac{2}{h} \sin (\frac{h \xi}{2})| \leq \min (\frac{2}{h}, |\xi|) \les h^{- \frac 34 - s + \gamma} |\xi|^{\frac 14 - s + \gamma}$, which is valid due to $0 \leq \gamma \leq \frac 34 + s$, we have
\begin{align}
\big\| \nb_h e^{\pm \frac{t}{h^2} \nb_h} P_{\leq \frac{\pi}{h}} u \big\|_{L_T^4 L_x^\infty} \les h^{- \frac 34 - s + \gamma} |\log h|^{\frac 12} \| u \|_{Y^{-s + \gamma}} .
\label{L4t2}
\end{align}

\noi
Moreover, by using Lemma~\ref{LEM:Hs_equi}, \eqref{vdecomp}, and the fact that $\EE v^h$ has frequency support on $[-\frac{\pi}{h}, \frac{\pi}{h}]$, we have
\begin{align}
\| v^h \|_{L_T^\infty L^2 (h \Z)} \sim \| \EE v^h \|_{L_T^\infty L^2 (\R)} \les h^{s} (\| u_+ \|_{L_T^\infty H^s (\R)} + \| u_- \|_{L_T^\infty H^s (\R)}) .
\label{rem6-2}
\end{align}

\noi
Thus, from \eqref{rem6-1}, \eqref{vdecomp}, \eqref{L4t1}, \eqref{L4t2}, and \eqref{rem6-2}, we get
\begin{align*}
\text{LHS of \eqref{rem6_goal1}} &\les h^{\frac 54 + s + \min (\frac 14 + s, 0) + \gamma} |\log h| T^{\frac 12} \big( \| u_+ \|_{X^s}^3 + \| u_- \|_{X^s}^3 \big) \| u \|_{Y^{- s + \gamma}} \\
&\quad \times \sup_{|R| \les h^{\frac 32 + s} (\| u_+ \|_{L_T^\infty H^s (\R)} + \| u_- \|_{L_T^\infty H^s (\R)})} |V^{(4)} (R)| .
\end{align*}

\noi
Since $\frac 54 + s + \min (\frac 14 + s, 0) > 0$ given $- \frac 34 < s \leq 0$, we obtain the desired estimate \eqref{rem6_goal1}.

For \eqref{rem6_goal2}, we proceed in the same way as above, but we need to deal with an additional term $V^{(4)} (\ta h^2 v_1^h) - V^{(4)} (\ta h^2 v_2^h)$ with $0 \leq \ta \leq 1$ in $L^\infty (h \Z)$. By using the mean value theorem and Lemma~\ref{LEM:embh}, we have
\begin{align*}
\big| V^{(4)} (\ta h^2 v_1^h) - V^{(4)} (\ta h^2 v_2^h) \big| \leq h^2 |v_1^h - v_2^h| \sup_{|R| \les h^{2} \max (\| v_1^h \|_{L_T^\infty L^\infty (h \Z)} , \| v_2^h \|_{L_T^\infty L^\infty (h \Z)} )} |V^{(5)} (R)| 
\end{align*}

\noi
for any $0 \leq \ta \leq 1$. 
Then, by Lemma~\ref{LEM:embh}, Lemma~\ref{LEM:Hs_equi}, \eqref{vdecomp}, and the fact that $\EE v_1^h$ and $\EE v_2^h$ have frequency supports on $[-\frac{\pi}{h}, \frac{\pi}{h}]$, we have
\begin{align*}
h^2 \| v_1^h - v_2^h \|_{L_T^\infty L^\infty (h \Z)} &\leq h^{\frac 32} \| v_1^h - v_2^h \|_{L_T^\infty L^2 (h \Z)} \\
&\sim h^{\frac 32} \| \EE v_1^h - \EE v_2^h \|_{L_T^\infty L^2 (\R)} \\
&\les h^{\frac 32 + s} (\| u_{+, 1} - u_{+, 2} \|_{L_T^\infty H^s (\R)} + \| u_{-, 1} - u_{-, 2} \|_{L_T^\infty H^s (\R)}) 
\end{align*}

\noi
and a similar estimate applies to $\| v_1^h \|_{L_T^\infty L^\infty (h \Z)}$ and $\| v_2^h \|_{L_T^\infty L^\infty (h \Z)}$. Since $h^{\frac 32 + s} \leq 1$ given $- \frac 34 < s \leq 0$, we obtain the desired difference estimate \eqref{rem6_goal2}.
\end{proof}

\section{Convergence of dynamics}
\label{SEC:conv}

\subsection{A priori bounds for solutions in critical function spaces}

Let us recall the definition of $w_\pm^h$ in \eqref{defwh}. From Proposition~\ref{PROP:FPUGWP}, \eqref{rhscale}, Lemma~\ref{LEM:embhs}, \eqref{defwh}, and Lemma~\ref{LEM:Hs_equi}, we deduce that if $0 < h \leq 1$ satisfies the condition \eqref{hcond}, then $w_\pm^h$ is the unique solution to \eqref{whDuh} in the space $C (\R_+; H^s (\R))$ for any $s \in \R$. 
Our goal in this subsection is to establish the following uniform-in-$h$ bounds on the $X^s$-norm (see \eqref{defXs}) of $w^h_\pm$.

\begin{proposition}
\label{PROP:whbdd}
Let $0 < h \leq 1$, $- \frac 34 < s \leq 0$, $r_0^h \in H^s (h \Z)$, and $|\nb_h|^{-1}  r_1^h \in H^{s} (h \Z)$. Suppose that the condition \eqref{hcond} holds and the condition \eqref{r0cond1-1} for $r_0^h$ and $r_1^h$ holds for some constant $C_1 > 0$. 
Let $w^h_\pm$ be defined in \eqref{defwh}, which is the unique solution to \eqref{whDuh} in the space $C([0, T]; H^s (\R))$ with initial data $r_{\pm, 0}^h$ defined in \eqref{rhinit} in terms of $r_0^h$ and $r_1^h$.
Then, there exists $T = T (C_1, M_{V}) > 0$, with $M_{V} > 0$ defined in \eqref{defMV}, such that we have the bound
\begin{align}
\| w^h_\pm \|_{X_T^s} \leq 2 \| r_{\pm, 0}^h \|_{H^s (\R)} \leq C_1.
\label{whbdd_goal1}
\end{align}
\end{proposition}
\begin{proof}
Let $0 < T_0 \leq 1$ be chosen later. 
By defining the right-hand side of \eqref{whDuh} as $\Gamma (w_\pm^h)$, we have
\begin{align}
\| \Gamma (w^h_\pm) \|_{X_{T_0}^s} \leq \| r_{\pm, 0}^h \|_{X_{T_0}^s} + \sum_{j = 1}^7 \textup{I}_j ,
\label{whbdd0}
\end{align}

\noi
where
\begin{align*}
\textup{I}_1 &:= \bigg\| \int_0^\cdot e^{\pm \frac{t'}{h^2} \nb_h} \nb_h P_{\leq \frac{\pi}{h}} \big( ( e^{\mp \frac{t'}{h^2} \nb_h} w_\pm^h (t') )^2 \big) dt' \bigg\|_{X_{T_0}^s} , \\
\textup{I}_2 &:= \bigg\| \int_0^\cdot e^{\pm \frac{t'}{h^2} \nb_h} \nb_h P_{\leq \frac{\pi}{h}} \big( \cos (\tfrac{2 \pi}{h} \cdot) ( e^{\mp \frac{t'}{h^2} \nb_h} w_\pm^h (t') )^2 \big) dt' \bigg\|_{X_{T_0}^s} , \\
\textup{I}_3 &:= \bigg\| \int_0^\cdot e^{\pm \frac{t'}{h^2} \nb_h} \nb_h P_{\leq \frac{\pi}{h}} \big( e^{\mp \frac{t'}{h^2} \nb_h} w_\pm^h (t') e^{\pm \frac{t'}{h^2} \nb_h} w_\mp^h (t') \big) dt' \bigg\|_{X_{T_0}^s} , \\
\textup{I}_4 &:= \bigg\| \int_0^\cdot e^{\pm \frac{t'}{h^2} \nb_h} \nb_h P_{\leq \frac{\pi}{h}} \big( \cos (\tfrac{2 \pi}{h} \cdot) e^{\mp \frac{t'}{h^2} \nb_h} w_\pm^h (t') e^{\pm \frac{t'}{h^2} \nb_h} w_\mp^h (t') \big) dt' \bigg\|_{X_{T_0}^s} , \\
\textup{I}_5 &:= \bigg\| \int_0^\cdot e^{\pm \frac{t'}{h^2} \nb_h} \nb_h P_{\leq \frac{\pi}{h}} \big( ( e^{\pm \frac{t'}{h^2} \nb_h} w_\mp^h (t') )^2 \big) dt' \bigg\|_{X_{T_0}^s} , \\
\textup{I}_6 &:= \bigg\| \int_0^\cdot e^{\pm \frac{t'}{h^2} \nb_h} \nb_h P_{\leq \frac{\pi}{h}} \big( \cos (\tfrac{2 \pi}{h} \cdot) ( e^{\pm \frac{t'}{h^2} \nb_h} w_\mp^h (t') )^2 \big) dt' \bigg\|_{X_{T_0}^s} , \\
\textup{I}_7 &:= \bigg\| \int_0^\cdot e^{\pm \frac{t'}{h^2} \nb_h} \nb_h P_{\leq \frac{\pi}{h}} \big( h^2 \mathcal{E} \big( \mathcal{R} (r^h (t')) \big) \big) dt' \bigg\|_{X_{T_0}^s} 
\end{align*}

\noi
with $r^h$ viewed in terms of $w_\pm^h$ via \eqref{rhwh}.

We first note that since the initial data $r_{\pm, 0}^h$ is constant in time, we use \eqref{rhinit}, Lemma~\ref{LEM:Hs_equi}, and the condition \eqref{r0cond1-1} to obtain
\begin{align}
\| r_{\pm, 0}^h \|_{X_{T_0}^s} \leq \| r_{\pm, 0}^h \|_{H^s (\R)} \leq \frac 12 \| r_0^h \|_{H^s (h \Z)} + \frac 12 \| h^2 |\nb_h|^{-1} r_1^h \|_{H^s (h \Z)} \leq \frac 12 C_1 .
\label{I0est}
\end{align}

\noi
For $\textup{I}_1$, using the duality in Lemma~\ref{LEM:XYdual} and the trilinear estimate \eqref{tri_goal} in Lemma~\ref{LEM:tri} with $s_1 = s_2 = s$, we get
\begin{align}
\textup{I}_1 \leq \sup_{\| v \|_{Y^{-s}} \leq 1} \bigg| \int_0^{T_0} \int_\R ( e^{\mp \frac{t'}{h^2} \nb_h} w_\pm^h (t') )^2 \big( \nb_h e^{\mp \frac{t'}{h^2} \nb_h} P_{\leq \frac{\pi}{h}} v \big) dx dt' \bigg| \les T_0^\ta \| w_\pm^h \|_{X_{T_0}^s}^2
\label{I1est}
\end{align}

\noi
for some $\ta > 0$. 
Similarly, using the duality in Lemma~\ref{LEM:XYdual} and the trilinear estimates in Lemma~\ref{LEM:rem1}, Lemma~\ref{LEM:rem2}, Lemma~\ref{LEM:rem3}, Lemma~\ref{LEM:rem4}, and Lemma~\ref{LEM:rem5} with $\gamma = 0$, we obtain
\begin{align}
\sum_{j = 2}^6 \textup{I}_j &\les T_0^{\ta} \big( \| w_\pm^h \|_{X_{T_0}^s}^2 + \| w_\pm^h \|_{X_{T_0}^s} \| w_\mp^h \|_{X_{T_0}^s} \big) .
\label{I26est} 
\end{align}

\noi
For $\textup{I}_7$, we use the duality in Lemma~\ref{LEM:XYdual} and \eqref{rem6_goal1} in Lemma~\ref{LEM:rem6} with $\gamma = 0$ to obtain
\begin{align}
\textup{I}_7 \les T_0^{\frac 12} \big( \| w_+^h \|_{X_{T_0}^s}^3 + \| w_-^h \|_{X_{T_0}^s}^3 \big) \sup_{|R| \les h^{\frac 32 + s} (\| w_+^h \|_{L_{T_0}^\infty H^s (\R)} + \| w_-^h \|_{L_{T_0}^\infty H^s (\R)})} |V^{(4)} (R)| .
\label{I7est}
\end{align}

\noi
Combining \eqref{whbdd0}, \eqref{I0est}, \eqref{I1est}, \eqref{I26est}, and \eqref{I7est} and using \eqref{V2Ys}, we obtain
\begin{align}
\begin{split}
\| \Gamma (w_\pm^h) \|_{X_{T_0}^s} &\leq \| r_{\pm, 0}^h \|_{H^s (\R)} + C T_0^\ta \big( 1 + \| w_+^h \|_{X_{T_0}^s}^3 + \| w_-^h \|_{X_{T_0}^s}^3 \big) \\
&\quad \times \bigg( 1 + \sup_{|R| \les h^{\frac 32 + s} (\| w_+^h \|_{X_{T_0}^s} + \| w_-^h \|_{X_{T_0}^s})} |V^{(4)} (R)| \bigg)
\end{split}
\label{Gw1}
\end{align}

\noi
for some constant $C > 0$.
By using similar steps with \eqref{rem6_goal2} in Lemma~\ref{LEM:rem6}, we obtain the following difference estimate:
\begin{align}
\begin{split}
\| &\Gamma (w_\pm^{h, 1}) - \Gamma (w_\pm^{h, 2}) \|_{X_{T_0}^s} \\
&\les T_0^\ta \big( \| w_+^{h, 1} - w_+^{h, 2} \|_{X_{T_0}^s} + \| w_-^{h, 1} - w_-^{h, 2} \|_{X_{T_0}^s} \big) \\
&\quad \times \big( 1 + \| w_+^{h, 1} \|_{X_{T_0}^s}^3 + \| w_+^{h, 2} \|_{X_{T_0}^s}^3 + \| w_-^{h, 1} \|_{X_{T_0}^s}^3 + \| w_-^{h, 2} \|_{X_{T_0}^s}^3 \big) \\
&\quad \times \bigg( 1 + \sup_{|R| \les  h^{\frac 32 + s} \sum_{j = 1}^2 (\| w_+^{h, j} \|_{X_{T_0}^s} + \| w_-^{h, j} \|_{X_{T_0}^s} )} \big( |V^{(4)} (R)| + |V^{(5)} (R)| \big) \bigg).
\end{split}
\label{Gw2}
\end{align}

\noi
Thus, from \eqref{Gw1}, \eqref{Gw2}, \eqref{I0est}, and \eqref{defMV}, by recalling the constant $M_V$ defined in \eqref{defMV} and taking $T_0 > 0$ small enough such that
\begin{align*}
T_0^\ta \ll (1 + C_1)^{-3} M_{V}^{-1} \leq (1 + 2 \| r_{\pm, 0}^h \|_{H^s (\R)} )^{-3} M_V^{-1} ,
\end{align*}

\noi
we obtain that $\Gamma$ is a contraction mapping on a ball in the space $X_{T_0}^s$ of radius $2 \| r_{\pm, 0}^h \|_{H^s (\R)}$, which gives a solution $w_\pm^h$ to \eqref{whDuh} in $X_{T_0}^s$. Since $w_\pm^h$ is the unique solution to \eqref{whDuh} in $C ([0, T_0]; H^s (\R))$ and $X_{T_0}^s$ embeds continuously in  $C ([0, T_0]; H^s (\R))$, we obtain the desired bound \eqref{whbdd_goal1}.
\end{proof}

We also have the following bound on the $X^s$-norm of the solution $w_\pm$ to the KdV equation \eqref{wDuh} in the interaction representation form.
\begin{proposition}
\label{PROP:wbdd}
Let $- \frac 34 < s \leq 0$ and $u_{\pm, 0} \in H^s (\R)$. 
Then, there exists $T = T (\| u_{\pm, 0} \|_{H^s}) > 0$ such that there exists a unique solution $w_\pm$ to \eqref{wDuh} in the space $X_T^s$ with initial data $u_{\pm, 0}$, satisfying the bound
\begin{align*}
\| w_\pm \|_{X_T^s} \leq 2 \| u_{\pm, 0} \|_{H^s (\R)} .
\end{align*}
\end{proposition}
\begin{proof}
The proof is similar and easier to that of Proposition~\ref{PROP:whbdd}. One only needs to invoke the trilinear estimate \eqref{tri_goal2} in Lemma~\ref{LEM:tri} and run a contraction argument. We omit the rest of the details.
\end{proof}

\subsection{Proof of the continuum limit for the FPU system}
\label{SUB:conv}

With the a priori bounds of the solution $w_\pm^h$ to the equation \eqref{whDuh} in Proposition~\ref{PROP:whbdd}, we can prove the main convergence statements in Theorem~\ref{THM:FPU}. As an intermediate step, we first show the following proposition on the convergence of $w_\pm^h$ to the solution $w_\pm$ to \eqref{wDuh} as $h \to 0$.
\begin{proposition}
\label{PROP:whconv}
Given $0 < h \leq 1$ and $s \in \R$, let $r_0^h \in H^s (h \Z)$, $|\nb_h|^{-1} r_1^h \in H^s (h \Z)$, and $u_{\pm, 0} \in H^s (\R)$.
Suppose that the condition \eqref{hcond} holds. Let $w^h_\pm$ be defined in \eqref{defwh}, which is the unique solution to \eqref{whDuh} in the space $C(\R_+; H^s (\R))$ with initial data $r_{\pm, 0}^h$ defined in \eqref{rhinit} in terms of $r_0^h$ and $r_1^h$.

\smallskip \noi
\textup{(i)} Let $- \frac 34 < s < 0$ and $0 < \dl \leq 1$ satisfying $\dl < \frac 32 + 2s$.
Suppose that $r_0^h$, $r_1^h$, and $u_{\pm, 0}$ satisfy the condition \eqref{r0cond1-1} for some constant $C_1 > 0$ and that $u_{\pm, 0}$ and $r_{\pm, 0}^h$ satisfy the condition \eqref{r0cond1-2} for some constant $C_2 > 0$. 
Then, there exists $T = T (C_1, M_{V})$, with $M_{V} > 0$ defined in \eqref{defMV}, such that, with $w_\pm$ being the unique solution to \eqref{wDuh} in the space $X_T^s$ with initial data $u_{\pm, 0}$ guaranteed by Proposition~\ref{PROP:wbdd}, we have
\begin{align}
\| w_\pm^h - w_\pm \|_{X_T^{s - \dl}} \leq \wt C h^{\frac 25 \dl} 
\label{whd_goal1}
\end{align}

\noi
for some constant $\wt C = \wt C (C_1, C_2, M_{V}) > 0$.

\smallskip \noi
\textup{(ii)} Let $s = 0$. Suppose that $r_0^h$, $r_1^h$, and $u_{\pm, 0}$ satisfy the condition \eqref{r0cond2-1} for some constant $C_1 > 0$ and that $u_{\pm, 0}$ and $r_{\pm, 0}^h$ satisfy the condition \eqref{r0cond2-2} for some constant $C_2 > 0$. 
Let $w_\pm$ be the unique solution to \eqref{wDuh} in the space $C(\R_+; L^2 (\R))$ with initial data $u_{\pm, 0}$. Let $T > 0$ be arbitrary. Then, we have
\begin{align}
\| w_\pm^h - w_\pm \|_{X_T^{-1}} \leq e^{C T} \wt C h^{\frac 25} 
\label{whd_goal2}
\end{align}

\noi
for some constants $C  = C(C_1, M_{V}) > 0$ and $\wt C = \wt C (C_1, C_2, M_{V}) > 0$.
\end{proposition}
\begin{proof}
We mainly prove part (i) and briefly discuss the change we need for part (ii) at the end of the proof. 
Let $0 < T_0 \leq 1$ be chosen later. From \eqref{whDuh} and \eqref{wDuh}, we have
\begin{align}
\| w_\pm^h - w_\pm \|_{X_{T_0}^{s - \dl}} \leq \| r_{\pm, 0}^h - u_{\pm, 0} \|_{X_{T_0}^{s - \dl}} + \sum_{j = 1}^7 \textup{II}_j ,
\label{whdiff0}
\end{align}

\noi
where
\begin{align*}
\textup{II}_1 &:= \bigg\| \int_0^\cdot e^{\pm \frac{t'}{h^2} \nb_h} \nb_h P_{\leq \frac{\pi}{h}} \big( ( e^{\mp \frac{t'}{h^2} \nb_h} w_\pm^h (t') )^2 \big) dt' - \int_0^\cdot e^{\pm \frac{t'}{24} \dx^3} \dx \big( ( e^{\mp \frac{t'}{24} \dx^3} w_\pm (t') )^2 \big) dt' \bigg\|_{X_{T_0}^{s - \dl}} , \\
\textup{II}_2 &:= \bigg\| \int_0^\cdot e^{\pm \frac{t'}{h^2} \nb_h} \nb_h P_{\leq \frac{\pi}{h}} \big( \cos (\tfrac{2 \pi}{h} \cdot) ( e^{\mp \frac{t'}{h^2} \nb_h} w_\pm^h (t') )^2 \big) dt' \bigg\|_{X_{T_0}^{s - \dl}} , \\
\textup{II}_3 &:= \bigg\| \int_0^\cdot e^{\pm \frac{t'}{h^2} \nb_h} \nb_h P_{\leq \frac{\pi}{h}} \big( e^{\mp \frac{t'}{h^2} \nb_h} w_\pm^h (t') e^{\pm \frac{t'}{h^2} \nb_h} w_\mp^h (t') \big) dt' \bigg\|_{X_{T_0}^{s - \dl}} , \\
\textup{II}_4 &:= \bigg\| \int_0^\cdot e^{\pm \frac{t'}{h^2} \nb_h} \nb_h P_{\leq \frac{\pi}{h}} \big( \cos (\tfrac{2 \pi}{h} \cdot) e^{\mp \frac{t'}{h^2} \nb_h} w_\pm^h (t') e^{\pm \frac{t'}{h^2} \nb_h} w_\mp^h (t') \big) dt' \bigg\|_{X_{T_0}^{s - \dl}} , \\
\textup{II}_5 &:= \bigg\| \int_0^\cdot e^{\pm \frac{t'}{h^2} \nb_h} \nb_h P_{\leq \frac{\pi}{h}} \big( ( e^{\pm \frac{t'}{h^2} \nb_h} w_\mp^h (t') )^2 \big) dt' \bigg\|_{X_{T_0}^{s - \dl}} , \\
\textup{II}_6 &:= \bigg\| \int_0^\cdot e^{\pm \frac{t'}{h^2} \nb_h} \nb_h P_{\leq \frac{\pi}{h}} \big( \cos (\tfrac{2 \pi}{h} \cdot) ( e^{\pm \frac{t'}{h^2} \nb_h} w_\mp^h (t') )^2 \big) dt' \bigg\|_{X_{T_0}^{s - \dl}} , \\
\textup{II}_7 &:= \bigg\| \int_0^\cdot e^{\pm \frac{t'}{h^2} \nb_h} \nb_h P_{\leq \frac{\pi}{h}} \big( h^2 \mathcal{E} \big( \mathcal{R} (r^h (t')) \big) \big) dt' \bigg\|_{X_{T_0}^{s - \dl}} 
\end{align*}

\noi
with $r^h$ viewed in terms of $w_\pm^h$ as in \eqref{rhwh}. For $\textup{II}_1$, we further decompose the difference and get
\begin{align}
\textup{II}_1 = \sum_{j = 1}^4 \text{II}_{1, j} ,
\label{whdiff00}
\end{align}

\noi
where
\begin{align*}
\textup{II}_{1, 1} &:= \bigg\| \int_0^\cdot e^{\pm \frac{t'}{h^2} \nb_h} (\nb_h - \dx) P_{\leq \frac{\pi}{h}} \big( ( e^{\mp \frac{t'}{h^2} \nb_h} w_\pm^h (t') )^2 \big) dt' \bigg\|_{X_{T_0}^{s - \dl}} , \\
\textup{II}_{1, 2} &:= \bigg\| \int_0^\cdot e^{\pm \frac{t'}{h^2} \nb_h} \dx P_{\leq \frac{\pi}{h}} \big( ( e^{\mp \frac{t'}{h^2} \nb_h} w_\pm^h (t') )^2 \big) dt' - \int_0^\cdot e^{\pm \frac{t'}{24} \dx^3} \dx P_{\leq \frac{\pi}{h}} \big( ( e^{\mp \frac{t'}{24} \dx^3} w_\pm^h (t') )^2 \big) dt' \bigg\|_{X_{T_0}^{s - \dl}} , \\
\textup{II}_{1, 3} &:= \bigg\| \int_0^\cdot e^{\pm \frac{t'}{24} \dx^3} \dx P_{> \frac{\pi}{h}} \big( ( e^{\mp \frac{t'}{24} \dx^3} w_\pm^h (t') )^2 \big) dt' \bigg\|_{X_{T_0}^{s - \dl}} , \\
\textup{II}_{1, 4} &:= \bigg\| \int_0^\cdot e^{\pm \frac{t'}{24} \dx^3} \dx \big( ( e^{\mp \frac{t'}{24} \dx^3} w_\pm^h (t') )^2 \big) dt' - \int_0^\cdot e^{\pm \frac{t'}{24} \dx^3} \dx \big( ( e^{\mp \frac{t'}{24} \dx^3} w_\pm (t') )^2 \big) dt' \bigg\|_{X_{T_0}^{s - \dl}} .
\end{align*}

\noi
By using the duality in Lemma~\ref{LEM:XYdual} and the trilinear difference estimates in Lemma~\ref{LEM:trid1}, Lemma~\ref{LEM:trid2}, and Lemma~\ref{LEM:trid3} with $\gamma = \dl \leq 1$ satisfying $\gamma < \frac 32 + 2 s$, we obtain
\begin{align}
\textup{II}_{1, 1} + \textup{II}_{1, 2} + \textup{II}_{1, 3} \les h^{\frac 25 \dl} \| w_\pm^h \|_{X_{T_0}^s}^2 .
\label{whdiff1}
\end{align}

\noi
For $\textup{II}_{1, 4}$, we use the duality in Lemma~\ref{LEM:XYdual} and the trilinear estimate \eqref{tri_goal2} in Lemma~\ref{LEM:tri} with $s_1 = s - \dl$ and $s_2 = s$ (where we also need $\dl < \frac 32 + 2s$ to guarantee $s_1 + s_2 > -\frac 32$) to obtain
\begin{align}
\textup{II}_{1, 4} \les T_0^{\ta} \| w_\pm^h - w_\pm \|_{X_{T_0}^{s - \dl}} \big( \| w_\pm^h \|_{X_{T_0}^s} + \| w_\pm \|_{X_{T_0}^s} \big) 
\label{whdiff2}
\end{align}

\noi
for some $\ta > 0$. For $\textup{II}_j$, $j = 2, \dots, 6$, we use the duality in Lemma~\ref{LEM:XYdual} and the trilinear estimates in Lemma~\ref{LEM:rem1}, Lemma~\ref{LEM:rem2}, Lemma~\ref{LEM:rem3}, Lemma~\ref{LEM:rem4}, and Lemma~\ref{LEM:rem5} (applied to $\gamma = \frac 12$ if $\dl > \frac 12$ and $\gamma = 1 + s$ if $\dl > 1 + s$) to obtain
\begin{align}
\sum_{j = 2}^6 \textup{II}_j \les h^{\min (\dl, \frac 12, 1 + s)} \big( \| w_\pm^h \|_{X_{T_0}^s}^2 + \| w_\pm^h \|_{X_{T_0}^s} \| w_\mp^h \|_{X_{T_0}^s} \big) .
\label{whdiff3}
\end{align}

\noi
For $\textup{II}_7$, we use the duality in Lemma~\ref{LEM:XYdual} and \eqref{rem6_goal1} in Lemma~\ref{LEM:rem6} (applied to $\gamma = \frac 34 + s$ if $\dl > \frac 34 + s$ in this proof) to obtain
\begin{align}
\textup{II}_7 \les h^{\min (\dl, \frac 34 + s)} \big( \| w_+^h \|_{X_{T_0}^s}^3 + \| w_-^h \|_{X_{T_0}^s}^3 \big) \sup_{|R| \les  h^{\frac 32 + s} (\| w_+^h \|_{L_{T_0}^\infty H^s (\R)} + \| w_-^h \|_{L_{T_0}^\infty H^s (\R)})} |V^{(4)} (R)| .
\label{whdiff4}
\end{align}

\noi
Combining \eqref{whdiff0}, \eqref{whdiff00}, \eqref{whdiff1}, \eqref{whdiff2}, \eqref{whdiff3}, and \eqref{whdiff4}, using the fact that $r_{\pm, 0}^h$ and $u_{\pm, 0}$ are constant in time, and recalling the constant $M_V$ defined in \eqref{defMV}, we obtain
\begin{align}
\begin{split}
\| &w_\pm^h - w_\pm \|_{X_{T_0}^{s - \dl}} \\
&\les \| r_{\pm, 0}^h - u_{\pm, 0} \|_{H^{s - \dl} (\R)} + T_0^{\ta} \| w_\pm^h - w_\pm \|_{X_{T_0}^{s - \dl}} \big( \| w_\pm^h \|_{X_{T_0}^s} + \| w_\pm \|_{X_{T_0}^s} \big) \\
&\quad + h^{\frac 25 \dl} \big( 1 + \| w_+^h \|_{X_{T_0}^s}^3 + \| w_-^h \|_{X_{T_0}^s}^3 \big) M_V ,
\end{split}
\label{whdiff5}
\end{align}

\noi
where we also used the fact that $\frac 25 \dl < \frac 34 + s < 1 + s$ and $\frac 25 \dl < \frac 12$ given $0 < \dl \leq 1$ and $\dl < \frac 32 + 2s$.
By using \eqref{V2Ys}, Proposition~\ref{PROP:whbdd} with $T_0 = T_0 (C_1, M_{V}) > 0$ sufficiently small with $M_{V} > 0$ defined in \eqref{defMV}, and Proposition~\ref{PROP:wbdd}, we get from \eqref{whdiff5} that
\begin{align}
\| w_\pm^h - w_\pm \|_{X_{T_0}^{s - \dl}} 
\les \| r_{\pm, 0}^h - u_{\pm, 0} \|_{H^{s - \dl} (\R)} + T_0^{\ta} C_1 \| w_\pm^h - w_\pm \|_{X_{T_0}^{s - \dl}} + h^{\frac 25 \dl} (1 + C_1^3) M_V .
\label{whdiff6}
\end{align}

\noi 
By further shrinking $T_0 > 0$ if necessary such that $T_0^{\ta} C_1 \ll 1$, we obtain from \eqref{whdiff6} that
\begin{align}
\| w_\pm^h - w_\pm \|_{X_{T_0}^{s - \dl}} 
&\les \| r_{\pm, 0}^h - u_{\pm, 0} \|_{H^{s - \dl} (\R)} + h^{\frac 25 \dl} (1 + C_1^3) M_{V} .
\label{whdiff7}
\end{align}

\noi
Thus, from \eqref{whdiff7}, by using the condition \eqref{r0cond1-2}, we obtain the desired difference result \eqref{whd_goal1} in part (i).

We now look at part (ii) with $s = 0$. From \eqref{defwh}, \eqref{rhpm}, Lemma~\ref{LEM:Hs_equi},  Corollary~\ref{COR:GWP}, and the condition \eqref{r0cond2-1}, we know that for any $t \in \R_+$, we have the bound
\begin{align}
\begin{split}
\| w^h_\pm (t) \|_{L^2 (\R)} 
&\leq \| \mathcal{E} r^h (t) \|_{L^2 (\R)} + \| h^2 |\nb_h|^{-1} \mathcal{E} \dt r^h (t) \|_{L^2 (\R)} \\
&\leq \| r^h (t) \|_{L^2 (h \Z)} + \| h^2 |\nb_h|^{-1} \dt r^h (t) \|_{L^2 (h \Z)} \\
&\leq 2 \| r_0^h \|_{L^2 (h \Z)} + 2 \| h^2 |\nb_h|^{-1} r_1^h \|_{L^2 (h \Z)} \\
&\leq 2 C_1 .
\end{split}
\label{whL2bdd}
\end{align}

\noi
We repeat the above procedure that lead to \eqref{whdiff5} on subsequent intervals and use Proposition~\ref{PROP:whbdd} and \eqref{whL2bdd} to obtain for any $j \in \N$ that
\begin{align*}
\| w_\pm^h - w_\pm \|_{X_{[j T_0, (j + 1) T_0]}^{-1}} &\les \| w_\pm^h (j T_0) - w_\pm (j T_0) \|_{H^{-1} (\R)} \\
&\quad + T_0^\ta \| w_\pm^h - w_\pm \|_{X_{[j T_0, (j + 1) T_0]}^{-1}} (\| w_\pm^h (j T_0) \|_{L^2 (\R)} + \| w_\pm (j T_0) \|_{L^2 (\R)} ) \\
&\quad + h^{\frac 25} (1 + \| w_+^h (j T_0) \|_{L^2 (\R)}^3 + \| w_-^h (j T_0) \|_{L^2 (\R)}^3) M_{V} \\
&\les \| w_\pm^h (j T_0) - w_\pm (j T_0) \|_{H^{-1} (\R)} + T_0^\ta C_1 \| w_\pm^h - w_\pm \|_{X_{[j T_0, (j + 1) T_0]}^{-1}} \\
&\quad + h^{\frac 25} (1 + C_1^3) M_{V} 
\end{align*}

\noi
for $T_0 = T_0 (C_1, M_{V}) > 0$ given as in Proposition~\ref{PROP:whbdd}.
By further shrinking $T_0 > 0$ if necessary such that $T_0^\ta C_1 \ll 1$, we get
\begin{align}
\| w_\pm^h - w_\pm \|_{X_{[j T_0, (j + 1) T_0]}^{-1}} \les \| w_\pm^h (j T_0) - w_\pm (j T_0) \|_{H^{-1} (\R)} + h^{\frac 25} (1 + C_1^3) M_{V} .
\label{induct}
\end{align}

\noi
Thus, with \eqref{induct} in hand, a simple inductive argument and a concatenation of all intervals then give the desired bound \eqref{whd_goal2} in part (ii).
\end{proof}

We are now ready to prove Theorem~\ref{THM:FPU}.
\begin{proof}[Proof of Theorem~\ref{THM:FPU}]
We only focus on the estimates for $\EE r^h$, as the estimates for $h^2 \nb_h^{-1} \EE \dt r^h$ follow from a similar manner. 
We first note that by \eqref{rhdecomp}, \eqref{wu}, \eqref{V2Ys}, Proposition~\ref{PROP:whbdd},   Proposition~\ref{PROP:wbdd}, and the condition \eqref{r0cond1-1}, we have
\begin{align*}
&\big\| \EE r^h - e^{- \frac{t}{h^2} \dx} u_+ - e^{\frac{t}{h^2} \dx} u_- \big\|_{C_T H^{s - \dl} (\R)} \\
&\leq \| w_+^h \|_{C_T H^s (\R)} + \| w_-^h \|_{C_T H^s (\R)} + \| w_+ \|_{C_T H^s (\R)} + \| w_- \|_{C_T H^s (\R)} \\
&\leq 100 \big( \| w_+^h \|_{X_T^s} + \| w_-^h \|_{X_T^s} + \| w_+ \|_{X_T^s} + \| w_- \|_{X_T^s} \big) \\
&\leq 200 C_1 ,
\end{align*}

\noi
which gives the uniform-in-$h$ bound in \eqref{conv_goal1}.
Also, by Lemma~\ref{LEM:Hs_equi}, Corollary~\ref{COR:GWP}, \eqref{mass}, and the condition \eqref{r0cond2-1}
\begin{align*}
&\big\| \EE r^h - e^{- \frac{t}{h^2} \dx} u_+ - e^{\frac{t}{h^2} \dx} u_- \big\|_{C_T H^{-1} (\R)} \\
&\leq \| r^h \|_{C_T L^2 (h \Z)} + \| u_+ \|_{C_T L^2 (\R)} + \| u_- \|_{C_T L^2 (\R)} \\
&\leq 2 C_1 ,
\end{align*}

\noi
which gives the uniform-in-$h$ bound in \eqref{conv_goal2}.

We now focus on the bounds in \eqref{conv_goal1} and \eqref{conv_goal2} with the convergence rates in $h$.
Let us prove these two estimates together with the understanding that $- \frac 34 < s < 0$, $0 < \dl \leq 1$ with $\dl < \frac 32 + 2s$ for part (i) and $s = 0$, $\dl = 1$ for part (ii). 
Let $T > 0$ be given as in Proposition~\ref{PROP:whconv}.
From \eqref{rhdecomp}, \eqref{wu}, and Plancherel's identity, we have
\begin{align}
\begin{split}
&\big\| \EE r^h - e^{- \frac{t}{h^2} \dx} u_+ - e^{\frac{t}{h^2} \dx} u_- \big\|_{C_T H^{s - \dl} (\R)} \\
&\leq \big\| e^{- \frac{t}{h^2} \nb_h} w_+^h - e^{- \frac{t}{h^2} \dx - \frac{t}{24} \dx^3} w_+ \big\|_{C_T H^{s - \dl} (\R)} + \big\| e^{\frac{t}{h^2} \nb_h} w_-^h - e^{\frac{t}{h^2} \dx + \frac{t}{24} \dx^3} w_- \big\|_{C_T H^{s - \dl} (\R)} \\
&\leq \big\| e^{- \frac{t}{h^2} \nb_h + \frac{t}{h^2} \dx + \frac{t}{24} \dx^3} w_+^h - w_+^h \big\|_{C_T H^{s - \dl} (\R)} + \| w_+^h - w_+ \|_{C_T H^{s - \dl} (\R)} \\
&\quad + \big\| e^{\frac{t}{h^2} \nb_h - \frac{t}{h^2} \dx - \frac{t}{24} \dx^3} w_-^h - w_-^h \big\|_{C_T H^{s - \dl} (\R)} + \| w_-^h - w_- \|_{C_T H^{s - \dl} (\R)}  .
\end{split} 
\label{rconv0}
\end{align}

\noi
By the mean value theorem, we get
\begin{align}
\Big| e^{\frac{i t}{h^3} (2 \sin (\frac{h \xi}{2}) - h \xi + \frac{(h \xi)^3}{24})} - 1 \Big| \les \min (t h^2 |\xi|^5, 1) \les t^{\frac 15 \dl} h^{\frac 25 \dl} |\xi|^\dl ,
\label{linconv}
\end{align}

\noi
so that by using Plancherel's identity, \eqref{linconv}, Proposition~\ref{PROP:whbdd} or the bound \eqref{whL2bdd} with the condition \eqref{r0cond2-1}, we obtain
\begin{align}
\big\| e^{\mp \frac{t}{h^2} \nb_h \pm \frac{t}{h^2} \dx \pm \frac{t}{24} \dx^3} w_\pm^h - w_\pm^h \big\|_{C_T H^{s - \dl} (\R)} 
\les T^{\frac 15 \dl} h^{\frac 25 \dl} \| w_\pm^h \|_{C_T H^s (\R)} 
\les T^{\frac 15 \dl} C_1 h^{\frac 25 \dl} .
\label{rconv1}
\end{align}

\noi
From \eqref{V2Ys} and Proposition~\ref{PROP:whconv}, we get
\begin{align}
\begin{split}
\| &w_+^h - w_+ \|_{C_T H^{s - \dl} (\R)} + \| w_-^h - w_- \|_{C_T H^{s - \dl} (\R)} \\
&\les \| w_+^h - w_+ \|_{X_T^{s - \dl}} + \| w_-^h - w_- \|_{X_T^{s - \dl}} \\
&\leq e^{C T} \wt C h^{\frac 25 \dl} 
\end{split}
\label{rconv2}
\end{align}

\noi
for some constants $C = C(C_1, M_{V}) > 0$ and $\wt C = \wt C (C_1, C_2, M_{V}) > 0$ with $M_{V} > 0$ defined in \eqref{defMV}.
Combining \eqref{rconv0}, \eqref{rconv1}, and \eqref{rconv2}, we obtain the desired difference estimates \eqref{conv_goal1} and \eqref{conv_goal2}. 
\end{proof}

\subsection{Proof of the continuum limit for the Toda lattice in Flaschka's form}
\label{SUB:toda2}

In this subsection, we prove Theorem~\ref{THM:toda}, the continuum limit for the Toda lattice in Flaschka's form. As mentioned earlier, we reduce the proof to that of Theorem~\ref{THM:FPU} by exploiting the connection between the variables of the scaled Toda lattice in Flaschka's form and the variables of the scaled FPU system.

Let us define the scaled versions of $\alpha$ and $\beta$ as
\begin{align*}
\al^h (t, \ld) := \frac{1}{h^2} \al \Big( \frac{t}{h^3}, \frac{\ld}{h} \Big) \quad \text{and} \quad \be^h (t, \ld) := \frac{1}{h^2} \be \Big( \frac{t}{h^3}, \frac{\ld}{h} \Big), \quad (t, \ld) \in \R_+ \times h \Z .
\end{align*}

\noi
Then, from \eqref{defab} and \eqref{rhscale}, we see that
\begin{align}
\al^h = \frac{1}{h^2} \exp \Big( \frac{h^2}{2} r^h \Big) - \frac{1}{h^2} \quad \text{and} \quad \be^h = \frac{h^2}{2} (\dh^+)^{-1} \dt r^h ,
\label{defabh}
\end{align}

\noi
where $\dh^+$ is the right discrete derivative defined in \eqref{dhpm}. Note that if $h^2 \al^h > -1$, then we have
\begin{align}
r^h = \frac{2}{h^2} \ln (1 + h^2 \alpha^h) .
\label{rhah}
\end{align}

\noi
We also see that $\gamma^h_\pm$ defined in \eqref{defgmh} can be written as
\begin{align}
\gamma_\pm^h (t, \ld) =
\begin{cases}
\alpha^h (t, \frac{\ld}{2}) & \text{if $\frac{\ld}{h}$ is even} \\
\mp \beta^h (t, \frac{\ld + h}{2}) & \text{if $\frac{\ld}{h}$ is odd},
\end{cases}
\quad (t, \ld) \in \R_+ \times h \Z .
\label{gmh}
\end{align}

Let us show the following key lemma establishing the relations among $\al^h$, $\be^h$, $r^h$, and $\dt r^h$. We state it using more general regularity assumptions, verifying the claim that we are able to show a local-in-time continuum limit for the Toda lattice in Flaschka's form with regularity $s > -\frac 34$.

\begin{lemma}
\label{LEM:gmdiff}
Given $0 < h \leq 1$, let $\al^h$ and $\be^h$ be given in \eqref{defabh} in terms of $r^h$ and $\dt r^h$. Let $t \geq 0$ and $- \frac 34 < s \leq 0$.

\smallskip \noi
\textup{(i)} If $h^2 \| \al^h (t) \|_{L^\infty (h \Z)} \leq \frac 12$, then we have
\begin{align}
\| r^h (t) - 2 \al^h (t) \|_{H^s (h \Z)} \les \| \al^h (t) \|_{H^s (h \Z)}^2
\label{gmab_goal1}
\end{align}

\noi
and
\begin{align}
\| h^2 |\nb_h|^{-1} \dt r^h (t) \|_{H^s (h \Z)} = 2 \| \be^h (t) \|_{H^s (h \Z)} .
\label{gmab_goal2}
\end{align}

\smallskip \noi
\textup{(ii)} Let $\gamma_\pm^h$ be given in \eqref{gmh}.
Then, for any $0 \leq \gamma \leq \frac 32 + 2s$, we have
\begin{align}
\begin{split}
\Big\| &\EE \gamma_\pm^h (t) - \frac 14 \Big( \EE r^h \Big( t, \frac{1}{2} \cdot \Big) \mp h^2 \nb_h^{-1} \EE \dt r^h \Big( t, \frac{1}{2} \cdot \Big) \Big) \Big\|_{H^{s - \gamma} (\R)} \\
&\les h^{\gamma} \exp \Big( \frac 12 \| r^h (t) \|_{H^s (h \Z)} \Big) + h^\gamma \big( \| r^h (t) \|_{H^s (h \Z)} + \| h^2 |\nb_h|^{-1} \dt r^h (t) \|_{H^s (h \Z)} \big) .
\end{split}
\label{gmab_goal}
\end{align}
\end{lemma}

\begin{proof}
For simplicity, we omit writing the argument $t$ in this proof. 

We first consider part (i). 
From \eqref{rhah}, we use the Taylor expansion to obtain
\begin{align}
r^h (\ld) - 2 \al^h (\ld) = 2 \sum_{k = 2}^\infty (-1)^{k + 1} \frac{h^{2k - 2}}{k} \al^h (\ld)^k .
\label{taylor1}
\end{align}

\noi
Thus, since $h^2 \| \al^h \|_{L^\infty (h \Z)} \leq \frac 12$, by H\"older's inequality, Lemma~\ref{LEM:embh}, and Lemma~\ref{LEM:embhs}, we have
\begin{align*}
\| r^h - 2 \al^h \|_{H^s (h \Z)}
&\leq 2 \sum_{k = 2}^\infty \frac{h^{2k - 2}}{k} \| (\al^h)^k \|_{L^2 (h \Z)} \\
&\leq 2 \sum_{k = 2}^\infty \frac{h^{2k - 2}}{k} \| \al^h \|_{L^\infty (h \Z)}^{k - 2} \| \al^h \|_{L^4 (h \Z)}^2 \\
&\leq 2 h^{\frac 32} \| \al^h \|_{L^2 (h \Z)}^2 \sum_{k = 2}^\infty \frac{1}{2^{k - 2} k} \\
&\les h^{\frac 32 + 2s} \| \al^h \|_{H^s (h \Z)}^2 ,
\end{align*}

\noi
which gives the desired estimate \eqref{gmab_goal1} since $\frac 32 + 2s > 0$ given $s > - \frac 34$. For \eqref{gmab_goal2}, we only need to notice that the Fourier multiplier operators $|\nb_h|^{-1}$ and $(\dh^+)^{-1}$ have the same modulus in $L^2$-based spaces.

We now consider part (ii). From \eqref{extend} and \eqref{gmh}, we can write the Fourier transform of $\EE \gamma_\pm^h$ as
\begin{align}
\begin{split}
\ft{\EE \gamma_\pm^h} (\xi) &= \ind_{[- \frac{\pi}{h}, \frac{\pi}{h}]} (\xi) h \sum_{\ld \in h \Z} \gamma_\pm^h (\ld) e^{- i \xi \ld} \\
&= \ind_{[- \frac{\pi}{h}, \frac{\pi}{h}]} (\xi) h \sum_{\ld \in h \Z} \al^h (\ld) e^{- 2 i \xi \ld} \mp \ind_{[- \frac{\pi}{h}, \frac{\pi}{h}]} (\xi) h \sum_{\ld \in h \Z} \be^h (\ld) e^{-2 i \xi \ld + i h \xi} .
\end{split}
\label{gmab1}
\end{align}

\noi
We also note that
\begin{align}
\frac 14 \ft{\EE r^h (\tfrac 12 \cdot)} (\xi) = \frac 12 \ind_{[- \frac{\pi}{2h}, \frac{\pi}{2h}]} (\xi) h \sum_{\ld \in h \Z} r^h (\ld) e^{- 2 i \xi \ld}
\label{Er1}
\end{align}

\noi
and
\begin{align}
\frac 14 h^2 \big( \nb_h^{-1} \EE \dt r^h (\tfrac 12 \cdot) \big)^\wedge (\xi) = \frac{h^2}{2} \ind_{[- \frac{\pi}{2h}, \frac{\pi}{2h}]} (\xi) \frac{h}{2 i \sin (h \xi)} h \sum_{\ld \in h \Z} \dt r^h (\ld) e^{- 2 i \xi \ld}
\label{Er2}
\end{align}

\noi
To deal with the $\al^h$ term on the right-hand side of \eqref{gmab1}, by using \eqref{defabh} and the Taylor expansion, we have
\begin{align}
\al^h (\ld) - \frac 12 r^h (\ld) = \sum_{k = 2}^\infty \frac{h^{2k - 2}}{2^k k!} r^h (\ld)^k .
\label{taylor2}
\end{align}

\noi
By using a change of variable, Plancherel's identity \eqref{Plan}, \eqref{taylor2}, Lemma~\ref{LEM:embh}, and Lemma~\ref{LEM:embhs}, we obtain
\begin{align*}
\Big\| &\ind_{[- \frac{\pi}{h}, \frac{\pi}{h}]} (\xi) h \sum_{\ld \in h \Z} \al^h (\ld) e^{- 2 i \xi \ld} - \frac 12 \ind_{[- \frac{\pi}{h}, \frac{\pi}{h}]} (\xi) h \sum_{\ld \in h \Z} r^h (\ld) e^{- 2 i \xi \ld} \Big\|_{L_\xi^2 (\R)} \\
&= \frac{1}{\sqrt{2}} \Big\| h \sum_{\ld \in h \Z} \al^h (\ld) e^{- i \xi \ld} - \frac 12 h \sum_{\ld \in h \Z} r^h (\ld) e^{- i \xi \ld} \Big\|_{L_\xi^2 ([- \frac{2\pi}{h}, \frac{2\pi}{h}])} \\
&\les \sum_{k = 2}^\infty \frac{h^{2k - 2}}{2^k k!} \| (r^h)^k \|_{L^2 (h \Z)} \\
&= \sum_{k = 2}^\infty \frac{h^{2k - 2}}{2^k k!} \| r^h \|_{L^{2k} (h \Z)}^k \\
&\les \sum_{k = 2}^\infty \frac{h^{(\frac 32 + s) k - \frac 32}}{2^k k!} \| r^h \|_{H^s (h \Z)}^k ,
\end{align*}

\noi
so that we get
\begin{align}
\begin{split}
\Big\| &\ind_{[- \frac{\pi}{h}, \frac{\pi}{h}]} (\xi) h \sum_{\ld \in h \Z} \al^h (\ld) e^{- 2 i \xi \ld} - \frac 12 \ind_{[- \frac{\pi}{h}, \frac{\pi}{h}]} (\xi) h \sum_{\ld \in h \Z} r^h (\ld) e^{- 2 i \xi \ld} \Big\|_{L_\xi^2 (\R)} \\
&\les h^{\frac 32 + 2s} \exp \Big( \frac 12 \| r^h \|_{H^s (h \Z)} \Big) .
\end{split}
\label{gmab2}
\end{align}

\noi
To deal with the $\be^h$ term on the right-hand side of \eqref{gmab1}, we note that from \eqref{defabh}, we can write
\begin{align}
\begin{split}
\ind_{[- \frac{\pi}{h}, \frac{\pi}{h}]} (\xi) h \sum_{\ld \in h \Z} \be^h (\ld) e^{-2 i \xi \ld + i h \xi} 
&= \frac{h^2}{2} \ind_{[- \frac{\pi}{h}, \frac{\pi}{h}]} (\xi) \frac{h e^{- i h \xi}}{2i \sin (h \xi)} h \sum_{\ld \in h \Z} \dt r^h (\ld) e^{-2 i \xi \ld + i h \xi} \\
&= \frac{h^2}{2} \ind_{[- \frac{\pi}{h}, \frac{\pi}{h}]} (\xi) \frac{h}{2i \sin (h \xi)} h \sum_{\ld \in h \Z} \dt r^h (\ld) e^{-2 i \xi \ld}
\end{split}
\label{gmab3}
\end{align} 

\noi
From \eqref{gmab1}, \eqref{Er1}, \eqref{Er2}, \eqref{gmab2}, \eqref{gmab3}, and the fact that $\frac 32 + 2s \geq \gamma$, we obtain
\begin{align*}
&\text{LHS of \eqref{gmab_goal}} \\
&\les h^{\frac 32 + 2s} \exp \Big( \frac 12 \| r^h \|_{H^s (h \Z)} \Big) + \Big\| \ind_{[- \frac{\pi}{h}, \frac{\pi}{h}] \setminus [- \frac{\pi}{2h}, \frac{\pi}{2h}]} (\xi) \jb{\xi}^{s - \gamma} h \sum_{\ld \in h \Z} r^h (\ld) e^{- 2 i \xi \ld} \Big\|_{L^2_\xi (\R)} \\
&\quad + h^2 \Big\| \ind_{[- \frac{\pi}{h}, \frac{\pi}{h}] \setminus [- \frac{\pi}{2h}, \frac{\pi}{2h}]} (\xi) \jb{\xi}^{s - \gamma} \frac{h}{2 i \sin (h \xi)} h \sum_{\ld \in h \Z} \dt r^h (\ld) e^{- 2 i \xi \ld} \Big\|_{L^2_\xi (\R)} \\
&\les h^{\gamma} \exp \Big( \frac 12 \| r^h \|_{H^s (h \Z)} \Big) + h^\gamma \big( \| r^h \|_{H^s (h \Z)} + \| h^2 |\nb_h|^{-1} \dt r^h \|_{H^s (h \Z)} \big) ,
\end{align*}

\noi
as desired, where we used the fact that $\jb{\frac{2}{h}\sin (\frac{h \xi}{2})} \leq \jb{\xi}$ for any $\xi \in \R$ in the last inequality.
\end{proof}

We remark that the condition $- \frac 34 < s \leq 0$ in Lemma~\ref{LEM:gmdiff} is likely not optimal, but it already suffices for our purpose.

We are now ready to prove Theorem~\ref{THM:toda}~(ii).
\begin{proof}[Proof of Theorem~\ref{THM:toda}~\textup{(ii)}]
We first note from Lemma~\ref{LEM:embh} and the condition \eqref{hcond2} that
\begin{align*}
h^2 \| \al_0^h \|_{L^\infty (h \Z)} \leq h^{\frac 32} \| \al_0^h \|_{L^2 (h \Z)} \leq \frac 14 .
\end{align*} 

\noi
Thus, we can define
\begin{align*}
r_0^h := \frac{2}{h^2} \ln (1 + h^2 \al_0^h) \quad \text{and} \quad r_1^h := \frac{2}{h^2} \dh^+ \be_0^h 
\end{align*}

\noi
and let $r_{\pm, 0}^h$ be defined in \eqref{rhinit} in terms of $r_0^h$ and $r_1^h$. 
From Lemma~\ref{LEM:gmdiff}~(i) and the condition \eqref{ab0cond2-1}, we have
\begin{align}
\| r^h_0 \|_{L^2 (h \Z)} + \| h^2 |\nb_h|^{-1} r^h_1 \|_{L^2 (h \Z)} + \| u_{+, 0} \|_{L^2 (\R)} + \| u_{-, 0} \|_{L^2 (\R)} \les C_1 + C_1^2 .
\label{rinit1}
\end{align}

\noi
Also, from Lemma~\ref{LEM:gmdiff}~(ii) and (i) and the condition \eqref{ab0cond2-2}, we have
\begin{align}
\| r_{\pm, 0}^h - u_{\pm, 0} \|_{H^{-1} (\R)} \les (e^{C C_1} + C_2 ) h 
\label{rinit2}
\end{align}

\noi
for some constant $C > 0$.
Thus, we can apply Theorem~\ref{THM:FPU}~(ii). Note that we do not need to check the condition \eqref{hcond} since it is only used to guarantee the existence of a unique solution to the scaled FPU system, but in the case of the Toda lattice in Flaschka's form, we already have it from Theorem~\ref{THM:toda}~(i). 
Moreover, the use of the solution bound in Corollary~\ref{COR:GWP} can be replaced by
\begin{align}
\begin{split}
\| &r^h \|_{C_T L^2 (h \Z)} + \| h^2 |\nb_h|^{-1} \dt r^h \|_{C_T L^2 (h \Z)} \\
&\les 1 + \| \al^h \|_{C_T L^2 (h \Z)}^2 + \| \be^h \|_{C_T L^2 (h \Z)}^2 \\
&\les 1 + C_1^2 
\end{split}
\label{rrconv2}
\end{align}

\noi
using Lemma~\ref{LEM:gmdiff}~(i) and the solution bound \eqref{abbdd}.

From Theorem~\ref{THM:FPU}~(ii), we know that for any $T > 0$, there exists a unique solution $u_\pm$ to the KdV equation \eqref{KdV} in a subspace of $C ([-T, T]; L^2 (\R))$ with initial data $u_\pm |_{t = 0} = u_{\pm, 0}$, and we know that if $(r^h, \dt r^h)$ is the solution to the scaled FPU system \eqref{rhwave0} with $V(r) = e^r - r - 1$ and initial data $(r^h, \dt r^h)|_{t = 0} = (r_0^h, r_1^h)$, then 
\begin{align}
\Big\| \frac 12 \big( \EE r^h \mp h^2 \nb_h^{-1} \EE \dt r^h \big) - e^{\mp \frac{t}{h^2} \dx} u_\pm \Big\|_{C_T H^{-1} (\R)} \leq \min \big( C' (1 + C_1^2), e^{C'' T} \wt C h^{\frac 25} \big)
\label{rrconv1}
\end{align}

\noi
for some constants $C' > 0$, $C'' = C'' (C_1) > 0$, and $\wt C = \wt C (C_1, C_2) > 0$.
Thus, from Lemma~\ref{LEM:gmdiff}~(ii), \eqref{rrconv1}, and \eqref{rrconv2}, we obtain
\begin{align*}
\Big\| \EE \gamma_\pm^h - \frac 12 (e^{\mp \frac{t}{h^2} \dx} u_\pm) \Big(\frac 12 \cdot \Big) \Big\|_{C_T H^{-1} (\R)} 
&\les \Big\| \EE \gamma_\pm^h - \frac 14 \Big( \EE r^h \Big( \frac 12 \cdot \Big) \mp h^2 \nb_h^{-1} \EE \dt r^h \Big( \frac 12 \cdot \Big) \Big) \Big\|_{C_T H^{-1} (\R)} \\
&\quad + \Big\| \frac 12 \big( \EE r^h \mp h^2 \nb_h^{-1} \EE \dt r^h \big) - e^{\mp \frac{t}{h^2} \dx} u_\pm \Big\|_{C_T H^{-1} (\R)} \\
&\les e^{C (1 + C_1^2)} h + \min \big( C' (1 + C_1^2), e^{C'' T} \wt C h^{\frac 25} \big) ,
\end{align*}

\noi
where the scaling factor $\frac 12 \cdot$ means the factor on the spatial variable.
This implies the desired estimate \eqref{abconv2}.
\end{proof}

\begin{ackno} \rm
The authors would like to thank Robert Schippa for a helpful discussion on the proof of the bilinear estimate Lemma~\ref{LEM:bilin}. We acknowledge the use of AI tools for improvement of the presentation and for correction of typographical errors. All mathematical arguments were verified and written entirely by the authors. H.K. and R.L. were supported by the DFG through the Hausdorff Center for Mathematics under Germany's Excellence Strategy - GZ 2047/1, Project-ID 390685813, SFB 1060, Project-ID 211504053  and SFB 1720, Project-ID 539309657.
\end{ackno}

\end{document}